\documentclass [twoside, reqno, 11pt] {amsart}
\usepackage{graphicx}

\usepackage{amsfonts}
\usepackage{amssymb}
\usepackage{a4}
\usepackage{color}

\newtheorem{thm}{Theorem}[section]
\newtheorem{cor}[thm]{Corollary}
\newtheorem{lem}[thm]{Lemma}
\newtheorem{prop}[thm]{Proposition}

\newtheorem{rem}[thm]{Remark}

\theoremstyle{definition}

\numberwithin{equation}{section}

\renewcommand{\Re}{\hbox{Re}\,}
\renewcommand{\Im}{\hbox{Im}\,}

\newcommand{\C}{\mathbb{C}}

\renewcommand{\div}{\operatorname{div}}

\newcommand{\N}{\mathbb{N}}

\newcommand{\R}{\mathbb{R}}

\newcommand{\supp}{\operatorname{supp}}

\def\tilde{\widetilde}
\def \bfo {\begin {eqnarray*} }
\def \efo {\end {eqnarray*} }
\def \ba {\begin {eqnarray*} }
\def \ea {\end {eqnarray*} }
\def \beq {\begin {eqnarray}}
\def \eeq {\end {eqnarray}}
\def \supp {\hbox{supp }}

\def \dist {\hbox{dist}}

\def \p {\partial}

\def\tilde{\widetilde}
\def \bfo {\begin {eqnarray*} }
\def \efo {\end {eqnarray*} }
\def \ba {\begin {eqnarray*} }
\def \ea {\end {eqnarray*} }
\def \beq {\begin {eqnarray}}
\def \eeq {\end {eqnarray}}
\def \supp {\hbox{supp }}

\def \dist {\hbox{dist}}

\def \p {\partial}

\begin{document}

 \title[Inverse transmission problems for magnetic Schr\"odinger operators ]{Inverse transmission problems for magnetic Schr\"odinger operators}

\author[Krupchyk]{Katsiaryna Krupchyk}

\address
        {K. Krupchyk, Department of Mathematics and Statistics \\
         University of Helsinki\\
         P.O. Box 68 \\
         FI-00014   Helsinki\\
         Finland}

\email{katya.krupchyk@helsinki.fi}

\maketitle

\begin{abstract} This paper is concerned with the study of inverse transmission problems for magnetic Schr\"odinger operators on bounded domains and in all of the Euclidean space, in the self-adjoint case. Assuming that the magnetic and electric potentials are known outside of a transparent obstacle, in the bounded domain case, we show that the obstacle, the transmission coefficients, as well as the magnetic field and electric potential inside the obstacle are uniquely determined from the knowledge of the set of the Cauchy data for the transmission problem, given on an open subset of the boundary of the domain. In the case of the transmission scattering problem, we obtain the same conclusion, when the scattering amplitude at a fixed frequency is known. The problems studied in this work were proposed in \cite{Isakov_2008}.

\end{abstract}

\section{Introduction and statement of results}

Let $\Omega\subset\R^n$, $n\ge 3$, be a bounded domain with connected Lipschitz boundary, and let $D\subset\subset\Omega$ be a bounded open set with Lipschitz boundary such that $D^+:=\Omega\setminus \overline{D}$ is connected.  
Setting also $D^-:=D$, and letting $A^\pm\in W^{1,\infty}(D^\pm,\R^n)$, $q^\pm\in L^\infty(D^\pm,\R)$, we consider the magnetic Schr\"odinger operators
\begin{align*}
\mathcal{L}_{A^\pm,q^\pm}(x,D_x):=&\sum_{j=1}^n(D_{x_j}+A_j^\pm(x))^2+q^\pm(x)\\
=&-\Delta-2iA^\pm(x)\cdot\nabla-i(\nabla\cdot A^\pm(x))+(A^\pm(x))^2+q^\pm(x),
\end{align*}
where $D_{x_j}=-i\p_{x_j}$. 
Let $(u^+,u^-)\in H^1(D^+)\times H^1(D^-)$ satisfy the magnetic Schr\"odinger equations
\begin{equation}
\label{eq_main}
\begin{aligned}
\mathcal{L}_{A^+,q^+}(x,D_x)u^+&=0\quad \textrm{in}\quad D^+,\\
\mathcal{L}_{A^-,q^-}(x,D_x)u^-&=0\quad \textrm{in}\quad D^-.
\end{aligned}
\end{equation}
Denote by  $\nu$ the almost everywhere defined outer unit normal to $\p D$ and to $\p \Omega$. Since $\Delta u^+\in L^2(D^+)$ and $\Delta u^-\in L^2(D^-)$, 
the traces of the normal derivatives $\p_\nu u^+$ and $\p_\nu u^-$ on $\p D$ are well-defined as  elements of $H^{-1/2}(\p D)$, see \cite[Chapter 3]{McLean_book} as well as Subsection \ref{sub_sec_sobol} below.
 In addition to \eqref{eq_main}, we require that $(u^+,u^-)$ satisfies the following transmission conditions on $\p D$,
\begin{equation}
\label{eq_trans_cond}
\begin{aligned}
u^+&=a u^-\quad \textrm{on}\quad \p D,\\
(\p_\nu+iA^+\cdot \nu)u^+&=b(\p_\nu+iA^-\cdot\nu)u^-+cu^-\quad \textrm{on}\quad \p D,
\end{aligned}
\end{equation}
as well as the Dirichlet boundary conditions on the boundary of  $\Omega$,
\begin{equation}
\label{eq_int_bounda_d}
u^+=g\quad \textrm{on}\quad \p \Omega.
\end{equation}
Here $a,b\in C^{1,1}(\overline{D},\R)$ (the space of $C^1$--functions with Lipschitz gradient in a neighborhood of $\overline{D}$),  $c\in C(\overline{D},\R)$, and $g\in H^{1/2}(\p \Omega)$. 
The transmission conditions \eqref{eq_trans_cond} encompass physical models of imperfect transmission arising in acoustics, elastodynamics, quantum scattering, and semiconductor physics.

In what follows we shall assume that $a,b>0$ in $\overline{D}$.  
Then the Fredholm alternative holds for the transmission problem \eqref{eq_main}, \eqref{eq_trans_cond}, and \eqref{eq_int_bounda_d}, see Proposition \ref{prop_2_4} below. 
Let $\gamma\subset \p\Omega$ be an open nonempty subset of the boundary of $\Omega$ and let  
\begin{align*}
\mathcal{C}_\gamma(A^+,q^+,&A^-,q^-,a,b,c; D):=\{(u^+|_{\p \Omega},(\p_\nu+iA^+\cdot \nu) u^+|_{\gamma}):\\ 
&(u^+,u^-)\in H^1(D^+)\times H^1(D^-)
\textrm{ solves }\eqref{eq_main}, \eqref{eq_trans_cond},\supp(u^+|_{\p \Omega})\subset\gamma\}
\end{align*}
be the set of the Cauchy data for the transmission problem \eqref{eq_main}, \eqref{eq_trans_cond}, associated with $\gamma$. 

The first inverse problem studied in this paper is as follows. Assume that  we are given the bounded domain $\Omega\subset \R^n$, the subset $\gamma\subset\p \Omega$, the magnetic and electric potentials $A^+$, $q^+$, and 
the set $\mathcal{C}_\gamma(A^+,q^+,A^-,q^-,a,b,c; D)$ of the Cauchy data for the transmission problem \eqref{eq_main}, \eqref{eq_trans_cond}, associated with $\gamma$.  The problem is whether we can recover the obstacle $D$, 
the transmission coefficients $a$, $b$, and $c$ on $\p D$, as well as the magnetic and electric potentials  $A^-$, $q^-$ in $D$.  

This problem was proposed in  \cite{Isakov_2008}, where the corresponding inverse transmission problem in the absence of magnetic potentials was investigated.   As it was pointed out in \cite{Isakov_2008}, in general one cannot hope to recover the transmission  coefficients $a$, $b$, and $c$ on $\p D$ uniquely, since the set of  the Cauchy data enjoys the following invariance property,   
\begin{equation}
\label{eq_int_1_0}
\mathcal{C}_\gamma(A^+,q^+, A^-,q^-,a,b,c; D)
=\mathcal{C}_\gamma(A^+,q^+, A^- ,q^-,\alpha a,\alpha b,\alpha c; D),
\end{equation}
for any  $\alpha>0$ on $\overline{D}$,  constant on each connected component of $D$.

In the presence of the magnetic potentials there is another gauge transformation that preserves the set of the Cauchy data. Namely,  for any function $\psi\in C^{1,1}(\overline{D}, \R)$, we have
\[
e^{-i\psi}\mathcal{L}_{A^-,q^-}e^{i\psi}=\mathcal{L}_{A^-+\nabla \psi,q^-}.
\]
Hence,  $(u^+,u^-)\in H^1(D^+)\times H^1(D^-)$ satisfies the transmission problem \eqref{eq_main}, \eqref{eq_trans_cond}  if and only if $(u^+,U^-)$, where $U^-=e^{-i\psi}u^-$,  satisfies 
\begin{align*}
&\mathcal{L}_{A^+,q^+}u^+=0\quad\textrm{in}\quad \Omega\setminus\overline{D}, \\
&\mathcal{L}_{A^-+\nabla \psi,q^-}U^-=0\quad\textrm{in}\quad D,\\
&u^+= a e^{i\psi} U^-\quad \textrm{on}\quad \p D,\\
&(\p_\nu+iA^+\cdot \nu)u^+= e^{i\psi}( b(\p_\nu+i(A^-+\nabla \psi)\cdot\nu)U^-+ cU^-)\quad \textrm{on}\quad \p D.
\end{align*}
Thus, for any function $\psi\in C^{1,1}(\overline{D}, \R)$ such that $\psi|_{\p D}=0$, we have
\begin{equation}
\label{eq_int_1}
\mathcal{C}_\gamma(A^+,q^+, A^-,q^-,a,b,c; D)
=\mathcal{C}_\gamma(A^+,q^+,  A^-+\nabla \psi,q^-, a, b, c; D). 
\end{equation}
Notice that the invariance of the set of the Cauchy data under the gauge transformation $A^-\mapsto A^-+\nabla \psi$ with $\psi|_{\p D}=0$  is the standard obstruction to the unique determination of the magnetic potential in inverse boundary value problems, see \cite{NakSunUlm_1995, Sun_1993}.

In general, the transmission problem \eqref{eq_main}, \eqref{eq_trans_cond}, and \eqref{eq_int_bounda_d} is non-self-adjoint.  As we shall see in Section  \ref{subsec_direct} below, the  self-adjointness of  the transmission problem   is guaranteed  by the assumptions that $A^\pm$, $q^\pm$ are real-valued,  and $ab=1$ on $\overline{D}$.  

In this paper we shall be concerned with inverse transmission problems in the self-adjoint case. The importance of 
 the self-adjoint transmission conditions comes in particular from the fact that they assure the continuity of the energy flux of the solution $(u^+,u^-)$ of the transmission problem along the boundary of the obstacle $D$, i.e. 
\[
\Im (\overline{u^+}(\p_\nu + iA^+\cdot\nu) u^+)=\Im (\overline{u^-}(\p_\nu + iA^-\cdot\nu) u^-)\quad \textrm{on}\quad \p D. 
\]   

Working with self-adjoint transmission problems,  the obstruction \eqref{eq_int_1_0}
can be eliminated, as shown in the following, first main result of this paper. 
\begin{thm} 
\label{thm_main_sa} 
Let $\Omega\subset\R^n$, $n\ge 3$, be a bounded domain with connected Lipschitz boundary, and $D_1,D_2\subset\subset\Omega$ be  bounded open subsets with Lipschitz boundaries such that $\Omega\setminus \overline{D_j}$ is connected, $j=1,2$.  Let $A^+\in W^{1,\infty}(\Omega,\R^n)$, $q^+\in L^\infty(\Omega,\R)$,  $A_j^-\in W^{1,\infty}(D_j,\R^n)$, $q_j^-\in L^\infty(D_j,\R)$,  $a_j, b_j\in C^{1,1}(\overline{D_j},\R)$, and $c_j\in C(\overline{D_j},\R)$, $j=1,2$. Assume that $a_j,b_j>0$ on $\overline{D_j}$, $a_jb_j=1$ on   $\overline{D_j}$, and 
\[
a_j(x)\ne  1 \quad \textrm{for all } x\in \p D_j,\quad j=1,2.
\]
If 
\[
\mathcal{C}_\gamma(A^+,q^+, A_1^-,q_1^-,a_1,b_1,c_1; D_1)
=\mathcal{C}_\gamma(A^+,q^+,  A_2^-,q_2^-,a_2,b_2,c_2; D_2),
\]
for an open non-empty subset $\gamma\subset \p \Omega$, then 
\[
D_1=D_2=:D, 
\]
and 
\[
 a_1=a_2,\quad b_1=b_2,\quad c_1=c_2, \quad \textrm{on}\quad \p D.
\] 
Furthermore, if $\p D$ is of class $C^{1,1}$, there is a  function $\psi\in C^{1,1}(\overline{D},\R)$,  $\psi|_{\p D}=0$, such that 
\begin{equation}
\label{eq_int_2_2}
A_2^-=A_1^-+\nabla\psi,\quad q_1^-=q_2^-,\quad \textrm{in}\quad D.
\end{equation}

\end{thm}

In order to recover the obstacle and the transmission coefficients in the proof of Theorem  \ref{thm_main_sa},  we shall follow closely \cite{Isakov_2008}, and use the method of singular solutions for the transmission problem, with singularities approaching the boundary of the obstacle.  The presence of the magnetic potentials complicates the arguments, and we have therefore attempted to give a careful discussion througout. 
When constructing the singular solutions, it becomes essential to assure that the unique solvability of the transmission problem can always be achieved by a small perturbation of the boundary of a domain.  Furthermore, this property 
is required when establishing some auxiliary Runge type results on approximation of solutions of transmission problems in subdomains by solutions in larger domains, in Subsection \ref{sec_runge}.  We show that this key property is enjoyed by the self-adjoint transmission problem in Section \ref{sec_domain_pertur}, through an application of the mini-max principle.

In the second part of Theorem \ref{thm_main_sa}, we assume that the boundary of $D$ is of class $C^{1,1}$.  Indeed, to the best of our knowledge, the most general boundary reconstruction result for the tangential component of a continuous magnetic potential from the knowledge of the Dirichlet--to--Neumann map has been obtained in \cite{Brown_Salo_2006}, when the boundary is of class $C^1$. Next, as far as we know, the most general result, in the sense of regularity,  for inverse boundary value problems for the magnetic Schr\"odinger operator has been proven in \cite{Salo_2004} making use of \cite{Brown_Salo_2006},  under the assumption that the boundary of the domain $D$ is of class $C^{1,\textrm{Dini}}$, the magnetic potential is of class $C^{\textrm{Dini}}(\overline{D})$, and the electric potential is of class $L^\infty(D)$. Here $C^{\textrm{Dini}}(\overline{D})$ stands for the space of Dini continuous functions, see \cite{Salo_2004}. We have therefore decided to avoid considering the issue of getting the minimal regularity assumptions on magnetic potentials and on the boundary of $D$, and will content ourselves with  
Lipschitz continuous magnetic potentials. The minimal regularity of the boundary of $D$ required when working with Lipschitz continuous magnetic potentials seems to be  $C^{1,1}$. In particular, this is due to the fact that, in general, a 
solution $\psi$  to the equation $(A_j^-+ \nabla \psi)\cdot\nu =0$ on  $\p D$ will be of class $C^{1,1}(\overline{D})$ only when $\p D$ is of class $C^{1,1}$, see \cite[Lemma  5.8]{Salo_2004}.

We would like to emphasize that the set of the Cauchy data in Theorem \ref{thm_main_sa} can be given on  an arbitrarily small open non-empty subset of the boundary of $\Omega$.  This is important from the point of view of applications, since in practice, performing measurements on the entire boundary could be either impossible or too cost consuming.  To the best of our knowledge, the only available result in the presence of an obstacle, where the measurements are performed on an arbitrarily small portion of the boundary, is the work \cite{Isakov_2008} for the Schr\"odinger operator without a magnetic potential.  When no obstacle is present and the electric and magnetic potentials are known near the boundary, it is proven in \cite{Ammari_Uhlmann_2004},  see also \cite{KrupLassasUhlmann_slab}, that the knowledge of the Cauchy data on an arbitrarily small part of the boundary determines uniquely the magnetic field and the electric potential in  the entire domain.  Dropping the assumption that the potentials are known near the boundary, some fundamental recent progress on inverse boundary value problems with partial  
measurements has been achieved in \cite{BukhUhl_2002, DKSU_2007, Ken_Sjo_Uhl_2007}.

The second part of the paper is devoted to the inverse scattering problem for the magnetic Schr\"odinger operator in the presence of a transparent obstacle.  Here we assume that $D\subset\R^n$, $n\ge 3$, is a bounded open set with Lipschitz boundary such that  $\R^n\setminus\overline{D}$ is connected. 
Let $A^+\in W^{1,\infty}(\R^n, \R^n)$,  $A^-\in W^{1,\infty}(D,\R^n)$, $q^+\in L^\infty(\R^n,\R)$, $q^-\in L^\infty(D,\R)$.  Assume that $A^+$ and $q^+$ are compactly supported. As before, let  $a,b\in C^{1,1}(\overline{D},\R)$, $c\in C(\overline{D},\R)$, be such that $a,b>0$ in $\overline{D}$ and $ab=1$ on  $\overline{D}$.   

Let $k>0$,  $\xi\in \mathbb{S}^{n-1}:=\{\xi\in\R^n:|\xi|=1\}$, and consider the scattering transmission problem, 
\begin{equation}
\label{eq_int_8_10}
\begin{aligned}
&(\mathcal{L}_{A^+,q^+}-k^2)u^+=0\quad \textrm{in}\quad \R^n\setminus\overline{D},\\
&(\mathcal{L}_{A^-,q^-}-k^2)u^-=0\quad\textrm{in}\quad D,\\
&u^+=au^-\quad\textrm{on}\quad \p D,\\
&(\p_\nu+iA^+\cdot\nu)u^+=b(\p_\nu+iA^-\cdot\nu)u^-+cu^-\quad\textrm{on}\quad \p D,\\
& u^+(x;\xi,k)=e^{ikx\cdot\xi}+u_0^+(x;\xi,k),\\
&(\p_r-ik)u_0^+=o(r^{-(n-1)/2}), \quad\textrm{as} \quad r=|x|\to\infty. 
\end{aligned}
\end{equation}
As it is shown in Corollary \ref{cor_scatt_8_3} below,  under the assumptions above, the problem \eqref{eq_int_8_10} has a unique solution $(u^+,u^-)\in H^1_{\textrm{loc}}(\R^n\setminus\overline{D})\times H^1(D)$. 

It is known that the scattered wave $u_0^+$ has the following asymptotic behavior, 
\[
u_0^+(x;\xi,k)=a(\theta,\xi,k)\frac{e^{ik|x|}}{|x|^{(n-1)/2}}+\mathcal{O}\bigg(\frac{1}{|x|^{(n+1)/2}}\bigg),\quad \theta=\frac{x}{|x|},\quad \textrm{as}\quad |x|\to \infty,
\]
 see \cite{Col_Kress_book, Odell_2006}.  The function $a(\theta,\xi,k):=a(A^+,q^+,A^-,q^-,a,b,c,D; \theta,\xi,k)$ is called the scattering amplitude. 

The second inverse problem studied in this paper is as follows.  Assume that we are given the scattering amplitude $a(\theta,\xi,k)$ for all $\theta,\xi\in \mathbb{S}^{n-1}$ and for some fixed $k>0$, as well as the magnetic and electric potentials $A^+$ and $q^+$. The problem is whether this information determines the obstacle $D$, the transmission coefficients $a$, $b$, and $c$ on $\p D$, as well as the magnetic and electric potentials $A^-$, $q^-$ in $D$.
In this direction we have the following result, which is a generalization of  \cite[Theorem 1.2]{Isakov_2008} and \cite{Valdivia_2004}. 

\begin{thm} 
\label{thm_main_2}

Let $D_1,D_2\subset\R^n$, $n\ge 3$, be bounded open sets with Lipschitz boundaries such that $\R^n\setminus \overline{D_j}$ is connected, $j=1,2$.  Let 
$A^+\in W^{1,\infty}(\R^n, \R^n)$, $q^+\in L^\infty(\R^n,\R)$,  $A_j^-\in W^{1,\infty}(D_j,\R^n)$, $q_j^-\in L^\infty(D_j,\R)$.  Assume that $A^+$ and $q^+$ are compactly supported. Let $a_j, b_j\in C^{1,1}(\overline{D_j},\R)$, and $c_j\in C(\overline{D_j},\R)$, $j=1,2$.  Assume that $a_j,b_j>0$ on $\overline{D_j}$, $a_jb_j=1$ on $\overline{D_j}$, and 
\[
a_j(x)\ne 1 \quad \textrm{for all } x\in \p D_j,\quad j=1,2.
\]
 If 
\[
a(A^+,q^+,A^-_1,q^-_1,a_1,b_1,c_1,D_1; \theta,\xi,k)=a(A^+,q^+,A^-_2,q^-_2,a_2,b_2,c_2,D_2; \theta,\xi,k),
\]
 for all $\theta,\xi\in \mathbb{S}^{n-1}$ and for some fixed $k>0$, then 
\[
D_1=D_2=:D, 
\]
and 
\[
 a_1= a_2,\quad b_1= b_2,\quad c_1=c_2, \quad \textrm{on}\quad \p D.
\] 
Furthermore, if $\p D$ is of class $C^{1,1}$, there is  a function $\psi\in C^{1,1}(\overline{D},\R)$,  $\psi|_{\p D}=0$, such that 
\[
A_2^- =A_1^-+\nabla\psi,\quad\textrm{in}\quad D,\quad q_1^-=q_2^-,\quad\textrm{in}\quad D.
\]
\end{thm}

The study of inverse obstacle scattering at a fixed frequency has a long tradition, starting with the uniqueness proof for the Dirichlet boundary conditions, going back to Schiffer \cite{Lax_Phillips_book}, see also \cite{Kirsch_Kress_1993}.  Another important technique for  the identification of obstacles, applicable to the Neumann and transmission problems, is the method of singular solutions, developed in  \cite{Isakov_1988, Isakov_1990}.  Among further important contributions to the circle of questions around inverse transmission obstacle scattering, we should mention \cite{Kirsch_Kress_1993, Kirsch_Paivar_1998, Odell_2008, Valdivia_2004}. See also the review paper  \cite{Isakov_2009} and the references given there.

The plan of the paper is as follows. After preliminaries in Section \ref{sec_preliminaries}, the solvability of the transmission problem \eqref{eq_main}, \eqref{eq_trans_cond}, and \eqref{eq_int_bounda_d} in the general non-self-adjoint case 
is discussed in Section \ref{subsec_direct} by means of a variational approach, convenient here due to the low regularity of the boundary of the obstacle.
In section \ref{sec_domain_pertur} we show that in the self-adjoint case,  the unique solvability of the transmission problem can be achieved by a small perturbation of the boundary of the domain.   
Section \ref{sec_inv_direct_1} is devoted to the solution of the inverse transmission  problem on a bounded domain and  to the proof of Theorem \ref{thm_main_sa}.  
In the approach to the reconstruction of the obstacle and of the transmission coefficients, following \cite{Isakov_2008}, we use the method of singular solutions for the transmission problem. The task of the reconstruction of the obstacle and of the transmission coefficients occupies Subsections \ref{sec_fund_sol} -- \ref{sec_transmiss_coef}.  

Once the obstacle and the transmission coefficients have been recovered, the determination of the magnetic and electric potentials inside the obstacle becomes possible. This is the subject of Subsection \ref{sec_rec_potentials}.
Proceeding in the spirit of inverse boundary value problems for the magnetic Schr\"odinger operator, in order to exploit a fundamental integral identity, valid inside the obstacle, it becomes essential to determine the values of the tangential component of the magnetic potential  on the boundary of the obstacle. To this end, in Section \ref{sec_boundary_rec}, we adapt the method of \cite{Brown_Salo_2006} of the boundary reconstruction of the tangential component of the magnetic potential to our situation, by combining it with an idea of \cite{SaloTzou_2009}. With the tangential components of the magnetic potential determined, the exploitation of the integral identity becomes possible, and using the machinery developed in \cite{DKSU_2007, Knu_Salo_2007, NakSunUlm_1995,  Salo_2004, Sun_1993} for inverse boundary value problems for the magnetic  Schr\"odinger operator,  we are able to determine the magnetic and electric potentials inside the obstacle, up to a natural gauge transformation.  

Section \ref{sec_scat_prob} is concerned with the scattering transmission problem. First, in Subsection \ref{subsec_direct_scat}, using the Lax-Phillips method, we investigate the existence of solutions to the scattering transmission problem in the non-self-adjoint case. This discussion generalizes \cite{Valdivia_2004}, where the case without magnetic potentials is treated. 
In Subsection \ref{subsec_inv_scat}, the inverse scattering transmission problem is studied, and following the arguments of  \cite{Isakov_2008}, we show that Theorem \ref{thm_main_2} is implied by Theorem \ref{thm_main_sa}. 

In Appendix \ref{ap_UC}, we present a unique continuation result for elliptic second order operators from Lipschitz boundaries, which is used several times in the main text. Although this result is essentially well-known, since it plays an important role in the paper,  we give it here for the convenience of the reader. Appendix \ref{appendix_fundamental_sol} contains some estimates for fundamental solutions of the magnetic Schr\"odinger operator, which are 
crucial when estimating singular solutions of the transmission problems. Finally, in Appendix \ref{appendix_integrals} we provide a brief discussion of asymptotic bounds on some volume and surface integrals, required in the reconstruction of the obstacle and of the transmission coefficients in the main text.

\section{Preliminaries} 

\label{sec_preliminaries}

\subsection{Sobolev spaces and traces}

\label{sub_sec_sobol}
Let $\Omega\subset \R^n$, $n\ge 3$,  be a bounded open set with Lipschitz boundary.  Let $a\in L^\infty(\Omega,\R)$, $a\ge a_0>0$ a.e. in $\Omega$,  and consider the space
\[
H_{a}(\Omega)=\{u\in H^1(\Omega):\div(a\nabla u)\in L^2(\Omega)\},
\]
which is a Hilbert space, equipped with the norm
\[
\|u\|_{H_a(\Omega)}^2=\|u\|_{H^1(\Omega)}^2+\|\div(a\nabla u)\|_{L^2(\Omega)}^2.
\] 
The map $u\mapsto (a\p_\nu u)|_{\p \Omega}$ is continuous on $H_a(\Omega)$ with values in $H^{-1/2}(\p \Omega)$, see \cite[Theorem 1.21]{Choulli_book}. We have for $u\in H_a(\Omega)$, 
\begin{equation}
\label{eq_trace_cont}
\|a\p_\nu u\|_{H^{-1/2}(\p \Omega)}\le C\|u\|_{H_a(\Omega)}, \quad C>0.
\end{equation}

In this paper, we shall work with a subspace of $H^1(\Omega)$, which contains $H_a(\Omega)$, for which the trace of the normal derivative of $u$ is still well-defined. To introduce this space, we shall need to consider the $L^2$--dual of $H^1(\Omega)$, given by
\[
\tilde H^{-1}(\Omega)=\{f\in H^{-1}(\R^n):\supp(f)\subset\overline{\Omega}\},
\] 
and $\|\cdot\|_{\tilde H^{-1}(\Omega)}=\|\cdot\|_{H^{-1}(\R^n)}$, 
see \cite{McLean_book}. Here we use the natural inner product on $L^2(\Omega)$, 
\[
(u,v)_{L^2(\Omega)}=\int_{\Omega} u\overline{v}dx,\quad u,v\in L^2(\Omega).
\]

\begin{rem}
Notice that in general, 
\[
\{f\in H^{-1}(\R^n):\supp(f)\subset \p \Omega\}\ne\emptyset.
\]
Thus, 
the restriction $f|_{\Omega}\in \mathcal{D}'(\Omega)$ does not  determine $f\in \tilde H^{-1}(\Omega)$ uniquely, and therefore, $\tilde H^{-1}(\Omega)$ cannot be imbedded into the space $\mathcal{D}'(\Omega)$.   
\end{rem}

The following result will allow us to define the trace of the normal derivatives of functions from a suitable subspace of $H^1(\Omega)$.   
\begin{prop}\cite[Lemma 4.3]{McLean_book} Let $\Omega\subset \R^n$ be a bounded open set with Lipschitz boundary. 
Let $u\in H^1(\Omega)$ and $f\in \tilde H^{-1}(\Omega)$ satisfy 
\[
-\div(a\nabla u)=f\quad \text{in} \quad \Omega,
\]
where $a\in L^{\infty}(\Omega,\R)$, $a\ge a_0>0$ a.e. in $\Omega$. Then there exists $g\in H^{-1/2}(\p \Omega)$ such that 
\begin{equation}
\label{eq_trace_lem}
\int_{\Omega}a\nabla u\cdot \nabla \overline{v} dx
=(f,v)_{\tilde H^{-1}(\Omega),H^1(\Omega)}+(g,v|_{\p \Omega})_{H^{-1/2}(\p \Omega),H^{1/2}(\p \Omega)}, 
\end{equation}
for any $v\in H^1(\Omega)$. Furthermore, $g$ is uniquely determined by $u$ and $f$, and we have
\[
\|g\|_{H^{-1/2}(\p \Omega)}\le C(\|u\|_{H^1(\Omega)}+\|f\|_{\tilde H^{-1}(\Omega)}). 
\]
\end{prop}

In what follows we shall write $g=(a\p_{\nu} u)|_{\p \Omega}$, when the element $f\in \tilde H^{-1}(\Omega)$ is given.  In particular, when $\div(a\nabla u)\in L^2(\Omega)$, we shall always make the natural choice 
\[
f= (-\div(a\nabla u))\chi_{\Omega}\in L^2(\R^n), 
\]
which allows us to recover the standard definition of the trace $(a\p_\nu u)|_{\p \Omega}$  of a function $u\in H_a(\Omega)$.  Here $\chi_\Omega$ is the characteristic function of $\Omega$.

In the sequel, we shall also need the following result. 
\begin{prop}\cite[Theorem 3.20]{McLean_book}
\label{prop_sob_mult}
Let $\Omega$ be an open subset  of $\R^n$. Let $v\in C_0^r(\R^n)$ for some integer $r\ge 1$, and let $s\in \R$, $|s|\le r$. If $u\in H^s(\Omega)$, then $vu\in H^s(\Omega)$ and 
\[
\|vu\|_{H^s(\Omega)}\le C_r\|v\|_{W^{r,\infty}(\R^n)}\|u\|_{H^s(\Omega)}. 
\]
The same conclusion holds with $H^s(\Omega)$ replaced by $\tilde H^s(\Omega)$.
\end{prop}

\subsection{Magnetic Schr\"odinger operators}
Let $\Omega\subset \R^n$ be a bounded open set  with Lipschitz boundary, and let $A\in W^{1,\infty}(\Omega,\C^n)$ and $q\in L^\infty(\Omega,\C)$.   
Let $f\in \tilde H^{-1}(\Omega)$ and let $u\in H^1(\Omega)$  satisfy the magnetic Schr\"odinger equation, 
\begin{equation}
\label{eq_def_green_1}
\mathcal{L}_{A,q}u=f\quad \text{on} \quad \Omega.
\end{equation}
Then we have the following first Green formula,
\begin{equation}
\label{eq_first_green}
\begin{aligned}
(f,v)_{\tilde H^{-1}(\Omega),H^1(\Omega)}& +((\p_\nu+iA\cdot\nu)u, v)_{H^{-1/2}(\p \Omega),H^{1/2}(\p \Omega)}=\int_{\Omega}\nabla u\cdot \nabla \overline{v}dx\\
&+i\int_{\Omega}A\cdot (u \nabla \overline{v}-  \overline{v}\nabla u)dx +\int_{\Omega}(A^2+q)u\overline{v}dx,
\end{aligned}
\end{equation}
for any $v\in H^1(\Omega)$, see  \eqref{eq_trace_lem}.  If $f_*\in \tilde H^{-1}(\Omega)$ and $v\in H^1(\Omega)$ satisfies  
\begin{equation}
\label{eq_def_green_2}
\mathcal{L}_{\overline A,\overline q}v=f_*\quad \text{on} \quad \Omega,
\end{equation}
then we also have the following first Green formula,
\begin{align*}
(f_*,u)_{\tilde H^{-1}(\Omega),H^1(\Omega)}&+((\p_\nu+i\overline{A}\cdot\nu)v, u)_{H^{-1/2}(\p \Omega),H^{1/2}(\p \Omega)}=\int_{\Omega}\nabla v\cdot \nabla \overline{u}dx\\
&+i\int_{\Omega}\overline{A}\cdot ( v \nabla \overline{u}- \overline{u} \nabla v)dx 
+\int_{\Omega}(\overline{A}^2+\overline{q})v\overline{u}dx,
\end{align*}
for any $u\in H^1(\Omega)$.  Hence, 
\begin{equation}
\label{eq_first_green_adjoint}
\begin{aligned}
&(f,v)_{\tilde H^{-1}(\Omega),H^1(\Omega)} +((\p_\nu+iA\cdot\nu)u, v)_{H^{-1/2}(\p \Omega),H^{1/2}(\p \Omega)}\\
&=
(u,f_*)_{H^1(\Omega), \tilde H^{-1}(\Omega)}+(u, (\p_\nu+i\overline{A}\cdot\nu)v)_{H^{1/2}(\p \Omega), H^{-1/2}(\p \Omega)},
\end{aligned}
\end{equation}
for any $u,v\in H^1(\Omega)$, which satisfy the equations \eqref{eq_def_green_1} and \eqref{eq_def_green_2}, respectively.  In particular, for any $u,v\in H^1(\Omega)$ such that $\Delta u,\Delta v\in L^2(\Omega)$, we have the second Green  formula,
\begin{equation}
\label{eq_second_green}
\begin{aligned}
(\mathcal{L}_{A,q}u,v)_{L^2(\Omega)}&+ ((\p_\nu+iA\cdot\nu)u, v)_{H^{-1/2}(\p \Omega),H^{1/2}(\p \Omega)}\\
&=(u,\mathcal{L}_{\overline A,\overline q}v)_{L^2(\Omega)}+ (u,(\p_\nu+i\overline A\cdot\nu)v)_{H^{1/2}(\p \Omega),H^{-1/2}(\p \Omega)}.
\end{aligned}
\end{equation}
 
In what follows we shall also work with the operator of the  form $
b\mathcal{L}_{A,q} a^{-1}$,
where $a,b\in C^{1,1}(\overline{\Omega},\R)$,  and $a, b>0$  on $\overline{\Omega}$. 
Let $f\in \tilde H^{-1}(\Omega)$ and let $w\in H^1(\Omega)$ satisfy, 
\[
b\mathcal{L}_{A,q} (a^{-1}w)=f\quad \text{on} \quad \Omega.
\]
Writing
\begin{align*}
b\mathcal{L}_{A,q} (a^{-1}w)=&-\div(ba^{-1}\nabla w)+ (a^{-1}\nabla b-ba^{-1}iA)\cdot\nabla w\\
&-\div( (b\nabla a^{-1} +a^{-1}biA)w)\\
&+(\nabla a^{-1}\cdot\nabla b+iA\cdot (a^{-1}\nabla b-b\nabla a^{-1})+ ba^{-1}(A^2+q))w,
\end{align*}
and using \eqref{eq_trace_lem}, we get the following first Green formula, when $v\in H^1(\Omega)$, 
\begin{equation}
\label{eq_first_green_mod}
\begin{aligned}
(f,v)_{\tilde H^{-1}(\Omega),H^1(\Omega)}&=\int_{\Omega} ba^{-1}\nabla w\cdot\nabla\overline{v}dx+
\int_{\Omega} (a^{-1}\nabla b-ba^{-1}iA )\cdot (\nabla w)\overline{v}dx\\
&+\int_{\Omega} (b\nabla a^{-1} +a^{-1}biA)w\cdot\nabla \overline{v}dx\\
&+\int_{\Omega} (\nabla a^{-1}\cdot\nabla b+iA\cdot (a^{-1}\nabla b-b\nabla a^{-1})+ ba^{-1}(A^2+q))w\overline{v}dx\\
&-(b(\p_\nu+iA\cdot\nu)(a^{-1}w),v)_{H^{-1/2}(\p \Omega),H^{1/2}(\p \Omega)}.
\end{aligned}
\end{equation}

For any $w,v\in H^1(\Omega)$ such that $\Delta w,\Delta v\in L^2(\Omega)$, we have the second Green  formula,
\begin{equation}
\label{eq_second_green_mod}
\begin{aligned}
(b\mathcal{L}_{A,q}(a^{-1}w),v)_{L^2(\Omega)}&+(b(\p_\nu+iA\cdot\nu)(a^{-1}w),v)_{H^{-1/2}(\p \Omega),H^{1/2}(\p \Omega)}\\
=(w,a^{-1}\mathcal{L}_{\overline{A},\overline{q}}(bv))_{L^2(\Omega)}&+
(w,a^{-1}(\p_\nu +i\overline{A}\cdot \nu)(bv))_{H^{1/2}(\p \Omega),H^{-1/2}(\p \Omega)},
\end{aligned}
\end{equation}
which is a consequence of \eqref{eq_second_green}.

Finally, notice that  $A\in W^{1,\infty}(\Omega)$ can be extended to the whole of $\R^n$ so that the extension,  which we denote by the same letter, satisfies  $A\in W^{1,\infty}(\R^n)$.  For the existence of such an extension in the case of Lipschitz domain $\Omega$, we refer to  \cite[Theorem 5, p. 181]{Stein_book}.  Also we have  $A\in C^{0,1}(\overline{\Omega})$.

\section{Direct transmission problem}

\label{subsec_direct}

Let $\Omega\subset\R^n$, $n\ge 3$, be a bounded open set  with Lipschitz boundary, and let $D\subset\subset\Omega$ be a bounded open subset with Lipschitz boundary.  We set as before
\[
D^-=D,\quad\textrm{and}\quad  D^+=\Omega\setminus\overline D.
\]
Let $A^\pm\in W^{1,\infty}(D^\pm,\C^n)$, $q^\pm\in L^\infty(D^\pm,\C)$, $a,b\in C^{1,1}(\overline{D},\R)$,  $c\in C(\overline{D},\R)$,  
$f^\pm\in \tilde H^{-1}(D^\pm)$,   $g_0\in H^{1/2}(\p D)$,  $g_1\in H^{-1/2}(\p D)$,  and $g\in H^{1/2}(\p \Omega)$.  
Notice that since the boundary of $D$ is merely Lipschitz, it is convenient to assume here that the transmission coefficients $a$ and $b$  are defined near $\overline{D}$, rather than on the boundary $\p D$. This is due to the fact that the $C^{1,1}$--regularity of $a$ is needed in order to eliminate the jump across the interface $\p D$ in the solution of the transmission problem, while still working with second order differential operators with bounded coefficients. The corresponding regularity of the coefficient $b$ is needed for similar purposes, when considering the adjoint transmission problem.

Assume furthermore that $a,b>0$ on $\overline{D}$.  For $(u^+,u^-)\in H^1(D^+)\times H^1(D^-)$, we consider the following inhomogeneous transmission problem,
\begin{equation}
\label{eq_transmission_inh}
\begin{aligned}
&\mathcal{L}_{A^+,q^+}u^+=f^+\quad \textrm{in}\quad D^+,\\
&\mathcal{L}_{A^-,q^-}u^-=f^-\quad \textrm{in}\quad D^-,\\
&u^+=au^-+g_0\quad \textrm{on}\quad \p D,\\
(&\p_\nu + iA^+\cdot\nu) u^+=b(\p_\nu+iA^-\cdot \nu)u^-+cu^-+g_1 \quad \textrm{on}\quad \p D,\\
&u^+=g\quad \textrm{on}\quad \p \Omega,
\end{aligned}
\end{equation}
and the corresponding homogeneous transmission problem,
\begin{equation}
\label{eq_transmission_hom}
\begin{aligned}
&\mathcal{L}_{A^+,q^+}u^+=0\quad \textrm{in}\quad D^+,\\
&\mathcal{L}_{A^-,q^-}u^-=0\quad \textrm{in}\quad D^-,\\
&u^+=au^-\quad \textrm{on}\quad \p D,\\
(&\p_\nu + iA^+\cdot\nu) u^+=b(\p_\nu+iA^-\cdot \nu)u^-+cu^-\quad \textrm{on}\quad \p D,\\
&u^+=0\quad \textrm{on}\quad \p \Omega. 
\end{aligned}
\end{equation}
In this paper we shall treat only the Dirichlet boundary conditions on the boundary of $\Omega$, and to this end we introduce the following subspace of $H^1(D^{+})$,
\[
H^1_{d}(D^+)=\{u^+\in H^1(D^+):u^+|_{\p \Omega}=0\}. 
\]

Let us now compute the adjoint transmission problem for the problem \eqref{eq_transmission_hom}. 
 By the second Green formula  \eqref{eq_second_green}, for $(u^+,u^-), (v^+,v^-)\in H^1_d(D^+)\times H^1(D^-)$ such that  $\Delta u^\pm, \Delta v^\pm\in L^2(D^\pm)$,  we get
\begin{equation}
\label{eq_der_adj_1}
\begin{aligned}
(&\mathcal{L}_{A^-,q^-} u^-,v^-)_{L^2(D^- )}+((\p_\nu+iA^-\cdot\nu)u^-,v^-)_{H^{-1/2}(\p D),H^{1/2}(\p D)}\\
&=(u^-,\mathcal{L}_{\overline{A^-},\overline{q^-}}v^-)_{L^2(D^-)} + (u^-,(\p_\nu+i\overline{A^-}\cdot\nu)v^{-})_{H^{1/2}(\p D),H^{-1/2}(\p D)},
\end{aligned}
\end{equation}
and 
\begin{equation}
\label{eq_der_adj_2}
\begin{aligned}
(&\mathcal{L}_{A^+,q^+} u^+,v^+)_{L^2(D^+ )}-((\p_\nu+iA^+\cdot\nu)u^+,v^+)_{H^{-1/2}(\p D),H^{1/2}(\p D)}\\
&=(u^+,\mathcal{L}_{\overline{A^+},\overline{q^+}}v^+)_{L^2(D^+)} - (u^+,(\p_\nu+i\overline{A^+}\cdot\nu)v^{+})_{H^{1/2}(\p D),H^{-1/2}(\p D)}.
\end{aligned}
\end{equation}
Adding \eqref{eq_der_adj_1} and \eqref{eq_der_adj_2}, and using the transmission conditions in \eqref{eq_transmission_hom}, we obtain that
\begin{align*}
(\mathcal{L}_{A^+,q^+}& u^+,v^+)_{L^2(D^+ )}+ (\mathcal{L}_{A^-,q^-} u^-,v^-)_{L^2(D^- )}\\
&=(u^+,\mathcal{L}_{\overline{A^+},\overline{q^+}}v^+)_{L^2(D^+)}
+ (u^-,\mathcal{L}_{\overline{A^-},\overline{q^-}}v^-)_{L^2(D^-)}\\
&+(u^-,-a(\p_\nu+i\overline{A^+}\cdot \nu)v^++(\p_\nu+i\overline{A^-}\cdot \nu)v^-+cv^+)_{H^{1/2}(\p D),H^{-1/2}(\p D)}\\
&+
((\p_\nu+iA^-\cdot \nu)u^-, b v^+-v^-)_{H^{-1/2}(\p D),H^{1/2}(\p D)}.
\end{align*}
Thus, the homogeneous adjoint transmission problem for the problem \eqref{eq_transmission_hom} is given by
\begin{equation}
\label{eq_adjoint_trans_hom}
\begin{aligned}
&\mathcal{L}_{\overline{A^+},\overline{q^+}}v^+=0\quad \textrm{in}\quad D^+,\\
&\mathcal{L}_{\overline{A^-},\overline{q^-}}v^-=0\quad \textrm{in}\quad D^-,\\
& v^+=b^{-1}v^-\quad \textrm{on}\quad \p D,\\
&(\p_\nu+i\overline{A^+}\cdot \nu)v^+=a^{-1}(\p_\nu+i\overline{A^-}\cdot \nu)v^-+ca^{-1}b^{-1} v^-\quad \textrm{on}\quad \p D,\\
&v^+=0\quad \textrm{on}\quad \p \Omega,
\end{aligned}
\end{equation}
and the inhomogeneous adjoint transmission problem is defined by 
\begin{equation}
\label{eq_adjoint_trans_inhom}
\begin{aligned}
&\mathcal{L}_{\overline{A^+},\overline{q^+}}v^+=f^+_*\quad \textrm{in}\quad D^+,\\
&\mathcal{L}_{\overline{A^-},\overline{q^-}}v^-=f^-_*\quad \textrm{in}\quad D^-,\\
& v^+=b^{-1}v^-+g_{0*}\quad \textrm{on}\quad \p D,\\
&(\p_\nu+i\overline{A^+}\cdot \nu)v^+=a^{-1}(\p_\nu+i\overline{A^-}\cdot \nu)v^-+ca^{-1}b^{-1} v^-+g_{1*}\quad \textrm{on}\quad \p D,\\
&v^+=g_{*}\quad \textrm{on}\quad \p \Omega,
\end{aligned}
\end{equation}

In what follows we shall need $w_0\in H^1(D^+)$, which is defined as the unique solution of the following Dirichlet problem,
\begin{equation}
\label{eq_w_0_def}
\begin{aligned}
\Delta w_0&=0\quad \textrm{in}\quad D^+,\\
w_0&=g_0\quad\textrm{on}\quad \p D,\\
w_0&=g\quad\textrm{on}\quad \p \Omega. 
\end{aligned}
\end{equation}
Then 
\begin{equation}
\label{eq_estim_w_0}
\|w_0\|_{H^1(D^+)}\le C(\|g_0\|_{H^{1/2}(\p D)}+\|g\|_{H^{1/2}(\p \Omega)}),\quad C>0.
\end{equation}
Next, $\p_{\nu}w_0|_{\p D}\in H^{-1/2}(\p D)$ is well defined, and \eqref{eq_trace_cont} implies that
\begin{equation}
\label{eq_estim_w_0_2}
\|\p_\nu w_0\|_{H^{-1/2}(\p D)}\le C\|w_0\|_{H^1(D^+)}\le C(\|g_0\|_{H^{1/2}(\p D)}+\|g\|_{H^{1/2}(\p \Omega)}). 
\end{equation}

We have the following result concerning the solvability of the transmission problem \eqref{eq_transmission_inh}. 

\begin{prop} 
\label{prop_2_4} Let $A^\pm\in W^{1,\infty}(D^\pm,\C^n)$, $q^\pm\in L^\infty(D^\pm,\C)$, $a,b\in C^{1,1}(\overline{D},\R)$,  $c\in C(\overline{D},\R)$. 
Assume that $a,b>0$  in $\overline{D}$.  
The transmission problem 
\eqref{eq_transmission_inh} is Fredholm in the sense that there are two mutually exclusive possibilities:

\begin{itemize}

\item[(i)]  The homogeneous transmission problem \eqref{eq_transmission_hom}
has only the trivial solution $(u^+,u^-)=(0,0)$. In this case,  for every $f^\pm\in \tilde H^{-1}(D^\pm)$, $g_0\in H^{1/2}(\p D)$, $g_1\in H^{-1/2}(\p D)$ and $g\in H^{1/2}(\p \Omega)$,  the inhomogeneous transmission problem \eqref{eq_transmission_inh} has a unique solution $(u^+,u^-)\in H^1(D^+)\times H^1(D^-)$, and 
\begin{equation}
\label{eq_energy_estimates_unique}
\begin{aligned}
\|u^+\|_{H^1(D^+)}+\|u^-\|_{H^1(D^-)}\le C(\|f^+\|_{\tilde H^{-1}(D^+)}+\|f^-\|_{\tilde H^{-1}(D^-)}\\+\|g_0\|_{H^{1/2}(\p D)}
+ \|g_1\|_{H^{-1/2}(\p D)}+\|g\|_{H^{1/2}(\p \Omega)}).
\end{aligned}
\end{equation}
Furthermore, for  each $f^\pm_*\in \tilde H^{-1}(D^\pm)$, $g_{0*}\in H^{1/2}(\p D)$, $g_{1*}\in H^{-1/2}(\p D)$ and $g_{*}\in H^{1/2}(\p \Omega)$,
the inhomogeneous adjoint problem \eqref{eq_adjoint_trans_inhom} has a unique solution $(v^+,v^-)\in H^1(D^+)\times H^1(D^-)$.

\item[(ii)] The homogeneous transmission problem \eqref{eq_transmission_hom}
has exactly $p$ linearly independent solutions for some $p\ge 1$. In this case, the homogeneous adjoint transmission problem 
\eqref{eq_adjoint_trans_hom} has exactly $p$ linearly independent solutions $(v_j^+,v_j^-)\in H^1_d(D^+)\times H^1(D^-)$, $1\le j\le p$, and the inhomogeneous transmission problem \eqref{eq_transmission_inh} is solvable if and only if 
\begin{align*}
&(f^+,v_j^+)_{\tilde H^{-1}( D^+),H^1( D^+)}+ (f^-,v_j^-)_{\tilde H^{-1}(D^-),H^1(D^-)}\\
&+(2iA^+\cdot\nabla w_0 +(i(\nabla \cdot A^+)- (A^+)^2-q^+)w_0, v_j^+)_{\tilde H^{-1}( D^+),H^1( D^+)}\\
&-(g_1-\p_\nu w_0 - iA^+\cdot\nu g_0, v_j^+)_{H^{-1/2}(\p D),H^{1/2}(\p D)}=0,\quad 1\le j\le p. 
\end{align*}
Here $w_0\in H^1(D^+)$ is  the  solution of the Dirichlet problem \eqref{eq_w_0_def}.
Furthermore, a solution $(u^+,u^-)\in H^1(D^+)\times H^1(D^-)$ of \eqref{eq_transmission_inh} satisfies the following a priori estimate,
\begin{equation}
\label{eq_energy_estimates}
\begin{aligned}
\|u^+\|_{H^1(D^+)}+\|u^-\|_{H^1(D^-)}\le C(\|f^+\|_{\tilde H^{-1}(D^+)}+\|f^-\|_{\tilde H^{-1}(D^-)}
+\|g_0\|_{H^{1/2}(\p D)}\\
+ \|g_1\|_{H^{-1/2}(\p D)}+\|g\|_{H^{1/2}(\p \Omega)}+\|u^+\|_{L^2(D^+)}+\|u^-\|_{L^2(D^-)}).
\end{aligned}
\end{equation}

\end{itemize}

\end{prop}

The remainder of this section will be devoted to the proof of Proposition \ref{prop_2_4}.  Let us first rewrite the transmission problem \eqref{eq_transmission_inh} by 
making the substitution $u^-=a^{-1}w^-$, in order to eliminate the coefficient $a$ in the first transmission condition. We also multiply the second equation of \eqref{eq_transmission_inh} by $b$, in order to take into account the second transmission condition in \eqref{eq_transmission_inh}, when using the Green formulae. We get
\begin{align*}
&\mathcal{L}_{A^+,q^+}u^+= f^+\quad \textrm{in}\quad D^+,\\
&b\mathcal{L}_{A^-,q^-}(a^{-1}w^-)
= bf^- \quad \textrm{in}\quad D^-,\\
&u^+=w^-+g_0\quad \textrm{on}\quad \p D,\\
&(\p_\nu + iA^+\cdot\nu) u^+=b(\p_\nu+iA^-\cdot \nu)(a^{-1}w^-)+ ca^{-1}w^-+g_1 \quad \textrm{on}\quad \p D,\\
&u^+=g\quad \textrm{on}\quad \p \Omega. 
\end{align*}
Notice that $bf^-$ is well-defined as an element of $\tilde H^{-1}(D^-)$
by Proposition \ref{prop_sob_mult}.

Setting $w^+=u^+- w_0$, where $w_0\in H^1(D^+)$ is  the  solution of the Dirichlet problem \eqref{eq_w_0_def}, we obtain that
\begin{equation}
\label{eq_trans_lem_1}
\begin{aligned}
&\mathcal{L}_{A^+,q^+}w^+=\tilde f^++ f^+\quad \textrm{in}\quad D^+,\\
&b\mathcal{L}_{A^-,q^-}(a^{-1}w^-)
= bf^-  \quad \textrm{in}\quad D^-,\\
&w^+=w^-\quad \textrm{on}\quad \p D,\\
&(\p_\nu + iA^+\cdot\nu) w^+=b(\p_\nu+iA^-\cdot \nu)(a^{-1}w^-) 
+  ca^{-1}w^-+\tilde g_1\quad \textrm{on}\quad \p D,\\
&w^+=0\quad \textrm{on}\quad \p \Omega.
\end{aligned}
\end{equation}
Here
\begin{equation}
\label{eq_f_lem}
\tilde f^+=-\mathcal{L}_{A^+,q^+}w_0=2iA^+\cdot\nabla w_0 +(i(\nabla \cdot A^+)- (A^+)^2-q^+)w_0\in  L^2(D^+),
\end{equation}
\begin{equation}
\label{eq_g_1_lem}
\tilde g_1=g_1-\p_\nu w_0 - iA^+\cdot\nu g_0\in H^{-1/2}(\p D). 
\end{equation}

We shall now discuss the variational formulation of the transmission problem \eqref{eq_trans_lem_1}. To that end, let us introduce the following sesquilinear form, associated with the operator $\mathcal{L}_{A^+,q^+}$  in the transmission problem \eqref{eq_trans_lem_1}, cf. the first Green formula \eqref{eq_first_green},
\begin{equation}
\label{eq_form_mag}
\begin{aligned}
\Phi^+:H^1_d(D^+)&\times H^1_d(D^+)\to \C,\\
\Phi^+(w^+,v^+)=&\int_{D^+}\nabla w^+\cdot \nabla \overline{v^+}dx+i\int_{D^+}A^+\cdot(w^+ \nabla \overline{v^+}- \overline{v^+}\nabla w^+)dx \\
&+\int_{D^+}((A^+)^2+q^+)w^+\overline{v^+}dx.
\end{aligned}
\end{equation}
Associated with 
 the operator  $b\mathcal{L}_{A^-,q^-}a^{-1}$ in the transmission problem \eqref{eq_trans_lem_1}, is  the following  sesquilinear form, cf. the first Green formula \eqref{eq_first_green_mod} for $b\mathcal{L}_{A^-,q^-}a^{-1}$,
\begin{equation}
\label{eq_form_mag_mod}
\begin{aligned}
\Phi^-&:H^1(D^-)\times H^1(D^-)\to \C,\\
\Phi^-&(w^-,v^-)=\int_{D^-} ba^{-1}\nabla w^-\cdot\nabla\overline{v^-}dx+
\int_{D^-} (a^{-1}\nabla b-ba^{-1}iA^-)\cdot (\nabla w^-)\overline{v^-}dx\\
&+\int_{D^-} (b\nabla a^{-1} +a^{-1}biA^-)w^-\cdot\nabla \overline{v^-}dx\\
&+\int_{D^-} (\nabla a^{-1}\cdot\nabla b+iA^-\cdot (a^{-1}\nabla b-b\nabla a^{-1})+ ba^{-1}((A^-)^2+q^-))w^-\overline{v^-}dx.
\end{aligned}
\end{equation}

For a solution $(w^+,w^-)\in H^1_d(D^+)\times H^1(D^-)$ of the transmission problem \eqref{eq_trans_lem_1}, the function
\[
w=\begin{cases} w^+& \textrm{on } D^+,\\
w^- &  \textrm{on } D^-,
\end{cases}
\]
satisfies $w\in H^1_0(\Omega)$, as $w^+=w^-$ on $\p D$. 

Let $v\in H^1_0(\Omega)$, and set 
\[
v^+=v|_{D^+},\quad v^-=v|_{D^-}. 
\]
Then by the first Green formula \eqref{eq_first_green} applied to the first equation in \eqref{eq_trans_lem_1}, we obtain that
\begin{equation}
\label{eq_trans_lem_2}
\begin{aligned}
\Phi^+(w^+,v^+)=&(\tilde f^++f^+,v^+)_{\tilde H^{-1}(D^+),H^1( D^+)}\\
&-((\p_\nu+iA^+\cdot \nu )w^+,v^+)_{H^{-1/2}(\p D),H^{1/2}(\p D)}.
\end{aligned}
\end{equation}
Using  the first Green formula \eqref{eq_first_green_mod} for $b\mathcal{L}_{A^-,q^-}a^{-1}$,  we get
\begin{equation}
\label{eq_trans_lem_3}
\begin{aligned}
\Phi^-(w^-,v^-)
=&(bf^-,v^-)_{\tilde H^{-1}(D^-),H^1(D^-)}\\
&+ (b(\p_\nu+iA^{-}\cdot\nu )(a^{-1}w^-),v^-)_{H^{-1/2}(\p D),H^{1/2}(\p D)}.
\end{aligned}
\end{equation}
Adding \eqref{eq_trans_lem_2}
and \eqref{eq_trans_lem_3}, and using the transmission conditions in \eqref{eq_trans_lem_1}, we have
\begin{equation}
\label{eq_trans_lem_main}
\begin{aligned}
\Phi^+(w^+,v^+)&+ \Phi^-(w^-,v^-)+\int_{\p D} ca^{-1}w\overline{v}dS=
(\tilde f^++f^+,v^+)_{\tilde H^{-1}( D^+),H^1( D^+)}\\
&+ (bf^-,v^-)_{\tilde H^{-1}(D^-),H^1(D^-)}-(\tilde g_1, v)_{H^{-1/2}(\p D),H^{1/2}(\p D)},
\end{aligned}
\end{equation}
for any $v\in H^1_0(\Omega)$. 
The sesquilinear form,
\begin{equation}
\label{eq_form_main_var}
\begin{aligned}
&\Phi:H^1_0(\Omega)\times H^1_0(\Omega)\to \C,\\
&\Phi(w,v)=\Phi^+(w^+,v^+)+ \Phi^-(w^-,v^-)+\int_{\p D} ca^{-1}w\overline{v} dS,
\end{aligned}
\end{equation}
is naturally associated to the transmission problem \eqref{eq_trans_lem_1}.

It follows from \eqref{eq_trans_lem_main} that  a solution  $(w^+,w^-)$ of the transmission problem \eqref{eq_trans_lem_1} satisfies the equation
\begin{equation}
\label{eq_variational_formul}
\begin{aligned}
\Phi(w,v)=&(\tilde f^++f^+,v^+)_{\tilde H^{-1}( D^+),H^1( D^+)}+ (bf^-,v^-)_{\tilde H^{-1}(D^-),H^1(D^-)}\\
&-(g_1-\p_\nu w_0 - iA^+\cdot\nu g_0, v)_{H^{-1/2}(\p D),H^{1/2}(\p D)},
\end{aligned}
\end{equation}  
for any $v\in H^1_0(\Omega)$. Here we have used the definition \eqref{eq_g_1_lem} of $\tilde g_1$. 

Conversely,  if $w\in H^1_0(\Omega)$ satisfies \eqref{eq_variational_formul} for any $v\in H^1_0(\Omega)$, then $(w^+,w^-)$ solves the transmission problem \eqref{eq_trans_lem_1}. Indeed, restricting the attention to $v\in C^\infty_0(D^+)$ and $v\in C^\infty_0(D^-)$ in  \eqref{eq_variational_formul},  we obtain that $w^+$ and $w^-$ satisfy the first and second equations of the transmission problem  \eqref{eq_trans_lem_1}.
Then by the first Green formulae \eqref{eq_first_green} and \eqref{eq_first_green_mod}, we obtain \eqref{eq_trans_lem_2} and \eqref{eq_trans_lem_3}, respectively,  for any $v\in H^1_0(\Omega)$. This, together with \eqref{eq_variational_formul}, implies that the second transmission condition in 
\eqref{eq_trans_lem_1} holds.

Hence, $w\in H^1_0(\Omega)$ satisfies \eqref{eq_variational_formul} for any $v\in H^1_0(\Omega)$, if and only if $u^+=w^++w^0\in H^1(D^+)$ and $u^-=a^{-1}w^-\in H^1(D^-)$ satisfy the transmission problem \eqref{eq_transmission_inh}.

We shall next make the following observation.

\begin{lem}
The sesquilinear form $\Phi:H^1_0(\Omega)\times H^1_0(\Omega)\to \C$ is
bounded, i.e. 
\[
|\Phi(w,v)|\le C\|w\|_{H^1(\Omega)} \|v\|_{H^1(\Omega)},
\]
and coercive on $H^1_0(\Omega)$, i.e. 
\begin{equation}
\label{eq_coercivity_estim}
\Re \Phi(w,w)\ge c\|w\|_{H^1(\Omega)}^2-C\|w\|_{L^2(\Omega)}^2,\quad c>0,
\end{equation}
for all $w\in  H^1_0(\Omega)$.

\end{lem}

\begin{proof} As the boundedness of $\Phi$ is clear, only the coercivity needs to be verified. By the definition \eqref{eq_form_main_var} of $\Phi$, we have
\begin{equation}
\label{eq_coercivity_estim_proof_1}
\Re \Phi(w,w)\ge \Re \Phi^+(w^+,w^+)+ \Re \Phi^-(w^-,w^-)-\int_{\p D} |ca^{-1}w\overline{w}| dS.
\end{equation}
It follows from  \eqref{eq_form_mag}  that
\begin{equation}
\label{eq_coercivity_estim_proof_2}
\begin{aligned}
\Re \Phi^+(w^+,w^+)&\ge \|\nabla w^+\|^2_{L^2(D^+)}-\int_{D^+ }|A^+ \cdot( w^+  \nabla \overline{w^+}-   
\overline{w^+}\nabla w^+) |dx\\
&-\int_{D^+}|((A^+)^2+q^+)w^+\overline{w^+}|dx\ge 
\| w^+\|^2_{H^1(D^+)}-\| w^+\|^2_{L^2(D^+)}\\
&-C\|w^+\|_{L^2(D^+)}\|\nabla w^+\|_{L^2(D^+)}-C\|w^+\|_{L^2(D^+)}^2\\
&\ge \| w^+\|^2_{H^1(D^+)}-C\frac{1}{\varepsilon}\|w^+\|_{L^2(D^+)}^2-C\varepsilon \| w^+\|^2_{H^1(D^+)}\\
& \ge\frac{1}{2} \| w^+\|^2_{H^1(D^+)}- C\|w^+\|_{L^2(D^+)}^2,
\end{aligned}
\end{equation}
provided that $\varepsilon >0$ is small enough.

Using  \eqref{eq_form_mag_mod} together with the fact that $ba^{-1}>0$ on $\overline{D}$, we get 
\begin{equation}
\label{eq_coercivity_estim_proof_3}
\begin{aligned}
&\Re \Phi^-(w^-,w^-)\ge  \int_{D^-} ba^{-1}|\nabla w^-|^2dx-
\int_{D^-} |(a^{-1}\nabla b-ba^{-1}iA^-)\cdot (\nabla w^-)\overline{w^-}|dx\\
&-\int_{D^-} |(b\nabla a^{-1} +a^{-1}biA^-)w^-\cdot\nabla \overline{w^-}|dx\\
&-\int_{D^-}| (\nabla a^{-1}\cdot\nabla b+iA^-\cdot (a^{-1}\nabla b-b\nabla a^{-1})+ ba^{-1}((A^-)^2+q^-))||w^-|^2dx\\
&\ge \frac{1}{C}(\| w^-\|^2_{H^1(D^-)}-\| w^-\|^2_{L^2(D^-)})-C\|w^-\|_{L^2(D^-)}\|\nabla w^-\|_{L^2(D^-)}-C\|w^-\|_{L^2(D^-)}^2\\
&\ge \frac{1}{C}\| w^-\|^2_{H^1(D^-)} -C\frac{1}{\varepsilon}\|w^-\|_{L^2(D^-)}^2-C\varepsilon \| w^-\|^2_{H^1(D^-)}\\
&\ge \frac{1}{2C}\| w^-\|^2_{H^1(D^-)}-C\|w^-\|_{L^2(D^-)}^2,
\end{aligned}
\end{equation}
provided that $\varepsilon >0$ is small enough. 

In order to estimate the last term in the right hand side of \eqref{eq_coercivity_estim_proof_1}, we notice that 
by interpolation estimates for Sobolev spaces, for any $\varepsilon>0$, there is $C(\varepsilon)>0$ such that
\[
\|w^-\|_{H^{3/4}(D^-)}\le \varepsilon \|w^-\|_{H^1(D^-)}+ C(\varepsilon)\|w^-\|_{L^2(D^-)},
\]
see \cite[Theorem 7.22]{Grubbbook2009}.  This fact together with the trace theorem implies that 
\begin{equation}
\label{eq_coercivity_estim_proof_4}
\begin{aligned}
\int_{\p D} |ca^{-1}w\overline{w}| dS  &\le C\|w^-\|_{L^2(\p D)}^2\le C\|w^-\|^2_{H^{1/4}(\p D)}\le C\|w^-\|_{H^{3/4}(D^-)}^2\\
&\le C(\varepsilon \|w^-\|_{H^1(D^-)}^2+ C(\varepsilon)\|w^-\|_{L^2(D^-)}^2). 
\end{aligned}
\end{equation}
Combining  \eqref{eq_coercivity_estim_proof_2} -- \eqref{eq_coercivity_estim_proof_4} and choosing $\varepsilon >0$ small enough, we obtain \eqref{eq_coercivity_estim}. 
This completes the proof of the lemma. 
\end{proof}

Recall next that the $L^2$ dual space of $H^1_0(\Omega)$ is $H^{-1}(\Omega)$, and let $\mathcal{A}:H^1_0(\Omega)\to H^{-1}(\Omega)$ be the bounded linear operator determined by the form $\Phi$, given by \eqref{eq_form_main_var}, 
\begin{equation}
\label{eq_op_A}
(\mathcal{A}w,v)_{H^{-1}(\Omega),H^1_0(\Omega)}=\Phi(w,v),
\end{equation}
for $w,v\in H^{1}_0(\Omega)$.  Since the inclusion $H^1_0(\Omega)\hookrightarrow L^2(\Omega)$ is compact, and the form $\Phi$ is coercive on $H^1_0(\Omega)$,  the operator $\mathcal{A}$ is Fredholm of index zero, see \cite[Theorem 2.34]{McLean_book}.

Let us now show that $w\in H^1_0(\Omega)$ solves the homogeneous equation $\mathcal{A}w=0$  if and only if $u^+=w^+\in H^1(D^+)$ and $u^-=a^{-1}w^-\in H^1(D^-)$ solve the homogeneous transmission problem \eqref{eq_transmission_hom}. Indeed, let $w\in H^1_0(\Omega)$ be a solution to $\mathcal{A}w=0$. Then
\begin{equation}
\label{eq_hom_form_1}
(\mathcal{A}w,v)_{H^{-1}(\Omega),H^1_0(\Omega)}=\Phi(w,v)=0,
\end{equation}
for any $v\in H^1_0(\Omega)$. Choosing  $v\in C^\infty_0(D^+)$ and $v\in C^\infty_0(D^-)$ be arbitrary functions, \eqref{eq_hom_form_1} implies that $w^+$ and $w^-$ satisfy  
\begin{align*}
&\mathcal{L}_{A^+,q^+}w^+=0\quad \textrm{in}\quad D^+,\\
&b\mathcal{L}_{A^-,q^-}(a^{-1}w^-)
= 0  \quad \textrm{in}\quad D^-. 
\end{align*}
Taking $v\in H^1_0(\Omega)$ and using the first Green formulae \eqref{eq_first_green} and \eqref{eq_first_green_mod} together with the definition \eqref{eq_form_main_var} of the form $\Phi$, we get 
\begin{align*}
0&=\Phi(w,v)=-((\p_\nu+iA^+\cdot \nu )w^+,v)_{H^{-1/2}(\p D),H^{1/2}(\p D)}\\
&+ (b(\p_\nu+iA^{-}\cdot\nu )(a^{-1}w^-),v)_{H^{-1/2}(\p D),H^{1/2}(\p D)}+\int_{\p D} ca^{-1}w^-\overline{v} dS,
\end{align*}
for any $v\in H^1_0(\Omega)$. Thus,
\[
(\p_\nu+iA^+\cdot \nu )w^+=b(\p_\nu+iA^{-}\cdot\nu )(a^{-1}w^-)+  ca^{-1}w^-,\quad\textrm{on}\quad \p D,
\]
and therefore, $u^+=w^+$ and $u^-=a^{-1}w^-$ solve the homogeneous transmission problem \eqref{eq_transmission_hom}. The opposite direction is obvious.

We shall next study the inhomogeneous case. When doing so, following \eqref{eq_variational_formul}, let us introduce the linear functional $F:H^1_0(\Omega)\to \C$ by
\begin{equation}
\label{eq_def__fun_F}
\begin{aligned}
(F,v)=&(\tilde f^++f^+,v^+)_{\tilde H^{-1}( D^+),H^1( D^+)}+ (bf^-,v^-)_{\tilde H^{-1}(D^-),H^1(D^-)}\\
&-(g_1-\p_\nu w_0 - iA^+\cdot\nu g_0, v)_{H^{-1/2}(\p D),H^{1/2}(\p D)},
\end{aligned}
\end{equation}
for any $v\in H^1_0(\Omega)$.  

\begin{lem}
\label{lem_estim_F}
The  functional $F:H^1_0(\Omega)\to \C$ is bounded and 
\begin{align*}
\|F\|_{H^{-1}(\Omega)}\le C&(\|f^+\|_{\tilde H^{-1}( D^+)}+\|f^-\|_{\tilde H^{-1}(D^-)}+  \|g_1\|_{H^{-1/2}(\p D)}\\
&+ \|g_0\|_{H^{1/2}(\p D)}+\|g\|_{H^{1/2}(\p \Omega)}). 
\end{align*}
\end{lem}

\begin{proof}
We have
\begin{equation}
\label{eq_F_1}
|(f^+,v^+)_{\tilde H^{-1}( D^+),H^1( D^+)}|\le C\|f^+\|_{\tilde H^{-1}( D^+)}\|v\|_{H^1(\Omega)},
\end{equation}
and  by Proposition \ref{prop_sob_mult}, we get
\begin{equation}
\label{eq_F_2}
|(bf^-,v^-)_{\tilde H^{-1}(D^-),H^1(D^-)}|\le C\|f^-\|_{\tilde H^{-1}(D^-)}\|v\|_{H^1(\Omega)}. 
\end{equation}
Using \eqref{eq_f_lem}   and  \eqref{eq_estim_w_0}, we obtain that
\begin{equation}
\label{eq_F_3}
\begin{aligned}
|(\tilde f^+,&v^+)_{L^2(D^+),L^2( D^+)}|\le\int_{D^+} |(2iA^+\cdot\nabla w_0)\overline{v^+}|dx\\
&+\int_{ D^+}
|(i(\nabla \cdot A^+)- (A^+)^2-q^+)w_0\overline{v^+}|dx\le C\|w_0\|_{H^1( D^+)}\|v\|_{L^2(\Omega)}\\
&\le 
C(\|g_0\|_{H^{1/2}(\p D)}+\|g\|_{H^{1/2}(\p \Omega)}) \|v\|_{H^1(\Omega)}. 
\end{aligned}
\end{equation}
With the help of \eqref{eq_estim_w_0_2} and the trace theorem, we have  
\begin{equation}
\label{eq_F_4}
\begin{aligned}
& |(g_1-\p_\nu w_0 - iA^+\cdot\nu g_0, v)_{H^{-1/2}(\p D),H^{1/2}(\p D)}|\\
&\le C(\|g_1\|_{H^{-1/2}(\p D)}+ \|\p_{\nu}w_0\|_{H^{-1/2}(\p D)}+\|iA^+\cdot\nu g_0 \|_{L^2(\p D)})\|v\|_{H^{1/2}(\p D)}\\
&\le C(\|g_1\|_{H^{-1/2}(\p D)}+ \|g_0\|_{H^{1/2}(\p D)}+\|g\|_{H^{1/2}(\p \Omega)})\|v\|_{H^1(\Omega)}.
\end{aligned}
\end{equation}
Combining \eqref{eq_F_1} -- \eqref{eq_F_4}, we obtain that
\begin{align*}
|(F,v)|\le C&(\|f^+\|_{\tilde H^{-1}( D^+)}+\|f^-\|_{\tilde H^{-1}(D^-)}+  \|g_1\|_{H^{-1/2}(\p D)}+ \|g_0\|_{H^{1/2}(\p D)}\\
&+\|g\|_{H^{1/2}(\p \Omega)})\|v\|_{H^1(\Omega)},
\end{align*}
which completes the proof. 
\end{proof}

Given $F$ in \eqref{eq_def__fun_F}, the function $w\in H^1_0(\Omega)$ satisfies $\mathcal{A}w=F$ in $\Omega$ precisely when 
$u^+=w^++w_0$ and $u^-=a^{-1}w^-$ satisfy the transmission problem \eqref{eq_transmission_inh}.

We shall next introduce the adjoint of $\mathcal{A}$ and relate the corresponding operator equation to the adjoint transmission problem.  
Let $\mathcal{A}^*:H^1_0(\Omega)\to H^{-1}(\Omega)$ be the adjoint to $\mathcal{A}$, i.e.
\[
(\mathcal{A}^*v,w)_{H^{-1}(\Omega),H^1_0(\Omega)}=(v,\mathcal{A}w)_{H^{1}_0(\Omega),H^{-1}(\Omega)}=\overline{(\mathcal{A}w,v)}_{H^{-1}(\Omega),H^1_0(\Omega)}=\overline{\Phi(w,v)},
\]
for any $w,v\in H^1_0(\Omega)$. 

Consider now the  inhomogeneous adjoint transmission problem \eqref{eq_adjoint_trans_inhom}. Similarly to what was done above we make the substitutions $y^-=b^{-1}v^-$ and $y^+=v^+-w_{0*}$, where 
$w_{0*}\in H^1(D^+)$ is the unique solution of the following Dirichlet problem,
\begin{align*}
\Delta w_{0*}&=0\quad \textrm{in}\quad D^+,\\
w_{0*}&=g_{0*}\quad\textrm{on}\quad \p D,\\
w_{0*}&=g_*\quad\textrm{on}\quad \p \Omega. 
\end{align*}
We get
\begin{equation}
\label{eq_inhom_ad}
\begin{aligned}
&\mathcal{L}_{\overline{A^+},\overline{q^+}}y^+=\tilde f^+_*+ f^+_*\quad \textrm{in}\quad D^+,\\
&a^{-1}\mathcal{L}_{\overline{A^-},\overline{q^-}}(by^-)
= a^{-1}f^-_*  \quad \textrm{in}\quad D^-,\\
&y^+=y^-\quad \textrm{on}\quad \p D,\\
&(\p_\nu + i\overline{A^+}\cdot\nu) y^+=a^{-1}(\p_\nu+i\overline{A^-}\cdot \nu)(by^-) 
+  ca^{-1}y^-+\tilde g_{1*}\quad \textrm{on}\quad \p D,\\
&y^+=0\quad \textrm{on}\quad \p \Omega,
\end{aligned}
\end{equation}
where
\begin{align*}
\tilde f^+_*&=-\mathcal{L}_{\overline{A^+},\overline{q^+}}w_{0*}\in  L^2(D^+),\\
\tilde g_{1*}&=g_{1*}-\p_\nu w_{0*} - i\overline{A^+}\cdot\nu g_{0*}\in H^{-1/2}(\p D). 
\end{align*}
A straightforward computation shows that the sesquilinear form, naturally  associated to the transmission problem \eqref{eq_inhom_ad}, is given by
\[
(y,z)\mapsto \overline{\Phi(z,y)},
\]
where $\Phi$ is defined by \eqref{eq_form_main_var}.

Applying the Fredholm alternative to the operator $\mathcal{A}$, we obtain the two mutually exclusive possibilities in the statement of Proposition \ref{prop_2_4}.
The estimate \eqref{eq_energy_estimates_unique} follows from the fact that in case (i), the operator $\mathcal{A}$ is invertible and the unique solution $w$ to $\mathcal{A}w=F$ satisfies 
\[
\|w\|_{H^1(\Omega)}\le C\|F\|_{H^{-1}(\Omega)}.
\]
It suffices then to apply Lemma \ref{lem_estim_F}. 

It only remains to prove the a priori estimate \eqref{eq_energy_estimates}. To this end, we observe that a solution $w$ of the equation $\mathcal{A}w=F$ satisfies
\[
\Phi(w,w)=(F,w)_{H^{-1}(\Omega),H^1_0(\Omega)},
\]
and therefore,
\[
|\Phi(w,w)|\le C\|F\|_{H^{-1}(\Omega)}\|w\|_{H^1(\Omega)}\le C\bigg(\frac{1}{\varepsilon}\|F\|^2_{H^{-1}(\Omega)}+\varepsilon\|w\|_{H^1(\Omega)}^2\bigg).
\]
Taking $\varepsilon>0$ to be small enough and using the coercivity estimate \eqref{eq_coercivity_estim}, we infer that 
\begin{equation}
\label{eq_energy_estimates_prep}
\|w\|_{H^1(\Omega)}\le C(\|F\|_{H^{-1}(\Omega)}+\|w\|_{L^2(\Omega)}). 
\end{equation}
Substituting $w^+=u^+-w_0$ and $w^-=au^-$ into \eqref{eq_energy_estimates_prep}, and using the fact that  
\[
\frac{1}{C}\|u^-\|_{H^1(D^-)} \le \|au^-\|_{H^1(D^-)}\le C\|u^-\|_{H^1(D^-)}, 
\]
together with Lemma \ref{lem_estim_F}, 
we obtain  the estimate \eqref{eq_energy_estimates}. This completes the proof of Proposition \ref{prop_2_4}.

\begin{rem}  
The transmission problem \eqref{eq_transmission_hom} is self-adjoint, i.e. the adjoint transmission problem \eqref{eq_adjoint_trans_hom} coincides with the given one, if $A^\pm$ and $q^\pm$ are real-valued and $ab=1$ in $\p D$.  As $\p D$ is merely Lipschitz, we shall require that the equality $ab=1$ holds in the entire region $\overline{D}$.  In this case one can easily check that $\Phi(w,v)=\overline{\Phi(v,w)}$ for $w,v\in H^1_0(\Omega)$ and therefore, the operator $\mathcal{A}$, given by \eqref{eq_op_A},  is self-adjoint.  
\end{rem}

\section{Unique solvability of the transmission problem by domain perturbation in the self-adjoint case}

\label{sec_domain_pertur}

Let $\Omega\subset\R^n$, $n\ge 3$,  be a bounded open set in $\R^n$ with Lipschitz boundary, and let $D\subset\subset\Omega$ be a bounded open set with Lipschitz boundary such that $\Omega\setminus\overline{D}$ is connected.  Let $A^+\in W^{1,\infty}(\R^n\setminus\overline{D},\R^n)$, $q^+\in L^{\infty}(\R^n\setminus\overline{D},\R)$,  $A^-\in W^{1,\infty}(D,\R^n)$, $q^-\in L^\infty(D,\R)$, $a,b\in C^{1,1}(\overline{D},\R)$, and  $c\in C(\overline{D},\R)$. 
Assume that $a,b>0$  in $\overline{D}$.  Then  the Fredholm alternative holds for the transmission problem  \eqref{eq_transmission_inh}.  In what follows we shall assume that $ab=1$ in $\overline{D}$.

The purpose of this section is to study the question whether the unique solvability 
of the problem  \eqref{eq_transmission_inh} can be achieved by a small perturbation of the boundary of $\Omega$.  Using the mini-max principle, here we shall answer this question affirmatively. 

The unique solvability of the transmission problem  \eqref{eq_transmission_inh} in $\Omega$ is equivalent to the fact that the transmission problem  
\begin{equation}
\label{eq_trans_eig_1}
\begin{aligned}
&\mathcal{L}_{A^+,q^+}w^+=0\quad \textrm{in}\quad  \Omega\setminus\overline{D},\\
&b\mathcal{L}_{A^-,q^-}(a^{-1}w^-)
= 0  \quad \textrm{in}\quad D,\\
&w^+=w^-\quad \textrm{on}\quad \p D,\\
&(\p_\nu + iA^+\cdot\nu) w^+=b(\p_\nu+iA^-\cdot \nu)(a^{-1}w^-) 
+  ca^{-1}w^-\quad \textrm{on}\quad \p D,\\
&w^+=0\quad \textrm{on}\quad \p  \Omega,
\end{aligned}
\end{equation}
only has the trivial solution.  As in the proof of Proposition \ref{prop_2_4}, the sesquilinear form $\Phi_\Omega:=\Phi$, given by \eqref{eq_form_main_var},  is naturally associated to the transmission problem \eqref{eq_trans_eig_1}. 
One can easily check that $\Phi_{\Omega}(w,v)=\overline{\Phi_{\Omega}(v,w)}$ for $w,v\in H^1_0(\Omega)$. In what follows, it will be convenient to work in the $L^2(\Omega)$--framework. It follows from \eqref{eq_coercivity_estim}
that  $\Phi_{\Omega}$, viewed as an unbounded form on $L^2(\Omega)$, is densely defined and closed on $H^1_0(\Omega)$.  Moreover, \eqref{eq_coercivity_estim} implies also that $\Phi_{\Omega}$ is 
semibounded from below, i.e. 
\[
\Phi_{\Omega}(w,w)\ge -C\|w\|_{L^2( \Omega)}, 
\]
for all $w\in H^1_0( \Omega)$.  Then there exists a unique self-adjoint operator $Q_{\Omega}:L^2( \Omega)\to L^2( \Omega)$ with domain $\mathcal{D}(Q_{\Omega})=\{w\in H^1_0( \Omega):\mathcal{A}_{\Omega}w\in L^2( \Omega)\}$. Here the bounded operator $\mathcal{A}_{\Omega}:H^1_0( \Omega)\to H^{-1}(\Omega)$ is defined by \eqref{eq_op_A}, and we have $(Q_{\Omega}w,v)_{L^2( \Omega)}=\Phi_{\Omega}(w,v)$ for any $w,v\in \mathcal{D}(Q)$, see 
\cite[Theorem VIII.15]{Reed_Simon_book_1}.  Moreover, $Q_{\Omega}\ge -C$ and $Q_\Omega$ has a discrete spectrum, consisting of real eigenvalues of finite multiplicity, accumulating at $+\infty$, 
 \begin{equation}
\label{eq_sets_4}
\lambda_1( \Omega)\le \lambda_2( \Omega)\le \cdots\le \lambda_n( \Omega)\to +\infty.
\end{equation}
The transmission problem \eqref{eq_trans_eig_1} in $\Omega$ only has the trivial solution precisely when zero is not an eigenvalue of the operator $Q_{\Omega}$.  Assuming that zero is an eigenvalue of $Q_{\Omega}$, the 
question that we are interested in is therefore, whether we can perturb the boundary of $\Omega$, to remove zero from the spectrum.

Let us first recall the following well-known consequence of the mini-max principle, applied to the eigenvalues of the semibounded self-adjoint operator $Q_{\Omega}$, see \cite[Theorem 4.5]{Leis_book_1986}.   

\begin{prop}
\label{prop_mon_eigen}
Let $\tilde \Omega\subset\R^n$ be  a bounded open set with Lipschitz boundary such that  $\Omega\subset \tilde \Omega$. Then $\lambda_n( \Omega)\ge \lambda_n( \tilde\Omega)$ for all $n=1,2,\dots$. 
\end{prop}

We shall need a sharper result, showing the strict monotonicity of the eigenvalues of the operator $Q_{\Omega}$, with respect to the domain. The proof of this result follows closely the proof of the strict monotonicity of the eigenvalues of a self-adjoint second order elliptic operator, due to \cite{Leis_1967}, see also \cite[Theorem 4.7]{Leis_book_1986}.  The key ingredient here is the unique continuation principle.

 \begin{prop}
\label{prop_mon_eigen_stric}

Let $D\subset\subset \Omega\subset \tilde \Omega$,  $\tilde \Omega\setminus\overline{\Omega}\neq \emptyset$, and let  $\tilde \Omega\setminus\overline{D}$ and $\Omega\setminus\overline{D}$ be connected.  Then $\lambda_n( \Omega)> \lambda_n( \tilde\Omega)$ for all $n=1,2,\dots$.
 
 \end{prop}
 
 \begin{proof}
 
Assume that $\lambda_n( \Omega)= \lambda_n( \tilde\Omega)$ for some $n$, and choose $m>n$ such that $\lambda_n( \tilde\Omega)<\lambda_m( \tilde\Omega)$.   As $\tilde \Omega\setminus\overline{\Omega}$ is non-empty, 
there are open sets $\Omega_j$, $j=1,\dots, m$,  such that 
\begin{equation}
\label{eq_sets_1}
\Omega=\Omega_1\subset\Omega_2\subset \dots\subset \Omega_m=\tilde \Omega,
\end{equation}
with $\Omega_j\setminus\overline{D}$ connected, for all $j=1,\dots, m$, and $U_j:=\Omega_j\setminus\overline{\Omega_{j-1}}$ non-empty, $j=2,\dots, m$. We also set $U_1:=\Omega_1$. 
 
It follows from \eqref{eq_sets_1} together with Proposition \ref{prop_mon_eigen}, and the fact that $\lambda_n(\Omega)= \lambda_n(\tilde \Omega)$, that $\lambda_n(\Omega_j)= \lambda_n(\tilde \Omega)$, $j=1,\dots, m$.  Let 
$w_j\in \mathcal{D}(Q_{\Omega_j})\subset H^1_0(\Omega_j)$ be a corresponding eigenfunction of the operator $Q_{\Omega_j}$, i.e. 
 \begin{equation}
\label{eq_sets_2}
 Q_{\Omega_j}w_j= \lambda_n(\tilde \Omega) w_j\quad \textrm{in}\quad \Omega_j,\quad j=1,\dots, m.
 \end{equation}
 Setting 
 \[
 \tilde w_j=\begin{cases} w_j& \textrm{in}\quad \Omega_j,\\
 0 &  \textrm{in}\quad \tilde \Omega\setminus\overline{\Omega_j},
 \end{cases}\in H^1_0(\tilde \Omega),\quad j=1,\dots, m-1, \quad \tilde w_m=w_m,
 \]
 we claim that $\tilde w_1,\dots,\tilde w_m$ are linearly independent. Indeed, assuming the contrary, we obtain that for some $j$,  $w_j=0$ in $U_j$. Notice that the case $j=1$ cannot occur. If $j\ge 2$, then it follows from 
 \eqref{eq_sets_2}  that 
 \[
 \lambda_n(\tilde \Omega) (w_j, v)_{L^2(\Omega_j)}=(Q_{\Omega_j}w_j, v)_{L^2(\Omega_j)}=\Phi_{\Omega_j}(w_j,v),
 \]
 for all $v\in H^1_0(\Omega_j)$.  This together with \eqref{eq_form_main_var} implies that $w^+_j=w_j|_{\Omega_j\setminus\overline{D}}$ and $w^-_j=w_j|_D$ satisfy the following transmission problem, 
  \begin{equation}
\label{eq_sets_3}
 \begin{aligned}
&\mathcal{L}_{A^+,q^+}w^+_j=\quad \lambda_n(\tilde \Omega) w^+_j \quad \textrm{in}\quad \Omega_j\setminus\overline{D},\\
&b\mathcal{L}_{A^-,q^-}(a^{-1}w^-_j)
= \lambda_n(\tilde \Omega) w^-_j  \quad \textrm{in}\quad D,\\
&w^+_j=w^-_j\quad \textrm{on}\quad \p D,\\
&(\p_\nu + iA^+\cdot\nu) w^+_j=b(\p_\nu+iA^-\cdot \nu)(a^{-1}w^-_j) 
+  ca^{-1}w^-_j\quad \textrm{on}\quad \p D.
\end{aligned}
\end{equation}
Since $w^+_j$ vanishes on the non-empty open set $U_j\subset \Omega_j\setminus\overline{D}$ and $\Omega_j\setminus\overline{D}$ is connected, by the classical unique continuation principle, applied to the first equation in \eqref{eq_sets_3}, we get that $w^+_j$ vanishes in  $\Omega_j\setminus\overline{D}$, see \cite[Chapter 17]{Horm_book_3}.  The transmission conditions in \eqref{eq_sets_3} yield that $w^-_j|_{\p D}=\p_\nu w^-_j|_{\p D}=0$, and therefore, by unique continuation from Lipschitz boundary, applied to the second equation in \eqref{eq_sets_3}, we conclude that $w^-_j=0$ in $D$, see Proposition \ref{prop_UC} in Appendix \ref{ap_UC}. This contradicts the fact that $w_j\not\equiv 0$ in $\Omega_j$, and hence, $\tilde w_1,\dots,\tilde w_m$ are linearly independent.  

Let $\{u_k\}_{k=1}^\infty$ be an orthonormal basis of eigenfunctions of the operator $Q_{\tilde \Omega}$ in $\tilde \Omega$, corresponding to the eigenvalues in \eqref{eq_sets_4}, i.e. $Q_{\tilde \Omega}u_k=\lambda_k(\tilde\Omega) u_k$ in $\tilde \Omega$. We have, when $v\in H^1_0(\tilde \Omega)$, 
\begin{equation}
\label{eq_sets_3_new}
\Phi_{\tilde \Omega}(v,v)=\sum_{k=1}^\infty \lambda_k(\tilde\Omega) |(v,u_k)_{L^2(\tilde \Omega)}|^2. 
\end{equation}
As $\tilde w_1,\dots,\tilde w_m$ are linearly independent, there is 
\[
v=\sum_{j=1}^m c_j\tilde w_j\in H^1_0(\tilde \Omega)
\]
such that $(v,u_k)_{L^2(\tilde \Omega)}=0$ for all $k=1,\dots, m-1$, and $\|v\|_{L^2(\tilde \Omega)}=1$.  Thus, it follows from \eqref{eq_sets_3_new} that 
\begin{equation}
\label{eq_sets_5}
\Phi_{\tilde \Omega}(v,v)\ge \lambda_m(\tilde \Omega). 
\end{equation}
 On the other hand,  \eqref{eq_sets_2} implies that 
 \[
 \Phi_{\tilde \Omega}(\tilde w_j,\tilde w_k)=\lambda_n(\tilde \Omega)(\tilde w_j,\tilde w_k)_{L^2(\tilde \Omega)}, 
 \]
 and therefore, 
 \begin{equation}
 \label{eq_sets_6}
 \Phi_{\tilde \Omega}(v,v)=\sum_{j,k=1}^m c_j\overline{c_k}\Phi_{\tilde \Omega}(\tilde w_j, \tilde w_k)=\lambda_n(\tilde \Omega)\|v\|_{L^2(\tilde \Omega)}^2=\lambda_n(\tilde \Omega).
 \end{equation}
 It follows from \eqref{eq_sets_5} and \eqref{eq_sets_6} that $\lambda_n(\tilde \Omega)\ge \lambda_m(\tilde \Omega)$, which is a contradiction. The proof is complete. 
 \end{proof}

 \begin{prop}
 \label{prop_zero_sequence_set}
 Let $D\subset\subset \Omega\subset \tilde \Omega$,  $\tilde \Omega\setminus\overline{\Omega}\ne \emptyset$, and let  $\tilde \Omega\setminus\overline{D}$ and $\Omega\setminus\overline{D}$ be connected.
Then for any sequence of open sets $\Omega_j$, $j=1,2,\dots$,  such that 
 \[
 \Omega=\Omega_1\subset\Omega_2\subset \Omega_3\subset\dots\subset \tilde\Omega,
 \]
  $\tilde\Omega\setminus\overline{\Omega_j}\ne \emptyset$,  $\Omega_{j+1}\setminus\overline{\Omega_j}\ne\emptyset$, and $\Omega_j\setminus\overline{D}$ is connected, $j=1,2,\dots$, there exists $j_0\in \N$ so that zero is not in the spectrum of the operator $Q_{\Omega_j}$ for all $j\ge j_0$. 

 \end{prop}
 
 \begin{proof}
 
 In order to prove this result we shall follow the argument of  \cite[Lemma 3.2]{Stefanov_1990}.  Assume the contrary, i.e. for any $j\in\N$, there exists $l_j\in \N$ so that $l_j\ge j$ and zero is an eigenvalue of the operator $Q_{\Omega_{l_j}}$. Let $k_{l_j}\in\N$ be such that $\lambda_{k_{l_j}}(\Omega_{l_j})=0$, $j=1,2,\dots$.  Without loss of generality we assume  that 
 \[
 l_1<l_2<\dots<l_j\to \infty, \quad j\to \infty. 
 \]
 
 By Proposition \ref{prop_mon_eigen_stric}, we get $0=\lambda_{k_{l_j}}(\Omega_{l_j})>\lambda_{k_{l_j}}(\Omega_{l_{j+1}})$. Since 
\[
\lambda_1(\Omega_{l_{j+1}})\le\dots \le   \lambda_{k_{l_j}}(\Omega_{l_{j+1}})\le \lambda_{k_{l_j}+1}(\Omega_{l_{j+1}})\le \dots \le \lambda_n(\Omega_{l_{j+1}})\to +\infty,
\]  
as $n\to +\infty$, we conclude that $k_{l_j}<k_{l_{j+1}}$, and therefore, $k_{l_j}\to +\infty$ as $j\to +\infty$.  Using again Proposition \ref{prop_mon_eigen_stric}, we have $0=\lambda_{k_{l_j}}(\Omega_{l_j})>\lambda_{k_{l_j}}(\tilde\Omega)\to +\infty$ as $j\to +\infty$.  This contradiction  completes the proof. 
  
\end{proof}

When deriving the Runge type approximation result in Subsection \ref{sec_runge}, we shall need the following consequence of Proposition \ref{prop_zero_sequence_set}. Here $B(x,\varepsilon)$ is an open ball of radius $\varepsilon>0$, centered at $x\in \R^n$. 

\begin{cor}
\label{prop_domain_pert}
Let $\Omega\subset\R^n$ be a bounded open set with Lipschitz boundary, and let $D_j\subset\subset\Omega$ be a bounded open set with Lipschitz boundary such that $\Omega\setminus\overline{D_j}$ is connected, $j=1,2$. 
Let $A^+\in W^{1,\infty}(\R^n,\R^n)$, $q^+\in L^{\infty}(\R^n,\R)$,  $A_j^-\in W^{1,\infty}(D_j,\R^n)$, $q^-_j\in L^\infty(D_j,\R)$, $a_j,b_j\in C^{1,1}(\overline{D_j},\R)$,  $c\in C(\overline{D_j},\R)$. 
Assume that $a_j,b_j>0$  in $\overline{D_j}$ and $a_jb_j=1$ in $\overline{D_j}$.  Then for any $x\in \p \Omega$, and any $\varepsilon>0$, there exists a bounded open set $\tilde \Omega\subset\R^n$ with Lipschitz boundary such that $\tilde \Omega\setminus\overline{D_j}$ is connected, $j=1,2$, $\Omega\subset \tilde \Omega\subset \Omega\cup B(x,\varepsilon)$, and the homogeneous transmission problem on $\tilde \Omega$,
\begin{equation}
\label{eq_sets_7}
\begin{aligned}
&\mathcal{L}_{A^+,q^+}(x,D_x)u^+=0\quad \textrm{in}\quad \tilde \Omega\setminus\overline{D_j},\\
&\mathcal{L}_{A^-_j,q^-_j}(x,D_x)u^-=0\quad \textrm{in}\quad D_j,\\
&u^+=a_j u^-\quad \textrm{on}\quad \p D_j,\\
&(\p_\nu+iA^+\cdot \nu)u^+=b_j(\p_\nu+iA_j^-\cdot\nu)u^-+c_ju^-\quad \textrm{on}\quad \p D_j,\\
&u^+=0\quad \textrm{on}\quad \p \tilde \Omega,
\end{aligned}
\end{equation}
 has only the trivial solution, for $j=1,2$.   

\end{cor}

Combining Proposition \ref{prop_zero_sequence_set} and \cite[Lemma 3.2]{Stefanov_1990}, we also get the following result, which will be useful in Section \ref{sec_scat_prob} when considering the scattering problem.  

\begin{cor}
\label{prop_domain_pert_2}

Let $k>0$ and let $D_j\subset\R^n$ be a bounded open set with Lipschitz boundary such that $\R^n\setminus\overline{D_j}$ is connected, $j=1,2$. 
Let $A^+\in W^{1,\infty}(\R^n,\R^n)$, $q^+\in L^{\infty}(\R^n,\R)$, be compactly supported,    $A_j^-\in W^{1,\infty}(D_j,\R^n)$, $q^-_j\in L^\infty(D_j,\R)$, $a_j,b_j\in C^{1,1}(\overline{D_j},\R)$,  $c\in C(\overline{D_j},\R)$. 
Assume that $a_j,b_j>0$  in $\overline{D_j}$ and $a_jb_j=1$ in $\overline{D_j}$.  Then there exists an open ball $B$ such $D_1,D_2\subset\subset B$, $\supp(A^+), \supp(q^+)\subset B$,  the homogeneous transmission problem \eqref{eq_sets_7} with $\tilde \Omega$ replaced by $B$ and $q^+$, $q^-_j$ replaced by $q^+-k^2$, $q^-_j-k^2$, has only the trivial solution,  $j=1,2$, and $k^2$ is not an eigenvalue of the Dirichlet Laplacian on $B$. 

\end{cor}

When recovering the obstacle in the proof of Theorem \ref{thm_main_sa},  we shall have to consider the following configuration.  Let $\Omega\subset\R^n$ be a bounded domain with connected Lipschitz boundary and let $D_j\subset\subset\Omega$ be a bounded open set with Lipschitz boundary such that $\Omega\setminus\overline{D_j}$ is connected, $j=1,2$. As the boundary of $D_j$ is Lipschitz, we see that $D_j$ has at most finitely many connected components.

Let $D(e)$ be the connected component of $\Omega\setminus(\overline{D_1}\cup \overline{D_2})$, whose boundary intersects $\p \Omega$.  We define $D(i):=\Omega\setminus\overline{D(e)}$.   The number of the connected components  of $D(i)$ does not exceed the sum of the numbers of the connected components of $D_1$ and $D_2$.  Indeed, this follows from the fact that the boundary of each connected component of $\Omega\setminus(\overline{D_1}\cup \overline{D_2})$, different from $D(e)$, intersects $\p(D_1\cup D_2)$.  Let $C_l$, $l=1,2,\dots, N$, be the connected components of $D(i)$, and let $V_l$ be a small connected neighborhood of $\overline{C_l}$ with $C^\infty$ boundary such that $\overline{V_l}\cap \overline{V_k}=\emptyset$, $k\ne l$, and such that for the union $V:=\cup_{l=1}^N V_l$, we have $\Omega\setminus\overline{V}$ is connected. 

Since $\p \Omega$ is connected, we observe that $\R^n\setminus V$ is connected, and therefore, $\R^n\setminus V_l$ is connected for $l=1,2,\dots, N$.  
It follows that the boundary of $\p V_l$ is connected. 

We denote by  $D_j^{(l)}$ the union of the connected components of $D_j$ belonging to $V_l$, $j=1,2$, $l=1,\dots, N$. Notice that $D_j^{(l)}\subset\subset V$ and $\R^n\setminus D_j^{(l)}$ is connected.  It follows therefore that the set $V_l\setminus\overline{D_j^{(l)}}$ is connected, $j=1,2$.

In Proposition \ref{prop_mon_eigen_stric} we have considered the sets $\Omega$ and $D$ such that $\Omega\setminus\overline{D}$ is connected. 
In the situation at hand, the set $V\setminus\overline{D}$ is no longer connected.  Nevertheless, we have the following generalization of Proposition \ref{prop_mon_eigen_stric}, where we let $j=1,2$  be fixed.
\begin{prop}
Let $\tilde V_l\subset\subset \Omega$, $l=1,\dots, N$, be an open connected subset with $C^\infty$ boundary such that $V_l\subset \tilde V_l$, $\tilde V_l\setminus\overline{V_l}\ne\emptyset$, $\tilde V_l\setminus\overline{D_j^{(l)}}$ is connected and $\overline{\tilde V_l}\cap \overline{\tilde V_k}=\emptyset$, $k\ne l$.  Set $\tilde V:=\cup_{l=1}^N \tilde V_l$. Then for the eigenvalues of the operators $Q_V$ and $Q_{\tilde V}$, we  have $\lambda_n(V)>\lambda_n(\tilde V)$, $n=1,2,\dots$. 
\end{prop} 
The proof of this result proceeds similarly to the proof of Proposition  \ref{prop_mon_eigen_stric}, applying it to each pair of the sets  $V_l$ and $\tilde V_l$.

Arguing as in the proof of Proposition \ref{prop_zero_sequence_set}, we finally obtain the following result. 
\begin{prop}
\label{prop_domain_pert_V}
Let $x_l\in \p V_l$, $l=1,\dots, N$, and let $\varepsilon>0$ be small. Then there exists a bounded domain $\tilde V_l$ with $C^\infty$ boundary such that $V_l\subset \tilde V_l\subset V_l\cup B(x_l,\varepsilon)$,  $\Omega\setminus \overline{\tilde V}$ is connected, and the transmission problem \eqref{eq_sets_7} with $\tilde \Omega$ replaced by $\tilde V$ has only the trivial solution, for $j=1,2$. Here we set $\tilde V=\cup_{l=1}^N \tilde V_l$. 
\end{prop}

\section{Inverse transmission problems on bounded domains. Proof of Theorem \ref{thm_main_sa}}
\label{sec_inv_direct_1}
\subsection{Singular solutions of the transmission problem}

\label{sec_fund_sol}

The purpose of this subsection is to construct singular solutions to the transmission problem, using fundamental solutions for the magnetic Schr\"odinger  operator with poles outside the domain. 
Using the elliptic estimates \eqref{eq_energy_estimates_unique} for the solutions of the transmission problem, we shall obtain $H^1$-estimates for these singular solutions.   

Let $\Omega\subset\R^n$, $n\ge 3$, be a bounded domain with connected Lipschitz boundary, and let
$D\subset\subset\Omega$ be a bounded open set with Lipschitz boundary such that $\Omega\setminus \overline{D}$ is connected.  Let $A^\pm\in W^{1,\infty}(D^\pm,\R^n)$, $q^\pm\in L^\infty(D^\pm,\R)$, $a,b\in C^{1,1}(\overline{D},\R)$,  $c\in C(\overline{D},\R)$, and $a,b>0$ on $\overline{D}$. 

Let $V\subset \R^n$ be an open subset with Lipschitz boundary such that
\[
D\subset\!\subset V\subset\!\subset\Omega,
\] 
and such that the following homogeneous transmission problem
\begin{equation}
\label{eq_hom_trans_V}
\begin{aligned}
&\mathcal{L}_{A^+,q^+}u^+=0\quad \textrm{in}\quad V\setminus\overline{D},\\
&\mathcal{L}_{A^-,q^-}u^-=0\quad \textrm{in}\quad D,\\
&u^+=au^-\quad \textrm{on}\quad \p D,\\
(&\p_\nu + iA^+\cdot\nu) u^+=b(\p_\nu+iA^-\cdot \nu)u^-+cu^- \quad \textrm{on}\quad \p D,\\
&u^+=0\quad \textrm{on}\quad \p V,
\end{aligned}
\end{equation}
has only the trivial solution.  Denote by 
$G(x,y)$ a fundamental solution of the operator $\mathcal{L}_{A^+,q^+}$, i.e. 
\[
\mathcal{L}_{A^+,q^+}(x,D_x)G(x,y)=\delta(x-y), \quad x,y\in \Omega. 
\]  
We refer to Appendix \ref{appendix_fundamental_sol} for the existence and  the basic properties of $G$.

Let  $y\in \Omega\setminus \overline{V}$ be fixed.   Then   the estimates \eqref{eq_fund_sol} from Appendix \ref{appendix_fundamental_sol} implies that $G(\cdot,y)\in H^1(V)$.  
By elliptic regularity, $G(\cdot,y)\in H^2_{\textrm{loc}}(V)$. The functions 
\begin{equation}
\label{eq_sing_sol_new}
E^+(\cdot,y)=G(\cdot,y)+E_0^+(\cdot,y),\quad E^-(\cdot,y)=G(\cdot,y)+E_0^-(\cdot,y),
\end{equation}
which will be referred to as singular solutions, solve the transmission problem
\begin{align*}
&\mathcal{L}_{A^+,q^+}(x,D_x)E^+(x,y)=0\quad \textrm{in}\quad V\setminus\overline{D},\\
&\mathcal{L}_{A^-,q^-}(x,D_x)E^-(x,y)=0\quad \textrm{in}\quad D,\\
&E^+(\cdot,y)=aE^-(\cdot,y)\quad \textrm{on}\quad \p D,\\
(&\p_\nu + iA^+\cdot\nu) E^+(\cdot,y)=b(\p_\nu+iA^-\cdot \nu)E^-(\cdot,y)+cE^-(\cdot,y) \quad \textrm{on}\quad \p D,\\
&E^+(\cdot,y)=G(\cdot,y)\quad \textrm{on}\quad \p V,
\end{align*}
when $(E_0^+(\cdot,y), E_0^-(\cdot,y))\in H^1(V\setminus\overline{D})\times H^1(D)$ solve  the following transmission problem,
\begin{equation}
\label{eq_transmission_E_0}
\begin{aligned}
&\mathcal{L}_{A^+,q^+}(x,D_x)E_0^+(x,y)=0\quad \textrm{in}\quad V\setminus\overline{D},\\
&\mathcal{L}_{A^-,q^-}(x,D_x)E_0^-(x,y)=f^-\quad \textrm{in}\quad D,\\
&E_0^+(\cdot,y)=aE_0^-(\cdot,y)+(a-1)G(\cdot,y)\quad \textrm{on}\quad \p D,\\
(&\p_\nu + iA^+\cdot\nu) E_0^+(\cdot,y)=b(\p_\nu+iA^-\cdot \nu)E_0^-(\cdot,y)+cE_0^-(\cdot,y)+g_1 \quad \textrm{on}\quad \p D,\\
&E^+_0(\cdot,y)=0\quad \textrm{on}\quad \p V,
\end{aligned}
\end{equation}
where 
\begin{align*}
f^-= & 2i(A^--A^+)\cdot\nabla G(\cdot,y)\\
&+(i\nabla\cdot (A^--A^+)+(A^+)^2-(A^-)^2+q^+-q^-)G(\cdot,y)\in L^2(D),\\
g_1=&(b-1)\p_\nu G(\cdot,y)+(i\nu\cdot (bA^--A^+)+c)G(\cdot,y)\in H^{-1/2}(\p D).
\end{align*}
Since by the choice of $V$ the transmission problem \eqref{eq_transmission_E_0} is uniquely solvable, the estimates  \eqref{eq_energy_estimates_unique} implies that 
\begin{equation}
\label{eq_energy_estimates_unique_E_0}
\begin{aligned}
\|E_0^+(\cdot,y)\|_{H^1(V\setminus\overline{D})}+\|E_0^-(\cdot,y)\|_{H^1(D)}\le C(\|f^-\|_{\tilde H^{-1}(D)}\\+\|(a-1)G(\cdot,y)\|_{H^{1/2}(\p D)}
+ \|g_1\|_{H^{-1/2}(\p D)}).
\end{aligned}
\end{equation}
Let us estimate all the terms in the right hand side of \eqref{eq_energy_estimates_unique_E_0}. 
First we have
\begin{equation}
\label{eq_energy_estimates_unique_E_0_1}
\|f^-\|_{\tilde H^{-1}(D)}\le C \|f^-\|_{L^2(D)}\le C\| G(\cdot,y)\|_{H^1(D)}.
\end{equation}
By Proposition \ref{prop_sob_mult} and the trace theorem, we get
\begin{equation}
\label{eq_energy_estimates_unique_E_0_2}
\|(a-1)G(\cdot,y)\|_{H^{1/2}(\p D)}\le C\|G(\cdot,y)\|_{H^{1/2}(\p D)} \le C\|G(\cdot,y)\|_{H^{1}(D)}. 
\end{equation}
Since 
\[
\Delta_x G(x,y)=-2iA^+\cdot\nabla_x  G(x,y)+(-i(\nabla\cdot A^+)+(A^+)^2+q^+)G(x,y)\in L^2(D),
\]
it follows from \eqref{eq_trace_cont} that
\begin{equation}
\label{eq_3_2}
\begin{aligned}
\|\p_\nu G(\cdot,y)\|_{H^{-1/2}(\p D)}\le C(\|\Delta_x G(x,y)\|_{L^2(D)}+ \|G(\cdot,y)\|_{H^1(D)})\le C \|G(\cdot,y)\|_{H^1(D)}.
\end{aligned}
\end{equation}
This together with the fact that 
$b\in C^{1,1}(\overline{D})$ implies that    
\begin{equation}
\label{eq_energy_estimates_unique_E_0_3}
\begin{aligned}
\|g_1\|_{ H^{-1/2}(\p D)}&\le C(\|\p_\nu G(\cdot,y)\|_{ H^{-1/2}(\p D)}+\|(i\nu\cdot (bA^--A^+)+c)G(\cdot,y)\|_{L^2(\p D)})\\
&\le C(\|\p_\nu G(\cdot,y)\|_{ H^{-1/2}(\p D)}+ \|G(\cdot,y)\|_{H^{1/2}(\p D)})\le C\|G(\cdot,y)\|_{H^{1}(D)}.
\end{aligned}
\end{equation}
Hence, it follows from \eqref{eq_energy_estimates_unique_E_0} with the help of the estimates \eqref{eq_energy_estimates_unique_E_0_1} -- \eqref{eq_energy_estimates_unique_E_0_3} that 
\begin{equation}
\label{eq_3_1}
\|E_0^+(\cdot,y)\|_{H^1(V\setminus\overline{D})}+\|E_0^-(\cdot,y)\|_{H^1(D)}\le C\|G(\cdot,y)\|_{H^{1}(D)}. 
\end{equation}
We conclude that the behavior of the singular solutions $E^+(\cdot,y)$ and $E^-(\cdot,y)$, introduced in \eqref{eq_sing_sol_new},  is essentially controlled by the behavior of the fundamental solution $G(\cdot,y)$ of the magnetic 
Schr\"odinger operator, as the pole $y$ is close to the obstacle.

\subsection{Runge type approximation result}

\label{sec_runge}

Let $\Omega\subset\R^n$, $n\ge 3$, be a bounded domain with connected Lipschitz boundary, and 
$D\subset\subset\Omega$ be a bounded open set with Lipschitz boundary such that $\Omega\setminus \overline{D}$ is connected. Let $\gamma\subset \p \Omega$ be an open nonempty subset of the boundary of $\Omega$. 
For $(u^+,u^-)\in H^1(\Omega\setminus\overline D)\times H^1(D)$, consider the following transmission problem,
\begin{equation}
\label{eq_trans_Omega}
\begin{aligned}
&\mathcal{L}_{A^+,q^+}u^+=0\quad \textrm{in}\quad \Omega\setminus\overline{D},\\
&\mathcal{L}_{A^-,q^-}u^-=0\quad \textrm{in}\quad D,\\
&u^+=au^-\quad \textrm{on}\quad \p D,\\
(&\p_\nu + iA^+\cdot\nu) u^+=b(\p_\nu+iA^-\cdot \nu)u^-+cu^- \quad \textrm{on}\quad \p D,\\
\end{aligned}
\end{equation}
and set
\begin{align*}
W(\Omega)=\{(u^+,u^-)\in H^1(\Omega\setminus\overline D)\times H^1(D) :  \ & (u^+,u^-)\textrm{ satisfies \eqref{eq_trans_Omega}},\\
&\supp(u^+|_{\p \Omega})\subset \gamma\}.
\end{align*}

Let $V\subset \R^n$ be an open set with Lipschitz boundary such that
\[
D\subset\!\subset V\subset\!\subset\Omega,
\] 
and $\Omega\setminus\overline{V}$ is connected. 
 For $(w^+,w^-)\in H^1(V\setminus\overline D)\times H^1(D)$, consider the following transmission problem,
\begin{equation}
\label{eq_trans_V}
\begin{aligned}
&\mathcal{L}_{A^+,q^+}w^+=0\quad \textrm{in}\quad V\setminus\overline{D},\\
&\mathcal{L}_{A^-,q^-}w^-=0\quad \textrm{in}\quad D,\\
&w^+=aw^-\quad \textrm{on}\quad \p D,\\
(&\p_\nu + iA^+\cdot\nu) w^+=b(\p_\nu+iA^-\cdot \nu)w^-+cw^- \quad \textrm{on}\quad \p D,\\
\end{aligned}
\end{equation}
and set
\begin{align*}
W(V)=\{(w^+,w^-)\in H^1(V\setminus\overline D)\times H^1(D) :  \ & (w^+,w^-)\textrm{ satisfies \eqref{eq_trans_V}}\}.
\end{align*}

Let 
\[
W(D)=\{w^-\in H^1(D): \mathcal{L}_{A^-,q^-}w^-=0\quad \textrm{on}\quad D\}. 
\]

\begin{lem}
\label{lem_density}
Assume that $ab=1$ on $\overline{D}$. 
\begin{itemize}
\item[(i)]
The set $W(\Omega)$ is dense in the set $W(V)$ in the $H^1(V\setminus\overline{D})\times H^1(D)$--topology. 

\item[(ii)] The set $W(\Omega)$ is dense in the set $W(D)$ in the $H^1(D)$--topology. 

\end{itemize}
\end{lem}

\begin{proof}

(i) First notice that the dual space of 
$
H^1(V\setminus\overline{D})\times H^1(D)$ �is the space $\tilde H^{-1}(V\setminus\overline{D})\times \tilde H^{-1}(D)$.
Then by the Hahn--Banach theorem, we need to show that for any $(f^+,f^-)\in \tilde H^{-1}(V\setminus\overline{D})\times \tilde H^{-1}(D)$ such that 
\begin{equation}
\label{eq_Hahn-Banach_1}
(f^+,u^+)_{\tilde H^{-1}(V\setminus\overline D),H^1(V\setminus\overline D)}+(f^-,u^-)_{\tilde H^{-1}(D),H^1(D)}=0,
\end{equation}
for any $(u^+,u^-)\in W(\Omega)$, we have
\begin{equation}
\label{eq_Hahn-Banach_2}
(f^+,w^+)_{\tilde H^{-1}(V\setminus\overline D),H^1(V\setminus\overline D)}+(f^-,w^-)_{\tilde H^{-1}(D),H^1(D)}=0,
\end{equation}
for any $(w^+,w^-)\in W(V)$.  

Let us extend $A^+\in W^{1,\infty}(\Omega\setminus \overline{D},\R^n)$ and $q^+\in L^\infty(\Omega\setminus \overline{D},\R)$ to the whole of $\R^n$ so that the extensions, which we denote by the same letters, satisfy  $A^+\in W^{1,\infty}(\R^n,\R^n)$ and $q^+\in L^\infty(\R^n,\R)$. In view of Corollary \ref{prop_domain_pert},  there is a bounded domain $\tilde \Omega\supset\Omega$ with connected Lipschitz boundary such that the sets $\tilde \Omega\setminus\overline{D}$, $\tilde \Omega\setminus\overline{V}$ are connected, $\p \Omega\setminus \gamma\subset \p \tilde \Omega$, and the homogeneous transmission problem \eqref{eq_transmission_hom} in $\tilde \Omega$ instead of $\Omega$ has only the trivial solution. 
Since $f^+\in \tilde H^{-1}(V\setminus\overline{D})$, we conclude that $f^+\in \tilde H^{-1}(\tilde \Omega\setminus \overline D)$, and $\supp (f^+)\subset \overline{V}\setminus{D}$.  Then by the choice of $\tilde \Omega$,  the following adjoint problem 
\begin{equation}
\label{eq_adjoint_density}
\begin{aligned}
&\mathcal{L}_{A^+,q^+}v^+=f^+\quad \textrm{in}\quad \tilde \Omega\setminus \overline D,\\
&\mathcal{L}_{A^-,q^-}v^-=f^-\quad \textrm{in}\quad D,\\
& v^+=b^{-1}v^-\quad \textrm{on}\quad \p D,\\
&(\p_\nu+i A^+ \cdot \nu)v^+=a^{-1}(\p_\nu+i A^-\cdot \nu)v^-+c a^{-1}b^{-1}v^-\quad \textrm{on}\quad \p D,\\
&v^+=0\quad \textrm{on}\quad \p \tilde\Omega,
\end{aligned}
\end{equation}
has a unique solution $(v^+,v^-)\in H^1(\tilde \Omega\setminus \overline D)\times H^1(D)$.  

Let $\tilde \gamma=\p \tilde \Omega\setminus\overline{\Omega}$ and let $(u^+,u^-)\in H^1(\tilde \Omega\setminus \overline D)\times H^1(D)$ satisfy the transmission problem \eqref{eq_trans_Omega} in $\tilde \Omega$ instead of $\Omega$ and $\supp(u^+|_{\p \tilde \Omega})\subset\tilde \gamma$. Thus, $(u^+|_{\Omega\setminus\overline D},u^-)\in W(\Omega)$.

By the second Green formula \eqref{eq_first_green_adjoint} for $u^+$ and $v^+$, we get 
\begin{align*}
&(\mathcal{L}_{A^+,q^+}u^+,v^+)_{L^2(\tilde \Omega\setminus\overline D)} -((\p_\nu+iA^+\cdot\nu)u^+, v^+)_{H^{-1/2}(\p D),H^{1/2}(\p D)}\\
&+ 
((\p_\nu+iA^+\cdot\nu)u^+, v^+)_{H^{-1/2}(\p \tilde \Omega),H^{1/2}(\p \tilde \Omega)}\\
&=
(u^+,f^+)_{H^1(\tilde \Omega\setminus \overline D), \tilde H^{-1}(\tilde \Omega\setminus \overline D)}-(u^+, (\p_\nu+i A^+\cdot\nu)v^+)_{H^{1/2}(\p D), H^{-1/2}(\p D)}\\
&+ (u^+, (\p_\nu+i A^+\cdot\nu)v^+)_{H^{1/2}(\p \tilde \Omega), H^{-1/2}(\p \tilde \Omega)}.
\end{align*}
Thus, using the fact that $\mathcal{L}_{A^+,q^+}u^+=0$ in $\tilde \Omega\setminus\overline{D}$ and  $v^+=0$ on $\p \tilde \Omega$, we get 
\begin{equation}
\label{eq_lem_density_1}
\begin{aligned}
(u^+, \p_\nu v^+)_{H^{1/2}(\p \tilde \Omega), H^{-1/2}(\p \tilde \Omega)}=-((\p_\nu+iA^+\cdot\nu)u^+, v^+)_{H^{-1/2}(\p D),H^{1/2}(\p D)}\\
+ (u^+, (\p_\nu+i A^+\cdot\nu)v^+)_{H^{1/2}(\p D), H^{-1/2}(\p D)}-(u^+,f^+)_{H^1(V\setminus \overline D), \tilde H^{-1}(V\setminus \overline D)}.
\end{aligned}
\end{equation}
By the second Green formula \eqref{eq_first_green_adjoint} for $u^-$ and $v^-$, we obtain that 
\begin{align*}
&(\mathcal{L}_{A^-,q^-}u^-,v^-)_{L^2(D)} +((\p_\nu+iA^-\cdot\nu)u^-, v^-)_{H^{-1/2}(\p D),H^{1/2}(\p D)}\\
&=
(u^-,f^-)_{H^1(D), \tilde H^{-1}(D)}+(u^-, (\p_\nu+i A^-\cdot\nu)v^-)_{H^{1/2}(\p D), H^{-1/2}(\p D)}.
\end{align*}
Since $\mathcal{L}_{A^-,q^-}u^-=0$ on $D$, we have
\begin{equation}
\label{eq_lem_density_2}
\begin{aligned}
0=-(u^-,f^-)_{H^1(D), \tilde H^{-1}(D)}+((\p_\nu+iA^-\cdot\nu)u^-, v^-)_{H^{-1/2}(\p D),H^{1/2}(\p D)}\\
-(u^-, (\p_\nu+i A^-\cdot\nu)v^-)_{H^{1/2}(\p D), H^{-1/2}(\p D)}.
\end{aligned}
\end{equation}
Adding \eqref{eq_lem_density_1} and \eqref{eq_lem_density_2}, and using 
\eqref{eq_Hahn-Banach_1} together with the transmission conditions in \eqref{eq_trans_Omega} and \eqref{eq_adjoint_density}, we get
\begin{align*}
(u^+, \p_\nu v^+&)_{H^{1/2}(\p \tilde \Omega), H^{-1/2}(\p \tilde \Omega)}=-((\p_\nu+iA^+\cdot\nu)u^+, v^+)_{H^{-1/2}(\p D),H^{1/2}(\p D)}\\
&+((\p_\nu+iA^-\cdot\nu)u^-, bv^+)_{H^{-1/2}(\p D),H^{1/2}(\p D)}\\ 
&+(au^-, (\p_\nu+i A^+\cdot\nu)v^+)_{H^{1/2}(\p D), H^{-1/2}(\p D)}\\
&-(u^-, (\p_\nu+i A^-\cdot\nu)v^-)_{H^{1/2}(\p D), H^{-1/2}(\p D)}\\
&=-(cu^-,v^+)_{L^2(\p D)}+(u^-,cv^+)_{L^2(\p D)}=0.
\end{align*}
Since $u^+|_{\p\tilde \Omega}$ can be an arbitrary smooth function with $\supp(u^+|_{\p\tilde \Omega})\subset \tilde \gamma$, we conclude that
$\p_\nu v^+=0$ on $\tilde \gamma$.  Thus, $v^+$ satisfies  $\mathcal{L}_{A^+, q^+}v^+=0$ on $\tilde \Omega\setminus\overline{V}$, and 
 $v^+=0$,  $\p_\nu v^+=0$ on $\tilde \gamma$.  As $A^+\in W^{1,\infty}(\tilde \Omega)$ and $q^+\in L^\infty(\tilde \Omega)$, and $\tilde \Omega\setminus\overline{V}$ is connected, by unique continuation from Lipschitz part of the boundary  we get $v^+=0$ on $\tilde \Omega\setminus\overline{V}$, see Proposition \ref{prop_UC} in Appendix \ref{ap_UC}. 
Since $v^+\in H^1(\tilde \Omega\setminus\overline{D})$ and $\Delta v^+\in L^2(\tilde \Omega\setminus \overline{V})$, we have 
\begin{equation}
\label{eq_v-normal}
v^+=0, \quad\textrm{and}\quad \p_{\nu} v^+=0\quad \textrm{on}\quad \p V. 
\end{equation}

Let $(w^+,w^-)\in W(V)$. Then by the second Green formula \eqref{eq_first_green_adjoint} on  $V\setminus\overline D$ for $w^+$ and $v^+$, we get 
\begin{align*}
&(\mathcal{L}_{A^+,q^+}w^+,v^+)_{L^2(V\setminus\overline D)} -((\p_\nu+iA^+\cdot\nu)w^+, v^+)_{H^{-1/2}(\p D),H^{1/2}(\p D)}\\
&+ 
((\p_\nu+iA^+\cdot\nu)w^+, v^+)_{H^{-1/2}(\p V),H^{1/2}(\p V)}\\
&=
(w^+,f^+)_{H^1(V\setminus \overline D), \tilde H^{-1}(V\setminus \overline D)}-(w^+, (\p_\nu+i A^+\cdot\nu)v^+)_{H^{1/2}(\p D), H^{-1/2}(\p D)}\\
&+ (w^+, (\p_\nu+i A^+ \cdot\nu)v^+)_{H^{1/2}(\p V), H^{-1/2}(\p  V)}.
\end{align*}
Since $\mathcal{L}_{A^+,q^+}w^+=0$ on $V\setminus\overline D$ and \eqref{eq_v-normal}, we get
\begin{equation}
\label{eq_w_f_1}
\begin{aligned}
(w^+,f^+)_{H^1(V\setminus \overline D), \tilde H^{-1}(V\setminus \overline D)}=&-((\p_\nu+iA^+\cdot\nu)w^+, v^+)_{H^{-1/2}(\p D),H^{1/2}(\p D)}\\
&+(w^+, (\p_\nu+i A^+ \cdot\nu)v^+)_{H^{1/2}(\p D), H^{-1/2}(\p D)}.
\end{aligned}
\end{equation}

By the second Green formula \eqref{eq_first_green_adjoint} for $w^-$ and $v^-$, we obtain that 
\begin{equation}
\label{eq_w_f_3}
\begin{aligned}
&(\mathcal{L}_{A^-,q^-}w^-,v^-)_{L^2(D)} +((\p_\nu+iA^-\cdot\nu)w^-, v^-)_{H^{-1/2}(\p D),H^{1/2}(\p D)}\\
&=
(w^-,f^-)_{H^1(D), \tilde H^{-1}(D)}+(w^-, (\p_\nu+i A^- \cdot\nu)v^-)_{H^{1/2}(\p D), H^{-1/2}(\p D)}.
\end{aligned}
\end{equation}
As $\mathcal{L}_{A^-,q^-}w^-=0$ on $D$, we have
\begin{equation}
\label{eq_w_f_2}
\begin{aligned}
(w^-,f^-)_{H^1(D), \tilde H^{-1}(D)}=&((\p_\nu+iA^-\cdot\nu)w^-, v^-)_{H^{-1/2}(\p D),H^{1/2}(\p D)}\\
&-(w^-, (\p_\nu+i A^- \cdot\nu)v^-)_{H^{1/2}(\p D), H^{-1/2}(\p D)}.
\end{aligned}
\end{equation}

Adding \eqref{eq_w_f_1} and \eqref{eq_w_f_2} and using the transmission conditions in \eqref{eq_trans_V}
and \eqref{eq_adjoint_density}, we get \eqref{eq_Hahn-Banach_2}.  This proves (i).

(ii). By the Hahn--Banach theorem, we need to show that for any $f^-\in \tilde H^{-1}(D)$ such that 
\[
(f^-,u^-)_{\tilde H^{-1}(D),H^1(D)}=0,
\]
for any $u^-\in W(\Omega)|_{D}$, we have
\begin{equation}
\label{eq_w_f_5}
(f^-,w^-)_{\tilde H^{-1}(D),H^1(D)}=0,
\end{equation}
for any $w^-\in W(D)$.  Let $(v^+,v^-)\in H^1(\tilde \Omega\setminus\overline D)\times H^1(D)$ be a unique solution to the adjoint transmission problem \eqref{eq_adjoint_density} on $\tilde \Omega$ with $f^+=0$. 
 Let $(u^+,u^-)\in H^1(\tilde \Omega\setminus \overline D)\times H^1(D)$ satisfy the transmission problem \eqref{eq_trans_Omega} in $\tilde \Omega$ instead of $\Omega$ and $\supp(u^+|_{\p \tilde \Omega})\subset\tilde \gamma$. Thus, $u^-\in W(\Omega)|_{D}$. In the same way as when deriving \eqref{eq_v-normal}, we get
 \[
v^+=0, \quad\textrm{and}\quad \p_{\nu} v^+=0\quad \textrm{on}\quad \p D. 
\]
This together with the transmission conditions in \eqref{eq_adjoint_density} implies that 
 \begin{equation}
\label{eq_w_f_4}
v^-=0, \quad\textrm{and}\quad \p_{\nu} v^-=0\quad \textrm{on}\quad \p D. 
\end{equation}
By the second Green formula \eqref{eq_first_green_adjoint} for $w^-$ and $v^-$, we obtain \eqref{eq_w_f_3}. It follows from \eqref{eq_w_f_3} with help of \eqref{eq_w_f_4} and the fact that $\mathcal{L}_{A^-,q^-}w^-=0$ in $D$ that 
\eqref{eq_w_f_5} is valid. The proof is complete.  
\end{proof}

\subsection{Determination of the obstacle} 

\label{sec_obstacle}

Assume that $D_1\ne D_2$. 
Let $(u_1^+,u_1^-)\in H^1(\Omega\setminus\overline{D_1})\times H^1(D_1)$ satisfy the following transmission problem with $j=1$,
\begin{equation}
\label{eq_5_1}
\begin{aligned}
&\mathcal{L}_{A^+,q^+}u_j^+=0\quad \textrm{in}\quad \Omega\setminus\overline{D_j},\\
&\mathcal{L}_{A_j^-,q_j^-}u_j^-=0\quad \textrm{in}\quad D_j,\\
&u_j^+=a_ju_j^-\quad \textrm{on}\quad \p D_j,\\
(&\p_\nu + iA^+\cdot\nu) u_j^+=b_j(\p_\nu+iA_j^-\cdot \nu)u_j^-+c_ju_j^- \quad \textrm{on}\quad \p D_j,
\end{aligned}
\end{equation}
and be such that $\supp(u_1^+|_{\p \Omega})\subset \gamma$.
Since  
\begin{equation}
\label{eq_5_0}
\mathcal{C}_\gamma(A^+,q^+, A_1^-,q_1^-,a_1,b_1,c_1; D_1)=\mathcal{C}_\gamma(A^+,q^+,A_2^-,q_2^-,a_2,b_2,c_2; D_2),
\end{equation}
there is $(u_2^+,u_2^-)\in H^1(\Omega\setminus\overline{D_2})\times H^1(D_2)$, which satisfies the transmission problem \eqref{eq_5_1} with $j=2$, and such that 
\begin{equation}
\label{eq_5_2}
\begin{aligned}
&u_1^+=u_2^+\quad \textrm{on}\quad \p \Omega,\quad \supp(u_2^+|_{\p \Omega})\subset \gamma,\\
&(\p_\nu+iA^+\cdot\nu) u_1^+=(\p_\nu+iA^+\cdot\nu) u_2^+\quad \textrm{on}\quad \gamma.
\end{aligned}
\end{equation}

Let $D(e)$ be the connected component of $\Omega\setminus(\overline{D_1}\cup\overline{D_2})$, whose boundary intersects $\p \Omega$. Then
\eqref{eq_5_1} implies that
\[
\mathcal{L}_{A^+,q^+}(u_1^+-u_2^+)=0\quad \textrm{on}\quad D(e).
\]
It follows from \eqref{eq_5_2} that 
\[
u_1^+=u_2^+, \quad \p_\nu u_1^+=\p_\nu u_2^+\quad \textrm{on}\quad \gamma.
\]
As $D(e)$ is connected, by unique continuation from a part of Lipschitz boundary, we get 
\begin{equation}
\label{eq_5_3}
u_1^+=u_2^+\quad\textrm{on}\quad  D(e).
\end{equation}

Now due to the connectedness of $\Omega\setminus \overline{D_1}$ and  $\Omega\setminus \overline{D_2}$, there is a point $x_0\in \p D_2$ such that $x_0\notin \overline{D_1}$ and $x_0\in \overline{D(e)}$. As $\p D_2$ is Lipschitz, we can assume that there is the unit outer normal $\nu(x_0)$ at the point $x_0$ to $\p D_2$.  
By the hypothesis of Theorem \ref{thm_main_sa}
$
a_2(x_0)\ne b_2(x_0)$,
and therefore, without loss of generality  we may assume that $b_2(x_0)-a_2(x_0)>0$. Thus, 
there exists an open ball $B$, centered at $x_0$, such that  
\[
b_2- a_2>0\quad \textrm{on}\quad \overline{B\cap D_2},
\]
and 
$\overline{B}\subset \Omega\setminus\overline{D_1}$. 
Define
\[
x_\delta=x_0+\delta\nu(x_0),
\]
for $\delta>0$ small so that $x_\delta\in B$.

Let $D(i)=\Omega\setminus\overline{D(e)}$. Then $D(i)$ has a finite number of connected components, denoted by $C_l$, $l=1,\dots,N$.  Let $V_l(\delta)$ be a small connected  neighborhood of $\overline{C_l}$ with $C^\infty$ boundary such that $\overline{V_l(\delta)}\cap \overline{V_k(\delta)}=\emptyset$, $l\ne k$, and such that for the union $V(\delta):=\cup_{l=1}^N V_l(\delta)$, we have 
$x_\delta\notin \overline{V(\delta)}$, $\Omega\setminus\overline{V(\delta)}$ is connected, and the homogeneous transmission problem,
\begin{align*}
&\mathcal{L}_{A^+,q^+}w_j^+=0\quad \textrm{in}\quad V(\delta)\setminus\overline{D_j},\\
&\mathcal{L}_{A_j^-,q_j^-}w_j^-=0\quad \textrm{in}\quad D_j,\\
&w_j^+=a_jw_j^-\quad \textrm{on}\quad \p D_j,\\
(&\p_\nu + iA_j^+\cdot\nu) w_j^+=b_j(\p_\nu+iA_j^-\cdot \nu)w_j^-+cw_j^- \quad \textrm{on}\quad \p D_j,\\
&w_j^+|_{\p V(\delta)}=0,
\end{align*}
 has only the trivial solution, $j=1,2$. The existence of the set $V(\delta)$ follows from Proposition \ref{prop_domain_pert_V}.  See also Figure 1 for the illustration of the configuration described above.

 \begin{figure}[htbp]
\begin{center}
\includegraphics[scale=0.8, viewport=80 300 628 600]{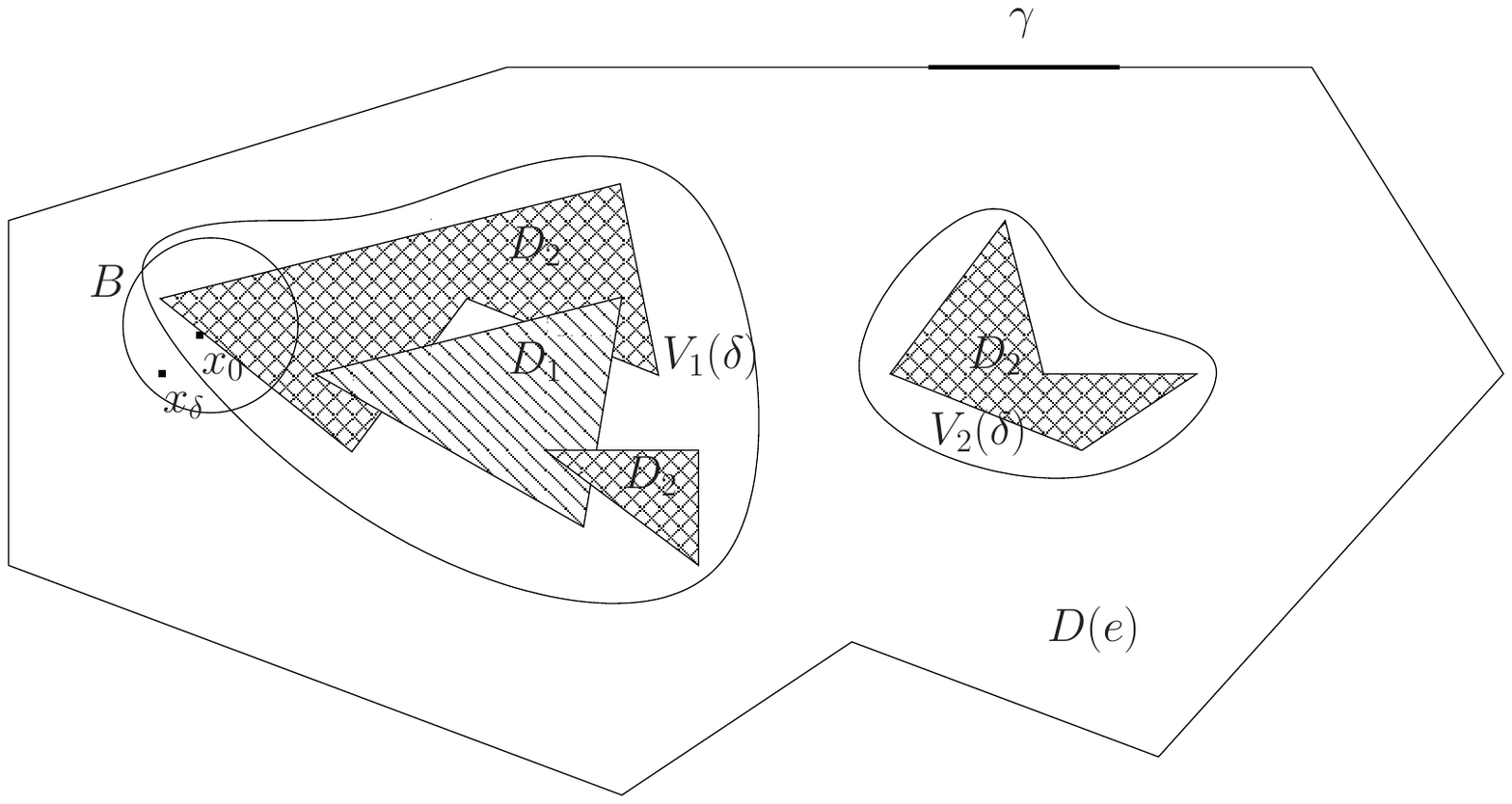}
\caption{A schematic illustration of the reconstruction of the obstacle in the proof of Theorem \ref{thm_main_sa}. Here $V(\delta)=V_1(\delta)\cup V_2(\delta)$.}
\end{center}
\end{figure}

As $u_1^+,u_2^+\in H^1(D(e))$ and $\Delta u_j^+\in L^2(D(e))$, it follows from \eqref{eq_5_3} that
 \begin{equation}
 \label{eq_5_4}
 u_1^+=u_2^+, \quad \p_\nu u_1^+=\p_\nu u_2^+\quad \textrm{on}\quad \p V(\delta).
 \end{equation}

Let $v\in H^1(V(\delta))$ be a solution to the equation
\begin{equation}
\label{eq_5_4_1}
\mathcal{L}_{A^+, q^+}v=0\quad \textrm{in}\quad V(\delta). 
\end{equation}
Since $u_1^+\in H^1(\Omega\setminus\overline{D_1})$ satisfies the equation 
\[
\mathcal{L}_{A^+,q^+}u_1^+=0\quad \textrm{in}\quad V(\delta)\setminus\overline{D_1},
\]
by the second Green formular \eqref{eq_second_green}, we get
\begin{equation}
 \label{eq_5_5}
\begin{aligned}
&0=(\mathcal{L}_{A^+,q^+}u_1^+, v)_{L^2(V(\delta)\setminus\overline{D_1})}-(u_1^+,\mathcal{L}_{A^+,q^+}v)_{L^2(V(\delta)\setminus\overline{D_1})}\\
&=(u_1^+,(\p_\nu+i A^+\cdot\nu)v)_{(H^{1/2},H^{-1/2})(\p V(\delta))}-((\p_\nu+iA^+\cdot\nu)u_1^+,v)_{(H^{-1/2}, H^{1/2})(\p V(\delta))}\\
&-(u_1^+,(\p_\nu+i A^+\cdot\nu)v)_{(H^{1/2}, H^{-1/2})(\p D_1)}+((\p_\nu+iA^+\cdot\nu)u_1^+,v)_{(H^{-1/2}, H^{1/2})(\p D_1)}.
\end{aligned}
\end{equation}
Here $\p_\nu$ is the exterior  normal derivative to $D_1$ and $V(\delta)$. 
As $u_2^+\in H^1(\Omega\setminus\overline{D_2})$ satisfies the equation 
\[
\mathcal{L}_{A^+,q^+}u_2^+=0\quad \textrm{in}\quad V(\delta)\setminus\overline{D_2},
\]
again by the second Green formular \eqref{eq_second_green}, we obtain that
\begin{equation}
 \label{eq_5_6}
\begin{aligned}
&0=(\mathcal{L}_{A^+,q^+}u_2^+, v)_{L^2(V(\delta)\setminus\overline{D_2})}-(u_2^+,\mathcal{L}_{A^+,q^+}v)_{L^2(V(\delta)\setminus\overline{D_2})}\\
&=(u_2^+,(\p_\nu+i A^+\cdot\nu)v)_{(H^{1/2}, H^{-1/2})(\p V(\delta))}-((\p_\nu+iA^+\cdot\nu)u_2^+,v)_{(H^{-1/2}, H^{1/2})(\p V(\delta))}\\
&-(u_2^+,(\p_\nu+i A^+\cdot\nu)v)_{(H^{1/2}, H^{-1/2})(\p D_2)}+((\p_\nu+iA^+\cdot\nu)u_2^+,v)_{(H^{-1/2}, H^{1/2})(\p D_2)}.
\end{aligned}
\end{equation}
Combining \eqref{eq_5_5} and \eqref{eq_5_6} with the help of \eqref{eq_5_4}, we have
\begin{equation}
 \label{eq_5_7}
\begin{aligned}
(&u_1^+,(\p_\nu+i A^+\cdot\nu)v)_{(H^{1/2},H^{-1/2})(\p D_1)}-((\p_\nu+iA^+\cdot\nu)u_1^+,v)_{(H^{-1/2},H^{1/2})(\p D_1)}\\
&=(u_1^+,(\p_\nu+i A^+\cdot\nu)v)_{(H^{1/2}, H^{-1/2})(\p V(\delta))}
-((\p_\nu+iA^+\cdot\nu)u_1^+,v)_{(H^{-1/2}, H^{1/2})(\p V(\delta))}\\
&=
(u_2^+,(\p_\nu+i A^+ \cdot\nu)v)_{(H^{1/2}, H^{-1/2})(\p D_2)}
-((\p_\nu+iA^+\cdot\nu)u_2^+,v)_{(H^{-1/2},H^{1/2})(\p D_2)},
\end{aligned}
\end{equation}
where $(u_j^+, u_j^-)\in H^1(\Omega\setminus\overline{D_j})\times H^1(D_j)$, $j=1,2$, satisfy the transmission problems \eqref{eq_5_1} on $\Omega$ and \eqref{eq_5_2}. 

We now claim that \eqref{eq_5_7} can be extended to all $(u_j^+,u_j^-)\in H^1(V(\delta)\setminus\overline D_j)\times H^1(D_j)$, $j=1,2$, which satisfy the transmission problems
\begin{equation}
 \label{eq_5_7_1}
\begin{aligned}
&\mathcal{L}_{A^+,q^+}u^+_j=0\quad \textrm{in}\quad V(\delta)\setminus\overline{D_j},\\
&\mathcal{L}_{A^-_j,q^-_j}u_j^-=0\quad \textrm{in}\quad D_j,\\
&u_j^+=a_ju_j^-\quad \textrm{on}\quad \p D_j,\\
(&\p_\nu + iA^+\cdot\nu) u_j^+=b_j(\p_\nu+iA_j^-\cdot \nu)u_j^-+c_ju_j^- \quad \textrm{on}\quad \p D_j,
\end{aligned}
\end{equation}
and such that 
\begin{equation}
\label{eq_5_7_2}
u_1^+|_{\p V(\delta)}=u_2^+|_{\p V(\delta)}. 
\end{equation}
Indeed, let $(u_1^+,u_1^-)\in H^1(V(\delta)\setminus\overline D_1)\times H^1(D_1)$ satisfy 
\eqref{eq_5_7_1} with $j=1$. 
Then by Lemma \ref{lem_density} there are $(u_{1k}^+,u_{1k}^-)\in H^1(\Omega\setminus\overline{D_1})\times H^1(D_1)$, $k=1,2,\dots$, which satisfy the transmission problem \eqref{eq_5_1} with $j=1$ on $\Omega$,  $\supp(u_{1k}^+|_{\p \Omega})\subset \gamma$, and 
\begin{equation}
 \label{eq_5_8_-1}
(u_{1k}^+,u_{1k}^-)\to (u_1^+,u_1^-)\quad \textrm{in}\quad H^1(V(\delta)\setminus\overline D_1)\times H^1(D_1),\quad k\to\infty. 
\end{equation}
It follows from \eqref{eq_5_0} that there are $(u_{2k}^+,u_{2k}^-)\in H^1(\Omega\setminus\overline{D_2})\times H^1(D_2)$, $k=1,2,\dots$, which satisfy the transmission problem \eqref{eq_5_1} with $j=2$ on $\Omega$  such that 
\begin{align*}
&u_{1k}^+=u_{2k}^+\quad \textrm{on}\quad \p \Omega,\\
&(\p_\nu+iA^+\cdot\nu) u_{1k}^+=(\p_\nu+iA^+\cdot\nu) u_{2k}^+\quad \textrm{on}\quad \gamma.
\end{align*}
In the same way as in the derivation of \eqref{eq_5_3}, we get
\begin{equation}
 \label{eq_5_8_0}
u_{1k}^+=u_{2k}^+\quad\textrm{on}\quad  D(e),\quad k=1,2,\dots.
\end{equation}
Thus,  
\begin{equation}
 \label{eq_5_8}
u_{2k}^+|_{\p V(\delta)}=u_{1k}^+|_{\p V(\delta)}\to u_1^+|_{\p V(\delta)}\quad\textrm{in}\quad H^{1/2}(\p V(\delta)),\quad k\to \infty.
\end{equation}
By the choice of $V(\delta)$,  it follows from \eqref{eq_energy_estimates_unique}  that
\begin{equation}
 \label{eq_5_9}
\|u_{2k}^+\|_{H^1(V(\delta)\setminus\overline{D_2})}+ \|u_{2k}^-\|_{H^1(D_2)}\le C\|u_{2k}^+|_{\p V(\delta)}\|_{H^{1/2}(\p V(\delta))},\quad k=1,2,\dots.
\end{equation}
Using \eqref{eq_5_8} and \eqref{eq_5_9}, we conclude that there is $(u_2^+,u_2^-)\in H^1(V(\delta)\setminus\overline D_2)\times H^1(D_2)$ such that 
\begin{equation}
 \label{eq_5_10}
(u_{2k}^+,u_{2k}^-)\to (u_2^+,u_2^-)\quad \textrm{in}\quad H^1(V(\delta)\setminus\overline D_2)\times H^1(D_2),\quad k\to\infty. 
\end{equation}
Hence, $(u_2^+,u_2^-)$ satisfy the transmission problem \eqref{eq_5_7_1} with $j=2$. It follows from \eqref{eq_5_8_-1} and \eqref{eq_5_10} that \eqref{eq_5_7} is valid for $u_1^+$ and $u_2^+$. 
 Furthermore, \eqref{eq_5_8_0} implies that 
\begin{equation}
\label{eq_trace_v(delta)}
u_1^+=u_2^+\quad \textrm{on}\quad D(e)\cap V(\delta).  
\end{equation}
On the other hand, let $(\tilde u_2^+,\tilde u_2^-)\in H^1(V(\delta)\setminus\overline D_2)\times H^1(D_2)$ satisfy the transmission problem \eqref{eq_5_7_1} with $j=2$ and $\tilde u_2^+|_{\p V(\delta)}=u_1^+|_{\p V(\delta)}$. 
Since by the choice of the set $V(\delta)$ the transmission problem \eqref{eq_5_7_1} with the Dirichlet boundary conditions has a unique solution, we have $(\tilde u_2^+,\tilde u_2^-)=(u_2^+,u_2^-)$. The claim is proved.

The subsequent analysis will take place in the region $B\cap D_2$. 
In what follows let $(u_j^+,u_j^-)\in H^1(V(\delta)\setminus\overline D_j)\times H^1(D_j)$, $j=1,2$, be such that they satisfy the transmission problems
 \eqref{eq_5_7_1} and the condition \eqref{eq_5_7_2}. 
By the first Green formula \eqref{eq_first_green}, we get
\begin{align*}
0=&\int_{B\cap D_2} b_2(\mathcal{L}_{A_2^-,q_2^-} u_2^-)\overline{v}dx=\int_{B\cap D_2}\nabla u_2^-\cdot\nabla (b_2\overline{v})dx\\
&+i\int_{B\cap D_2} (A_2^-u_2^-\cdot\nabla (b_2\overline{v})
-(A_2^-\cdot\nabla u_2^-)b_2\overline{v})dx+
\int_{B\cap D_2}((A_2^-)^2+q_2^-)u_2^-b_2\overline{v}dx\\
&-(b_2(\p_\nu +iA_2^-\cdot\nu)u_2^-,v)_{(H^{-1/2},H^{1/2})(\p (B\cap D_2))}.
\end{align*}
This implies that
\begin{equation}
\label{eq_5_11}
\begin{aligned}
\int_{B\cap D_2}b_2\nabla u_2^-\cdot\nabla \overline{v}dx&=-\int_{B\cap D_2} (\nabla b_2\cdot \nabla u_2^-)\overline{v}dx\\
&-i\int_{B\cap D_2} A_2^-\cdot (u_2^-\nabla \overline{v}
-\overline{v}\nabla u_2^-)b_2dx\\
&-\int_{B\cap D_2}(iA_2^-\cdot\nabla b_2+((A_2^-)^2+q_2^-)b_2)u_2^-\overline{v}dx\\
&+(b_2(\p_\nu +iA_2^-\cdot\nu)u_2^-,v)_{(H^{-1/2},H^{1/2})(\p (B\cap D_2))}.
\end{aligned}
\end{equation}
It follows from \eqref{eq_5_4_1} that
\begin{equation}
\label{eq_5_4_1_conj}
\mathcal{L}_{-A^+, q^+}\overline v=0\quad \textrm{in}\quad V(\delta). 
\end{equation}
Using the first Green formula \eqref{eq_first_green}, we obtain that
\begin{align*}
0&=\int_{B_2\cap D_2} (\mathcal{L}_{-A^+, q^+}\overline v)a_2u_2^-dx=\int_{B\cap D_2}\nabla \overline{v}\cdot\nabla (a_2 u_2^-)dx\\
&+i\int_{B\cap D_2} (-A^+\overline{v}\cdot\nabla (a_2u_2^-)
+(A^+\cdot\nabla \overline{v})a_2u_2^-)dx+
\int_{B\cap D_2}((A^+)^2+q^+)\overline{v}a_2u_2^-dx\\
&-((\p_\nu -iA^+\cdot\nu)\overline{v},a_2\overline{u_2^-})_{(H^{-1/2},H^{1/2})(\p (B\cap D_2))}.
\end{align*}
This yields that 
\begin{equation}
\label{eq_5_12}
\begin{aligned}
\int_{B\cap D_2}a_2\nabla u_2^-\cdot\nabla \overline{v}dx=&-\int_{B\cap D_2} (\nabla a_2\cdot \nabla \overline{v}) u_2^-dx\\
&+i\int_{B\cap D_2} A^+\cdot(\overline{v} \nabla u_2^- 
-u_2^-\nabla \overline{v} )a_2dx\\
&+\int_{B\cap D_2}(iA^+\cdot\nabla a_2-((A^+)^2+q^+)a_2)u_2^-\overline{v}dx\\
&+(a_2 u_2^-,(\p_\nu +i A^+ \cdot\nu)v)_{(H^{1/2},H^{-1/2})(\p (B\cap D_2))}.
\end{aligned}
\end{equation}
Our next step is to subtract \eqref{eq_5_12}  from \eqref{eq_5_11}. To that end consider first the boundary terms in \eqref{eq_5_11} and \eqref{eq_5_12}. 
Setting
\[
\Gamma(B)=\p D_2\cap B,
\]
and using the transmission conditions in \eqref{eq_5_7_1} with $j=2$ and \eqref{eq_5_7}, 
we get 
\begin{equation}
\label{eq_I_1}
\begin{aligned}
&I_1:=(b_2(\p_\nu +iA_2^-\cdot\nu)u_2^-,v)_{(H^{-1/2},H^{1/2})(\Gamma(B))}\\
&- (a_2 u_2^-,(\p_\nu +i A^+\cdot\nu)v)_{(H^{1/2},H^{-1/2})(\Gamma(B))}
=-\int_{\Gamma(B)}c_2u_2^-\overline{v}dS \\
&+ ((\p_\nu +iA^+\cdot\nu)u_2^+,v)_{(H^{-1/2},H^{1/2})(\p D_2)}
- (u_2^+,(\p_\nu +i A^+\cdot\nu)v)_{(H^{1/2},H^{-1/2})(\p D_2)}\\
&- ((\p_\nu +iA^+\cdot\nu)u_2^+,v)_{(H^{-1/2},H^{1/2})(\p D_2\setminus\Gamma(B))}\\
&+ (u_2^+,(\p_\nu +i A^+\cdot\nu)v)_{(H^{1/2},H^{-1/2})(\p D_2\setminus\Gamma(B))}=-\int_{\Gamma(B)}c_2u_2^-\overline{v}dS \\
&+ ((\p_\nu +iA^+\cdot\nu)u_1^+,v)_{(H^{-1/2},H^{1/2})(\p D_1)}
- (u_1^+,(\p_\nu +i A^+\cdot\nu)v)_{(H^{1/2},H^{-1/2})(\p D_1)}\\
&- ((\p_\nu +iA^+\cdot\nu)u_2^+,v)_{(H^{-1/2},H^{1/2})(\p D_2\setminus\Gamma(B))}\\
&+ (u_2^+,(\p_\nu +i A^+ \cdot\nu)v)_{(H^{1/2},H^{-1/2})(\p D_2\setminus\Gamma(B))}.
\end{aligned}
\end{equation}
The idea in the above computation is to move further away from the pole $x_\delta$. Notice that $\Gamma(B)$ is the portion of the boundary of $D_2$ that is closest to the pole $x_\delta$, and we do not want to have traces of the normal derivatives, integrated over $\Gamma(B)$, in the expression for  $I_1$.

Letting
\begin{equation}
\label{eq_I_2}
\begin{aligned}
I_2:=&(b_2(\p_\nu +iA_2^-\cdot\nu)u_2^-,v)_{(H^{-1/2},H^{1/2})(\p B\cap D_2)}\\
&- (a_2 u_2^-,(\p_\nu +i A^+ \cdot\nu)v)_{(H^{1/2},H^{-1/2})(\p B\cap D_2)},
\end{aligned}
\end{equation}
and subtracting \eqref{eq_5_12}  from \eqref{eq_5_11}, we get
\begin{equation}
\label{eq_5_13}
\begin{aligned}
\int_{B\cap D_2}(b_2-a_2)\nabla u_2^-\cdot\nabla \overline{v}dx=\int_{B\cap D_2} ( (\nabla a_2\cdot \nabla \overline{v}) u_2^--(\nabla b_2\cdot \nabla u_2^-)\overline{v})dx\\
+\int_{B\cap D_2}(ia_2A^+-ib_2A_2^-)\cdot(u_2^-\nabla \overline{v}- \overline{v}\nabla u_2^-)dx-\int_{B\cap D_2} ru_2^-\overline{v}dx+I_1+I_2,
\end{aligned} 
\end{equation}
where
\begin{equation}
\label{eq_r_5_13}
r=iA_2^-\cdot\nabla b_2+((A_2^-)^2+q_2^-)b_2+iA^+\cdot\nabla a_2-((A^+)^2+q^+)a_2\in L^\infty(B\cap D_2).
\end{equation}
Notice that \eqref{eq_5_13} is valid for any 
 $(u_j^+,u_j^-)\in H^1(V(\delta)\setminus\overline D_j)\times H^1(D_j)$, $j=1,2$,  such that they satisfy the transmission problems
 \eqref{eq_5_7_1} and the condition \eqref{eq_5_7_2}, and for any $v\in H^1(V(\delta))$ satisfying \eqref{eq_5_4_1}.  

We shall use \eqref{eq_5_13} with the singular solutions, constructed in  Subsection \ref{sec_fund_sol},  with the poles at $y=x_\delta\notin \overline{V(\delta)}$. 
Let us introduce these singular solutions in our context. 
Denote by $G(x,y)$ the fundamental solution of the operator $\mathcal{L}_{A^+,q^+}$, i.e.
\[
\mathcal{L}_{A^+,q^+}(x,D_x)G(x,y)=\delta(x-y), \quad x,y\in \Omega.
\]
We assume, as we may, that this fundamental solution enjoys  the properties \eqref{eq_fund_sol} and \eqref{eq_integral}. For the existence of such fundamental solutions, we refer to  Appendix \ref{appendix_fundamental_sol}. 

Define 
\begin{equation}
\label{eq_5_14}
\begin{aligned}
&v=v(\cdot, x_\delta)=G(\cdot, x_\delta),\\
&u_j^+=u_j^+(\cdot, x_\delta)=G(\cdot,x_\delta)+E^+_{0j}(\cdot,x_\delta),\\
& u_j^-=u_j^-(\cdot, x_\delta)=G(\cdot,x_\delta)+E^-_{0j}(\cdot,x_\delta),\quad j=1,2,
\end{aligned}
\end{equation}
where 
$(E_{0j}^+(\cdot,x_\delta), E_{0j}^-(\cdot,x_\delta))\in H^1(V(\delta)\setminus\overline{D_j})\times H^1(D_j)$ solve  the following transmission problem,
\begin{equation}
\label{eq_5_15}
\begin{aligned}
&\mathcal{L}_{A^+,q^+}(x,D_x)E_{0j}^+(x,x_\delta)=0\quad \textrm{in}\quad V(\delta)\setminus\overline{D_j},\\
&\mathcal{L}_{A_j^-,q_j^-}(x,D_x)E_{0j}^-(x,x_\delta)=f_j^-\quad \textrm{in}\quad D_j,\\
&E_{0j}^+(\cdot,x_\delta)=a_jE_{0j}^-(\cdot,x_\delta)+(a_j-1)G(\cdot,x_\delta)\quad \textrm{on}\quad \p D_j,\\
(&\p_\nu + iA^+\cdot\nu) E_{0j}^+(\cdot,x_\delta)=b_j(\p_\nu+iA_j^-\cdot \nu)E_{0j}^-(\cdot,x_\delta)+c_jE_{0j}^-(\cdot,x_\delta)+g_{1j} \textrm{ on } \p D_j,\\
&E^+_{0j}(\cdot,x_\delta)=0\quad \textrm{on}\quad \p V(\delta),
\end{aligned}
\end{equation}
where 
\begin{align*}
f^-_j= & 2i(A_j^--A^+)\cdot\nabla G(\cdot,x_\delta)\\
&+(i\nabla\cdot (A_j^--A^+)+(A^+)^2-(A_j^-)^2+q^+-q_j^-)G(\cdot,x_\delta)\in L^2(D_j),\\
g_{1j}=&(b_j-1)\p_\nu G(\cdot,x_\delta)+(i\nu\cdot (b_jA_j^--A^+)+c_j)G(\cdot,x_\delta)\in H^{-1/2}(\p D_j).
\end{align*}

As $\overline{B}\subset \Omega\setminus\overline{D_1}$, we have $\dist(x_\delta, \p D_1)\ge 1/C$, and therefore, 
it follows from \eqref{eq_3_1} and \eqref{eq_fund_sol} that 
\begin{equation}
\label{eq_5_16}
\begin{aligned}
\|E_{01}^+(\cdot,x_\delta)\|_{H^1(V(\delta)\setminus\overline{D_1})}&+\|E_{01}^-(\cdot,x_\delta)\|_{H^1(D_1)}\le C\|G(\cdot,x_\delta)\|_{H^1(D_1)}\le C,
\end{aligned}
\end{equation}
and \eqref{eq_B_0_2} implies that 
\begin{equation}
\label{eq_5_16_6}
\begin{aligned}
\|E_{02}^+(\cdot,x_\delta)\|_{H^1(V(\delta)\setminus\overline{D_2})}+\|E_{02}^-(\cdot,x_\delta)\|_{H^1(D_2)}\le C \|G(\cdot,  x_\delta)\|_{H^1(D_2)}
\le C\delta^{1-n/2},
\end{aligned}
\end{equation}
as $\delta\to 0$.
It is important to mention that \eqref{eq_5_16} implies that the singular behavior of $u_1^\pm$ is the same as the behavior of $G(\cdot,x_\delta)$ when $\delta\to 0$. However, it follows from \eqref{eq_5_16_6} that the term  $E_{02}^\pm(\cdot,x_\delta)$ may not be considered as a remainder  when $\delta\to 0$,  in the definition of $u_2^\pm$. 

Thus, in the left hand side of \eqref{eq_5_13} we would like to have $u_1^+$ instead of $u_2^-$.  To that end, we have
\begin{equation}
\label{eq_ad_5_15_1}
\begin{aligned}
&\int_{B\cap D_2}(b_2-a_2)\nabla u_2^-\cdot\nabla \overline{v}dx=-\int_{B\cap D_2}\nabla\cdot ((b_2-a_2) \nabla \overline{v})u_2^-dx\\
&+((b_2-a_2)u_2^-,\p_\nu v)_{(H^{1/2},H^{-1/2})(\Gamma(B))}+((b_2-a_2)u_2^-,\p_\nu v)_{(H^{1/2},H^{-1/2})(\p B\cap D_2)}\\
&=-\int_{B\cap D_2}\nabla\cdot ((b_2-a_2) \nabla \overline{v})u_2^-dx\\
&+((b_2-a_2)a_2^{-1}u_1^+,\p_\nu v)_{(H^{1/2},H^{-1/2})(\Gamma(B))}+((b_2-a_2)u_2^-,\p_\nu v)_{(H^{1/2},H^{-1/2})(\p B\cap D_2)}.
\end{aligned}
\end{equation}
In the last line we have used the fact that 
\[
u_2^-=a_2^{-1}u_2^+=a_2^{-1}u_1^+\quad\textrm{on}\quad \Gamma(B),
\]
which follows from the transmission conditions in \eqref{eq_5_7_1} and the equality $u_1^+=u_2^+$ on $\p D(e)$, see the discussion after \eqref{eq_trace_v(delta)}. 

On the other hand, 
\begin{equation}
\label{eq_ad_5_15_2}
\begin{aligned}
&((b_2-a_2)a_2^{-1}u_1^+,\p_\nu v)_{(H^{1/2},H^{-1/2})(\Gamma(B))}= \int_{B\cap D_2}a_2^{-1}(b_2-a_2)\nabla u_1^+\cdot\nabla \overline{v}dx\\
&+\int_{B\cap D_2}\nabla\cdot (a_2^{-1}(b_2-a_2) \nabla \overline{v})u_1^+dx
-((b_2-a_2)a_2^{-1}u_1^+,\p_\nu v)_{(H^{1/2},H^{-1/2})(\p B\cap D_2)}.
\end{aligned}
\end{equation}
Substituting \eqref{eq_ad_5_15_2} into \eqref{eq_ad_5_15_1}, we get
\begin{equation}
\label{eq_ad_5_15_3}
\begin{aligned}
&\int_{B\cap D_2}(b_2-a_2)\nabla u_2^-\cdot\nabla \overline{v}dx=\int_{B\cap D_2}a_2^{-1}(b_2-a_2)\nabla u_1^+\cdot\nabla \overline{v}dx\\
&+\int_{B\cap D_2}\nabla\cdot (a_2^{-1}(b_2-a_2) \nabla \overline{v})u_1^+dx -\int_{B\cap D_2}\nabla\cdot ((b_2-a_2) \nabla \overline{v})u_2^-dx\\
&-((b_2-a_2)a_2^{-1}u_1^+,\p_\nu v)_{(H^{1/2},H^{-1/2})(\p B\cap D_2)}+ ((b_2-a_2)u_2^-,\p_\nu v)_{(H^{1/2},H^{-1/2})(\p B\cap D_2)}.
\end{aligned}
\end{equation}
It follows from \eqref{eq_5_4_1_conj} that
\begin{equation}
\label{eq_ad_5_15_4}
\Delta \overline{v}=2iA^+\cdot\nabla \overline{v} + (i(\nabla\cdot A^+)+(A^+)^2+q^+)\overline{v} \quad \textrm{on}\quad B\cap D_2. 
\end{equation}

Substituting \eqref{eq_ad_5_15_3} and \eqref{eq_ad_5_15_4} into \eqref{eq_5_13}, we obtain that
\begin{equation}
\label{eq_5_13_new}
\begin{aligned}
&\int_{B\cap D_2}a_2^{-1}(b_2-a_2)\nabla u_1^+\cdot\nabla \overline{v}dx=\int_{B\cap D_2} ( (\nabla a_2\cdot \nabla \overline{v}) u_2^--(\nabla b_2\cdot \nabla u_2^-)\overline{v})dx\\
&+\int_{B\cap D_2}(ia_2A^+-ib_2A_2^-)\cdot(u_2^-\nabla \overline{v}-\nabla u_2^-\overline{v})dx-\int_{B\cap D_2} ru_2^-\overline{v}dx\\
&\int_{B\cap D_2} u_2^-(\nabla(b_2-a_2) +2i (b_2-a_2)A^+)\cdot \nabla \overline{v}dx\\
&-\int_{B\cap D_2} u_1^+(\nabla(a_2^{-1}(b_2-a_2)) + 2ia_2^{-1}(b_2-a_2)A^+)\cdot\nabla \overline{v}dx\\
&+\int_{B_2\cap D_2} (i(\nabla \cdot A^+)+(A^+)^2+q^+)(u_2^-(b_2-a_2)-u_1^+a_2^{-1}(b_2-a_2))\overline{v}dx+I_1+I_2\\
&+((b_2-a_2)a_2^{-1}u_1^+,\p_\nu v)_{(H^{1/2},H^{-1/2})(\p B\cap D_2)}- ((b_2-a_2)u_2^-,\p_\nu v)_{(H^{1/2},H^{-1/2})(\p B\cap D_2)},
\end{aligned}
\end{equation}
where $r$, $I_1$ and $I_2$ are given by  \eqref{eq_r_5_13}, \eqref{eq_I_1} and \eqref{eq_I_2}, respectively. 
Notice that the idea of deriving \eqref{eq_5_13_new} is to collect the most singular term on the left hand side.  The identity \eqref{eq_5_13_new} will now play the main role in the recovery of the obstacle. 

It is important to mention here that the function $a_2^{-1}(b_2-a_2)>0$ on $\overline{B\cap D_2}$, and that both $\nabla u_1^+$ and $\nabla v$ behave as the gradient of the fundamental solution of the Laplacian, as $\delta$ goes to zero. 
More precisely, recall that $u_1^+$ and $v$ are given by \eqref{eq_5_14}. Then by \eqref{eq_prop_2}, we have
\begin{align*}
&\nabla u_1^+(x,x_\delta)=\nabla_x G_0(x,x_\delta)+ \nabla_x E^+_{01}(x,x_\delta)+ R(x,x_\delta),\\
&\nabla_x v(x,x_\delta)=\nabla_x  G(x,x_\delta)=\nabla_x G_0(x,x_\delta)+ R(x,x_\delta), 
\end{align*}
where $G_0$ is the fundamental solution of $-\Delta$, which is given by \eqref{eq_G_0}, and 
\[
R=\mathcal{O}(|x-x_\delta|^{2-n}),\quad \textrm{as}\quad \delta\to 0.
\]

Substituting $v$, $(u_j^+,u_j^-)$, $j=1,2$, given by \eqref{eq_5_14} into \eqref{eq_5_13_new}, we see that  the left hand side of \eqref{eq_5_13_new}  has the form,
\begin{align*}
\int_{B\cap D_2}a_2^{-1}(b_2-a_2)\nabla u_1^+\cdot\nabla \overline{v}dx=\int_{B\cap D_2}a_2^{-1}(b_2-a_2)|\nabla_x G_0(x,x_\delta)|^2 dx+ I_0,
\end{align*}
 where 
 \begin{equation}
 \label{eq_I_0}
 \begin{aligned}
 &I_0:=\int_{B\cap D_2}a_2^{-1}(b_2-a_2)\nabla_x E^+_{01}(x,x_\delta)\cdot\nabla_x \overline{ G(x,x_\delta)}dx\\
 &+\int_{B\cap D_2}a_2^{-1}(b_2-a_2)(\nabla_x G_0\cdot\overline{R}+ R\cdot \nabla_x G_0)dx+ \int_{B\cap D_2}a_2^{-1}(b_2-a_2) |R|^2dx.
 \end{aligned}
 \end{equation}
It follows from \eqref{eq_G_0} that 
\[
\nabla_xG_0(x,x_\delta)=-\frac{1}{\Upsilon_n|x-x_\delta|^{n-1}}\frac{x-x_\delta}{|x-x_\delta|},
\]
and therefore, as $a_2^{-1}(b_2-a_2)>0$ on $\overline{B\cap D_2}$, using \eqref{eq_D_4} we have the following estimate, 
\begin{equation}
\label{eq_5_upper}
\frac{1}{C}\delta^{2-n}\le\int_{B\cap D_2}a_2^{-1}(b_2-a_2)|\nabla G_0(\cdot,x_\delta)|^2 dx, \quad \textrm{as}\quad \delta\to 0. 
\end{equation}

Moving $I_0$ to the right hand side of  \eqref{eq_5_13_new}, we shall show that the absolute value of the right hand side of \eqref{eq_5_13_new} is bounded by $C\delta^{1-n/2}\mu_n(\delta)$, where 
$\mu_n(\delta)$ is defined by 
\begin{equation}
\label{eq_mu_n}
\mu_n(\delta)=\begin{cases}
(\log\frac{1}{\delta})^{1/2}, & n=3,\\
\delta^{(3-n)/2}, & n\ge 4.
\end{cases}
\end{equation}

Notice that here the quantity $\delta^{1-n/2}$ is related to the $H^1$-norm of the fundamental solution of the magnetic Schr\"odinger operator on $D_2$, see \eqref{eq_B_0_2}, while $\mu_n(\delta)$ is related to the $L^2$-norm of the trace of the fundamental solution of the magnetic Schr\"odinger operator on the boundary of $D_2$, see \eqref{eq_B_0_3}. 

In what follows we shall also need $\tau_n(\delta)$, which is given by
\begin{equation}
\label{eq_tau_n}
\tau_n(\delta)=\begin{cases}
1& n=3,\\
(\log\frac{1}{\delta})^{1/2}& n=4,\\
\delta^{2-n/2} & n\ge 5,
\end{cases}
\end{equation}
which is related to the $L^2$ norm of the fundamental solution to the magnetic Schr\"odinger operator on $D_2$, see \eqref{eq_B_0_1}. 

For any $n\ge 3$, we have
\begin{equation}
\label{eq_5_tau_mu_delta}
1\le \tau_n(\delta)<\!< \mu_n(\delta)<\!< \delta^{1-n/2},\quad\textrm{as}\quad \delta\to 0,
\end{equation}
and 
\begin{equation}
\label{eq_5_tau_mu_delta_2}
 \delta^{1-n/2}\mu_n(\delta)=o(\delta^{2-n}),\quad\textrm{as}\quad \delta\to 0. 
\end{equation}

Thus, the idea is to get a contradiction as $\delta\to 0$, which will show that $D_1=D_2$.

Let us first estimate the absolute value of $I_0$, which is given by�\eqref{eq_I_0}.  Indeed, by \eqref{eq_5_16} and \eqref{eq_B_0}, we get 
\begin{equation}
\label{eq_I_0_1}
\begin{aligned}
\bigg|\int_{B\cap D_2}&a_2^{-1}(b_2-a_2)\nabla_x E^+_{01}(\cdot,x_\delta)\cdot\nabla_x \overline{G(x,x_\delta)}dx\bigg|\\
&\le C \|\nabla E^+_{01}(\cdot,x_\delta)\|_{L^2(B\cap D_2)}\|\nabla  G(\cdot ,x_\delta)\|_{L^2(B\cap D_2)}\le C\delta^{1-n/2}.
\end{aligned}
\end{equation}
Using \eqref{eq_D_3},  we obtain that
\begin{equation}
\label{eq_I_0_2}
0\le \int_{B\cap D_2}a_2^{-1}(b_2-a_2) |R|^2 dx \le C\int_{B\cap D_2}\frac{dx}{|x-x_\delta|^{2(n-2)}}\le C (\tau_n(\delta))^2.
\end{equation}

By \eqref{eq_D_3}, we have 
\begin{equation}
\label{eq_I_0_3}
\begin{aligned}
\bigg|\int_{B\cap D_2}a_2^{-1}(b_2-a_2) (\nabla_x G_0\cdot\overline{R}+ R\cdot \nabla_x G_0)dx\bigg|&\le C\int_{B\cap D_2}\frac{dx}{|x-x_\delta|^{2n-3}}\\
&\le C(\mu_n(\delta))^2, \quad \textrm{as}\quad \delta\to 0.
\end{aligned}
\end{equation}
The inequalities  \eqref{eq_I_0_1},  \eqref{eq_I_0_2}, and  \eqref{eq_I_0_3} imply that 
\begin{equation}
\label{eq_term5_0}
|I_0|\le C\delta^{1-n/2}\mu_n(\delta). 
\end{equation}

In order to continue estimating terms in the right hand side of \eqref{eq_5_13_new} it will be convenient to collect some auxiliary estimates. 
Let $\tilde D\subset \!\subset V(\delta)$ be an open subset with Lipschitz boundary. The following estimate is a direct consequence of \eqref{eq_fund_sol} and \eqref{eq_B_0_2}, 
\[
\|v(\cdot,  x_\delta)\|_{H^1(\tilde D)}\le \begin{cases} C,& \dist(x_\delta, \p \tilde D)\ge 1/C_1, \quad C_1>0,\\
C\delta^{1-n/2}, & \textrm{otherwise}, 
\end{cases}
\]
as $\delta\to 0$.
Furthermore, \eqref{eq_3_2} together with the trace theorem implies that
\begin{equation}
\label{eq_5_16_5}
\begin{aligned}
\|\p_\nu v(\cdot,  x_\delta)\|_{H^{-1/2}(\p \tilde D)}+ \| v(\cdot,  x_\delta)\|_{H^{1/2}(\p \tilde D)}\le C \|v(\cdot,  x_\delta)\|_{H^1(\tilde D)}.
\end{aligned}
\end{equation}
 
As $\overline{B}\subset \Omega\setminus\overline{D_1}$, we have $\dist(x_\delta, \p D_1)\ge 1/C$ and therefore, \eqref{eq_5_16_5}
implies that 
\begin{equation}
\label{eq_5_16_3}
\|\p_\nu v(\cdot,  x_\delta)\|_{H^{-1/2}(\p D_1)}+ \| v(\cdot,  x_\delta)\|_{H^{1/2}(\p D_1)}\le C \|v(\cdot,  x_\delta)\|_{H^1(D_1)}\le C.
\end{equation}
It also follows from \eqref{eq_5_16_5} that 
\begin{equation}
\label{eq_5_16_3_1}
\|\p_\nu v(\cdot,  x_\delta)\|_{H^{-1/2}(\p D_2)}+ \| v(\cdot,  x_\delta)\|_{H^{1/2}(\p D_2)}\le C \|v(\cdot,  x_\delta)\|_{H^1(D_2)}\le C\delta^{1-n/2}.
\end{equation}

Using the fact that 
\[
\mathcal{L}_{A^+,q^+}(x,D_x)E^+_{0j}(x,x_\delta)=0\quad\textrm{in}\quad V(\delta)\setminus\overline{D_j},\quad j=1,2, 
\]
see the first equation of \eqref{eq_5_15},  and using \eqref{eq_trace_cont}, and \eqref{eq_5_16}, we get
\begin{equation}
\label{eq_5_16_2}
\begin{aligned}
\|\p_\nu E^+_{01}(\cdot,  x_\delta)\|_{H^{-1/2}(\p D_1)}+ \| E^+_{01}(\cdot,  x_\delta)\|_{H^{1/2}(\p D_1)}\le C \| E^+_{01}(\cdot,  x_\delta)\|_{H^1(V(\delta)\setminus\overline{D_1})}
\le C.
\end{aligned}
\end{equation}
Similarly,  using \eqref{eq_5_16_6}, we obtain that 
\begin{equation}
\label{eq_5_16_8}
\begin{aligned}
\|\p_\nu E^+_{02}(\cdot,  x_\delta)\|_{H^{-1/2}(\p D_2)}+ \| E^+_{02}(\cdot,  x_\delta)\|_{H^{1/2}(\p D_2)}&\le C \| E^+_{02}(\cdot,  x_\delta)\|_{H^1(V(\delta)\setminus\overline{D_2})}\\
&\le C\delta^{1-n/2}.
\end{aligned}
\end{equation}

It follows from  \eqref{eq_5_14}, \eqref{eq_5_16_3} and \eqref{eq_5_16_2}  that 
\begin{equation}
\label{eq_5_16_4}
\begin{aligned}
\|\p_\nu u_1^+\|_{H^{-1/2}(\p D_1)}&+\|u_1^+\|_{H^{1/2}(\p D_1)}
\le \|\p_\nu G(\cdot, x_\delta)\|_{H^{-1/2}(\p D_1)}
+ \|G(\cdot,x_\delta)\|_{H^{1/2}(\p D_1)}\\
& +\|\p_\nu E^+_{01}(\cdot, x_\delta)\|_{H^{-1/2}(\p D_1)}
+ \|E^+_{01}(\cdot,x_\delta)\|_{H^{1/2}(\p D_1)}\le C.
\end{aligned}
\end{equation}
Similarly, \eqref{eq_5_14}, \eqref{eq_5_16_3_1} and  \eqref{eq_5_16_8} yield that 
\begin{equation}
\label{eq_5_16_9}
\begin{aligned}
\|\p_\nu u_2^+\|_{H^{-1/2}(\p D_2)}&+\|u_2^+\|_{H^{1/2}(\p D_2)}\le \|\p_\nu G(\cdot, x_\delta)\|_{H^{-1/2}(\p D_2)}
+ \|G(\cdot,x_\delta)\|_{H^{1/2}(\p D_2)}\\
& +\|\p_\nu E^+_{02}(\cdot, x_\delta)\|_{H^{-1/2}(\p D_2)}
+ \|E^+_{02}(\cdot,x_\delta)\|_{H^{1/2}(\p D_2)}\le C\delta^{1-n/2}.
\end{aligned}
\end{equation}

Let us start estimating the boundary terms on the right hand side of \eqref{eq_5_13_new}. First for the second and third terms in the last expression for  $I_1$, given by \eqref{eq_I_1}, 
using  \eqref{eq_5_16_3} and \eqref{eq_5_16_4}, we have 
\begin{equation}
\label{eq_term5_1}
\begin{aligned}
|((\p_\nu +i&A^+\cdot\nu)u_1^+,v)_{(H^{-1/2},H^{1/2})(\p D_1)}
- (u_1^+,(\p_\nu +i A^+\cdot\nu)v)_{(H^{1/2},H^{-1/2})(\p D_1)}|\\
&\le C(\|\p_\nu u_1^+\|_{H^{-1/2}(\p D_1)}\|v\|_{H^{1/2}(\p D_1)} + \|u_1^+\|_{H^{1/2}(\p D_1)}\|\p_\nu v\|_{H^{-1/2}(\p D_1)}\\
&+ \|u_1^+\|_{L^2(\p D_1)}\|v\|_{L^2(\p D_1)})\le C.
\end{aligned}
\end{equation}

Let $\tilde D=D_2\setminus(\overline{B\cap D_2})$. Then as $\dist(x_\delta, \p \tilde D)\ge 1/C$,  \eqref{eq_5_16_5} implies that 
\begin{equation}
\label{eq_5_17}
\begin{aligned}
\|\p_\nu v(\cdot,  x_\delta)\|_{H^{-1/2}(\p D_2\setminus\Gamma(B))}&+ \| v(\cdot,  x_\delta)\|_{H^{1/2}(\p D_2\setminus\Gamma(B))}\le C \|v(\cdot,  x_\delta)\|_{H^1(\tilde D)}\le C.
\end{aligned}
\end{equation}

Let us now estimate the last two terms in $I_1$, defined by \eqref{eq_I_1}. Using \eqref{eq_5_16_9} and \eqref{eq_5_17}, we have
\begin{equation}
\label{eq_term5_2}
\begin{aligned}
|((&\p_\nu +iA^+\cdot\nu)u_2^+,v)_{(H^{-1/2},H^{1/2})(\p D_2\setminus\Gamma(B))}\\
&-
(u_2^+,(\p_\nu +i A^+\cdot\nu)v)_{(H^{1/2},H^{-1/2})(\p D_2\setminus\Gamma(B))}|\\
&\le C(\|\p_\nu u_2^+\|_{H^{-1/2}(\p D_2)}\|v\|_{H^{1/2}(\p D_2\setminus \Gamma(B))}
 + \|u_2^+\|_{H^{1/2}(\p D_2)}\|\p_\nu v\|_{H^{-1/2}(\p D_2\setminus \Gamma(B))}\\
&+ \|u_2^+\|_{H^{1/2}(\p D_2)}\|v\|_{H^{1/2}(\p D_2\setminus \Gamma(B))})\le 
C\delta^{1-n/2}.
\end{aligned}
\end{equation}

Let us estimate the first term in the last expression for $I_1$ in \eqref{eq_I_1}. To that end we shall need the following estimate, which is a consequence of the representation \eqref{eq_5_14} of $u_2^-$ together with \eqref{eq_5_16_6} and \eqref{eq_5_16_3_1}, 
\begin{equation}
\label{eq_5_24}
\|u_2^-\|_{H^1(B\cap D_2)}\le \|G(\cdot,x(\delta))\|_{H^1(D_2)} + \|E_{02}^-(\cdot,x(\delta))\|_{H^1(D_2)}\le C\delta^{1-n/2}.
\end{equation}
Using \eqref{eq_B_0_3} and \eqref{eq_5_24}, we get 
\begin{equation}
\label{eq_term5_3}
\begin{aligned}
\bigg|\int_{\Gamma(B)}c_2u_2^-\overline{v}dS\bigg| &\le C\|u_2^-\|_{L^2(\Gamma(B))}\|v\|_{L^2(\Gamma(B))}\le C\|u_2^-\|_{H^1(B\cap D_2)}\mu_n(\delta)\\
&\le C\delta^{1-n/2}\mu_n(\delta),
\end{aligned}
\end{equation}
where $\mu_n(\delta)$ is defined by \eqref{eq_mu_n}.

Summing up,  \eqref{eq_term5_1}, \eqref{eq_term5_2} and \eqref{eq_term5_3} imply that 
\begin{equation}
\label{eq_term5_I_1}
|I_1|\le C\delta^{1-n/2}\mu_n(\delta), 
\end{equation}
as $\delta\to 0$. 

We shall estimate $I_2$, given by \eqref{eq_I_2}.  We have 
\begin{equation}
\label{eq_5_18_0}
\begin{aligned}
|&I_2|=|(b_2(\p_\nu +iA_2^-\cdot\nu)u_2^-,v)_{(H^{-1/2},H^{1/2})(\p B\cap D_2)}\\
&- (a_2 u_2^-,(\p_\nu +i A^+\cdot\nu)v)_{(H^{1/2},H^{-1/2})(\p B\cap D_2)}|\\
&\le C(\|\p_\nu u_2^-\|_{H^{-1/2}(\p B\cap D_2)}\|v\|_{H^{1/2}(\p B\cap D_2)}
 + \|u_2^-\|_{H^{1/2}(\p B\cap D_2)}\|\p_\nu v\|_{H^{-1/2}(\p B\cap D_2)}\\
&+ \|u_2^-\|_{H^{1/2}(\p B\cap D_2)}\|v\|_{H^{1/2}(\p B\cap D_2)}).
\end{aligned}
\end{equation}
We proceed by estimating all terms in the right hand side of the above inequality. 
To that end 
let $\tilde D=D_2\setminus(\overline{B\cap D_2})$. Then as $\dist(x_\delta, \p \tilde D)\ge 1/C$,  \eqref{eq_5_16_5} implies that 
\begin{equation}
\label{eq_5_18}
\begin{aligned}
\|\p_\nu v(\cdot,  x_\delta)\|_{H^{-1/2}(\p B\cap D_2)}+ \| v(\cdot,  x_\delta)\|_{H^{1/2}(\p B\cap D_2)}\le C \|v(\cdot,  x_\delta)\|_{H^1(\tilde D)}
\le C.
\end{aligned}
\end{equation}
By the definition \eqref{eq_5_14} of $u_2^-$, we get
\begin{equation}
\label{eq_5_19}
\begin{aligned}
\|\p_\nu u_2^-\|_{H^{-1/2}(\p B\cap D_2)}&+\| u_2^-\|_{H^{1/2}(\p B\cap D_2)}\\
&\le \|\p_\nu G(\cdot,x_\delta)\|_{H^{-1/2}(\p B\cap D_2)}
+ \|G(\cdot,x_\delta)\|_{H^{1/2}(\p B\cap D_2)}\\
&+ \|\p_\nu E_{02}^-(\cdot,x_\delta))\|_{H^{-1/2}(\p B\cap D_2)}+ \|E_{02}^-(\cdot,x_\delta)\|_{H^{1/2}(\p B\cap D_2)}.
\end{aligned}
\end{equation}
We shall estimate all terms in \eqref{eq_5_19}. First similarly to \eqref{eq_5_18}, we have
\begin{equation}
\label{eq_5_20}
\begin{aligned}
\|\p_\nu G(\cdot,x_\delta)\|_{H^{-1/2}(\p B\cap D_2)}
&+ \|G(\cdot,x_\delta)\|_{H^{1/2}(\p B\cap D_2)}\le C.
\end{aligned}
\end{equation}
The second equation in \eqref{eq_5_15} implies that
\begin{align*}
\Delta E_{02}^-(\cdot,x_\delta)=&-2iA_2^-\cdot \nabla E_{02}^-(\cdot,x_\delta)+(-i\nabla\cdot A_2^-+(A_2^-)^2+q_2^-)E_{02}^-(\cdot,x_\delta)\\
&+f_2^-\quad \textrm{on}\quad D_2,\quad \quad \|f_2^-\|_{L^2(D_2)}\le C\|G(\cdot,x_\delta)\|_{H^1(D_2)}.
\end{align*}
Thus, we conclude from \eqref{eq_trace_cont} and  \eqref{eq_5_16_6}  that 
\begin{equation}
\label{eq_5_21}
\begin{aligned}
\|\p_\nu E_{02}^-(&\cdot,x_\delta)\|_{H^{-1/2}(\p B\cap D_2)}\le C(\|\Delta E_{02}^-(\cdot,x_\delta)\|_{L^2(B\cap D_2)}
+\|E_{02}^-(\cdot,x_\delta)\|_{H^1(B\cap D_2)})\\
&\le C(\|E_{02}^-(\cdot,x_\delta)\|_{H^1(D_2)}
+ \|G(\cdot,x_\delta)\|_{H^1( D_2)})\le C\delta^{1-n/2}.
\end{aligned}
\end{equation}
It follows from \eqref{eq_5_19} with the help of \eqref{eq_5_20}, \eqref{eq_5_21} and the trace theorem that 
\begin{equation}
\label{eq_5_22}
\|\p_\nu u_2^-\|_{H^{-1/2}(\p B\cap D_2)}+\| u_2^-\|_{H^{1/2}(\p B\cap D_2)}\le C\delta^{1-n/2}. 
\end{equation}
We conclude from \eqref{eq_5_18_0}, \eqref{eq_5_18} and \eqref{eq_5_22} that 
\begin{equation}
\label{eq_term5_4}
|I_2|\le C\delta^{1-n/2}.
\end{equation}
We shall now complete estimating the boundary terms in the right hand side of \eqref{eq_5_13_new}.  First, similarly to the estimate \eqref{eq_term5_4} for $I_2$, we get
\begin{equation}
\label{eq_term5_5}
|((b_2-a_2)u_2^-,\p_\nu v)_{(H^{1/2},H^{-1/2})(\p B\cap D_2)}   |\le C\delta^{1-n/2}.
\end{equation}

Using \eqref{eq_5_18} and \eqref{eq_5_16}, we obtain that
\begin{equation}
\label{eq_term5_6}
\begin{aligned}
|((b_2-a_2)a_2^{-1}u_1^+,&\p_\nu v)_{(H^{1/2},H^{-1/2})(\p B\cap D_2)}|\le C\|u_1^+\|_{H^{1/2}(\p B\cap D_2)}\|\p_\nu v\|_{H^{-1/2}(\p B\cap D_2)}\\
&\le C(\|G(\cdot, x_\delta)\|_{H^{1/2}(\p B\cap D_2)}+ \|E^+_{01}(\cdot,x_\delta)\|_{H^{1/2}(\p B\cap D_2)} )\\
&\le C(\|G(\cdot, x_\delta)\|_{H^{1}(\tilde D)}+ \|E^+_{01}(\cdot,x_\delta)\|_{H^{1}(V(\delta)\setminus\overline{D_1})})\le C.  
\end{aligned}
\end{equation}

Let us next estimate the following terms in the right hand side of \eqref{eq_5_13_new}, which contain only $\overline{v}$, $u_2^-$ and $\nabla u_2^-$,
\begin{align*}
I_3:= &-\int_{B\cap D_2} ru_2^-\overline{v}dx -\int_{B\cap D_2} (\nabla b_2\cdot \nabla u_2^-)\overline{v}dx\\
&-\int_{B\cap D_2}(ia_2A^+-ib_2A_2^-)\cdot\nabla u_2^-\overline{v}dx\\
&+\int_{B_2\cap D_2} (i(\nabla \cdot A^+)+(A^+)^2+q^+)(b_2-a_2)u_2^-\overline{v}dx.
\end{align*}
Then using \eqref{eq_5_24} and \eqref{eq_B_0_1}, we have
\begin{equation}
\label{eq_term5_7}
|I_3|\le C\|u_2^-\|_{H^1(B\cap D_2)}\|v\|_{L^2(B\cap D_2)}\le C\delta^{1-n/2}\tau_n(\delta),
\end{equation}
where $\tau_n(\delta)$ is given by \eqref{eq_tau_n}.

Similarly, we also have
\begin{equation}
\label{eq_term5_8}
\begin{aligned}
\bigg|\int_{B_2\cap D_2} (i(\nabla \cdot A^+)+(A^+)^2+q^+)a_2^{-1}(b_2-a_2)u_1^+\overline{v}dx\bigg|\\
\le C\|u_1^+\|_{L^2 (B\cap D_2)}\|v\|_{L^2(B\cap D_2)}\le C(\tau_n(\delta))^2. 
\end{aligned}
\end{equation}

We shall  estimate the following terms in the left hand side of \eqref{eq_5_13_new}, which contain $\nabla \overline{v}$ and $u_2^-$, 
\begin{align*}
I_4:=\int_{B\cap D_2} (\nabla a_2\cdot \nabla \overline{v}) u_2^-dx +\int_{B\cap D_2}(ia_2A^+-ib_2A_2^-)\cdot u_2^-\nabla \overline{v}dx\\
+\int_{B\cap D_2} u_2^-(\nabla(b_2-a_2) +2i (b_2-a_2)A^+)\cdot \nabla \overline{v}dx.
\end{align*}
We shall only estimate the first term in $I_4$, the estimates for the other two terms  being similar.  
Since we do not have an $L^2$-estimate for $u_2^-$ which is better than the $H^1$-estimate, we shall first integrate by parts. We have
\begin{align*}
\int_{B\cap D_2} (\nabla a_2\cdot \nabla \overline{v}) u_2^-dx=&-\int_{B\cap D_2} ( \overline{v} \nabla a_2\cdot \nabla u_2^- + \overline{v} u_2^-\Delta a_2) dx\\
&+\int_{\p (B\cap D_2)} (\p_\nu a_2) \overline{v} u_2^- dS.
\end{align*}
Proceeding similarly as for $I_3$, we get
\[
\bigg| \int_{B\cap D_2} ( \overline{v} \nabla a_2\cdot \nabla u_2^- + \overline{v} u_2^-\Delta a_2) dx \bigg|\le C\delta^{1-n/2}\tau_n(\delta). 
\]
Using the trace theorem, \eqref{eq_5_24} and \eqref{eq_B_0_3},   we obtain that 
\begin{align*}
\bigg|  \int_{\p (B\cap D_2)} (\p_\nu a_2) \overline{v} u_2^- dS  \bigg|\le C\|u_2^-\|_{L^2(\p (B\cap D_2))}\|v\|_{L^2(\p (B\cap D_2))}\\
\le C\|u_2^-\|_{H^1(B\cap D_2)}\|v\|_{L^2(\p (B\cap D_2))}\le C\delta^{1-n/2}\mu_n(\delta),
\end{align*}
where $\mu_n(\delta)$ is defined by \eqref{eq_mu_n}.  
Hence,
\[
\bigg| \int_{B\cap D_2} (\nabla a_2\cdot \nabla \overline{v}) u_2^-dx \bigg|\le C\delta^{1-n/2}\mu_n(\delta),
\]
and similarly,
\begin{equation}
\label{eq_term5_9}
|I_4|\le C\delta^{1-n/2}\mu_n(\delta). 
\end{equation}

In order to finally estimate the following term in the right hand side of  \eqref{eq_5_13_new},
\[
I_5:=\int_{B\cap D_2} B\cdot u_1^+\nabla \overline{v}dx, \quad  B:=\nabla(a_2^{-1}(b_2-a_2)) + 2ia_2^{-1}(b_2-a_2)A^+,
\]
we integrate by parts. We have
\[
I_5=-\int_{B\cap D_2} ((B\cdot \nabla u_1^+)\overline{v}+(\nabla\cdot B) u_1^+ \overline{v})dx+\int_{\p (B\cap D_2)} (\nu\cdot B) u_1^+\overline{v}dS,
\]
and therefore, using \eqref{eq_B_0_1} and  \eqref{eq_B_0_3}, we get 
\begin{equation}
\label{eq_term5_10}
\begin{aligned}
|I_5|\le C(\|u_1^+\|_{H^1(B\cap D_2)}\|v\|_{L^2(B\cap D_2)}+\|u_1^+\|_{L^2(\p (B\cap D_2))}\|v\|_{L^2(\p (B\cap D_2))})\\
\le C\|u_1^+\|_{H^1(B\cap D_2)}(\tau_n(\delta)+\mu_n(\delta))\le C\delta^{1-n/2}\mu_n(\delta).
\end{aligned}
\end{equation}
In the last inequality we have used the fact that 
\[
\|u_1^+\|_{H^1(B\cap D_2)}\le \|G(\cdot,x_\delta)\|_{H^1(B\cap D_2)}+\|E^+_{01}(\cdot,x_\delta)\|_{H^1(B\cap D_2)}\le C\delta^{1-n/2}. 
\] 

Summarizing and keeping in mind \eqref{eq_5_tau_mu_delta},  we conclude from \eqref{eq_5_13_new} with the help of all estimates \eqref{eq_term5_0}, \eqref{eq_term5_I_1},  \eqref{eq_term5_4}, \eqref{eq_term5_5}, \eqref{eq_term5_6}, \eqref{eq_term5_7}, \eqref{eq_term5_8}, \eqref{eq_term5_9} and \eqref{eq_term5_10} that the absolute value of the right hand side of \eqref{eq_5_13_new} is bounded by $C\delta^{1-n/2}\mu_n(\delta)$.  
This together with \eqref{eq_5_upper} implies that 
\[
\frac{1}{C}\delta^{2-n}\le \int_{B\cap D_2}a_2^{-1}(b_2-a_2)|\nabla G_0(\cdot,x(\delta))|^2 dx \le C\delta^{1-n/2}\mu_n(\delta),\quad \textrm{as}\quad \delta\to 0.
\]
In view of \eqref{eq_5_tau_mu_delta_2} we get a contradiction, which shows that $D_1=D_2=:D$.

\begin{rem}
When showing the unique identifiability of the obstacle, we have only used 
the condition $a_jb_j=1$ in $\overline{D_j}$ in order to achieve the unique solvability of the transmission problems in the regions $V(\delta)$ and $\tilde \Omega$, the latter in the application of Lemma \ref{lem_density}. 
In particular, this condition has not been used in all the estimates of this subsection. 
\end{rem}

\subsection{Recovery of the transmission coefficients on the boundary of the obstacle}

\label{sec_transmiss_coef}
In this subsection, we shall carry out all the computations and estimates not using the condition $a_jb_j=1$ in $\overline{D}$, $j=1,2$, with the only occurrence where this condition is needed being an application of Lemma \ref{lem_density} below.

Let $(u_1^+,u_1^-)\in H^1(\Omega\setminus\overline{D})\times H^1(D)$ satisfy the following transmission problem with $j=1$,
\begin{equation}
\label{eq_6_1}
\begin{aligned}
&\mathcal{L}_{A^+,q^+}u_j^+=0\quad \textrm{in}\quad \Omega\setminus\overline{D},\\
&\mathcal{L}_{A_j^-,q_j^-}u_j^-=0\quad \textrm{in}\quad D,\\
&u_j^+=a_ju_j^-\quad \textrm{on}\quad \p D,\\
(&\p_\nu + iA^+\cdot\nu) u_j^+=b_j(\p_\nu+iA_j^-\cdot \nu)u_j^-+c_ju_j^- \quad \textrm{on}\quad \p D,
\end{aligned}
\end{equation}
such that $\supp(u_1^+|_{\p \Omega})\subset \gamma$. 
Then since  
\[
\mathcal{C}_\gamma(A^+,q^+, A_1^-,q_1^-,a_1,b_1,c_1; D)=\mathcal{C}_\gamma(A^+,q^+,A_2^-,q_2^-,a_2,b_2,c_2; D),
\]
there is $(u_2^+,u_2^-)\in H^1(\Omega\setminus\overline{D})\times H^1(D)$, which satisfies the transmission problem \eqref{eq_6_1} with $j=2$, and such  that 
\begin{equation}
\label{eq_6_2}
\begin{aligned}
&u_1^+=u_2^+\quad \textrm{on}\quad \p \Omega,\quad \supp(u_2^+|_{\p \Omega})\subset \gamma,\\
&(\p_\nu+iA^+\cdot\nu) u_1^+=(\p_\nu+iA^+\cdot\nu) u_2^+\quad \textrm{on}\quad \gamma.
\end{aligned}
\end{equation} 
It follows from \eqref{eq_6_1} that
\[
\mathcal{L}_{A^+,q^+}(u_1^+-u_2^+)=0\quad \textrm{on}\quad \Omega\setminus\overline{D},
\]
and \eqref{eq_6_2} that 
\[
u_1^+=u_2^+, \quad \p_\nu u_1^+=\p_\nu u_2^+\quad \textrm{on}\quad \gamma.
\]
As $A^+\in W^{1,\infty}$ and $q^+\in L^\infty$, and 
 $\Omega\setminus\overline{D}$ is connected, by unique continuation we get 
\[
u_1^+=u_2^+\quad\textrm{on}\quad \Omega\setminus\overline{D}.
\]
Since $u_1^+,u_2^+\in H^1(\Omega\setminus\overline{D})$ and $\Delta u_j^+\in L^2(\Omega\setminus\overline{D})$, we have 
\begin{equation}
\label{eq_6_3}
u_1^+=u_2^+,\quad \p_\nu u_1^+=\p_\nu u_2^+\quad \textrm{on}\quad \p D. 
\end{equation}
Then by \eqref{eq_6_3} and the first transmission condition in \eqref{eq_6_1} we get
\begin{equation}
\label{eq_6_4_1}
u_2^-=\frac{a_1}{a_2}u_1^-\quad \textrm{on}\quad \p D. 
\end{equation}
The second transmission condition in \eqref{eq_6_1} together with \eqref{eq_6_3} implies that 
\begin{equation}
\label{eq_6_4_2}
(\p_\nu+iA_2^-\cdot \nu) u_2^-=\frac{b_1}{b_2}(\p_\nu +iA_1^-\cdot \nu) u_1^-+\frac{1}{b_2a_2}(c_1a_2-c_2a_1)u_1^- \quad \textrm{on}\quad \p D. 
\end{equation}
We have
\begin{equation}
\label{eq_6_5}
\begin{aligned}
\mathcal{L}_{A_2^-,q_2^-}(u_2^--u_1^-)=& 2 i(A_2^--A_1^-)\cdot\nabla u_1^-+(i\nabla\cdot(A_2^--A_1^-)\\
&+(A_1^-)^2-(A_2^-)^2+q_1^--q_2^-)u_1^-\quad \textrm{in}\quad D. 
\end{aligned}
\end{equation}

Let $v^-\in H^1(D)$ be a solution to 
\begin{equation}
\label{eq_6_5_1}
\mathcal{L}_{A_2^-, q_2^-}v^-=0\quad \textrm{in}\quad D.
\end{equation}
Then by the second Green formula \eqref{eq_second_green},  we get
\begin{equation}
\label{eq_6_6}
\begin{aligned}
(\mathcal{L}_{A_2^-,q_2^-}(u_2^--u_1^-),v^-)_{L^2(D)}=&-((\p_\nu+iA_2^-\cdot \nu) (u_2^--u_1^-),v^-)_{(H^{-1/2},H^{1/2})(\p D)}\\
&+((u_2^--u_1^-),(\p_\nu +i A_2^- \cdot \nu ) v^-)_{(H^{1/2},H^{-1/2})(\p D)}.
\end{aligned}
\end{equation}
Substituting \eqref{eq_6_4_1}, \eqref{eq_6_4_2} and \eqref{eq_6_5} into \eqref{eq_6_6}, we obtain that
\begin{equation}
\label{eq_6_7}
\begin{aligned}
&\int_{D} (2 i(A_2^--A_1^-)\cdot\nabla u_1^- +i(\nabla \cdot (A_2^--A_1^-))u_1^-)\overline{v^-})dx\\
&+\int_D ((A_1^-)^2-(A_2^-)^2+q_1^--q_2^-)u_1^-\overline{v^-}dx= \bigg(\bigg(\frac{a_1}{a_2}-1\bigg)u_1^-,\p_\nu v^-\bigg)_{(H^{1/2},H^{-1/2})(\p D)}\\
&-\bigg(\bigg(\frac{b_1}{b_2}-1\bigg)\p_\nu u_1^-, v^-\bigg)_{(H^{-1/2},H^{1/2})(\p D)}
+ \bigg( \frac{1}{b_2a_2}(c_2a_1-c_1a_2)u_1^-,v^-\bigg)_{L^2(\p D)}\\
&+(2i(A_2^-\cdot\nu)u_1^-,v^-)_{L^2(\p D)}-\bigg(i\bigg(\frac{b_1}{b_2}(A_1^-\cdot\nu)+ \frac{a_1}{a_2}(A_2^-\cdot\nu)\bigg)u_1^-,v^-\bigg)_{L^2(\p D)}.
\end{aligned}
\end{equation}

In what follows we shall use a linear continuous extension of a Lipschitz function $a\in \textrm{Lip}(\p D)$ to a Lipschitz function $a^*\in \textrm{Lip}(\overline{D})$.
In order to define such an extension, following \cite{Isakov_2008}, we consider a finite open cover $\{U_j\}_{j=1}^N$ of $\p D$ such that $U_j\cap \p D$  is the graph of a Lipschitz function $x_n=\gamma_j(x')$, $x'=(x_1,\dots,x_{n-1})$, after appropriate relabeling of the  coordinates. Let $\chi_j\in C_0^\infty(U_j)$ be such that $0\le \chi_j\le 1$ and $\sum_{j=1}^N\chi_j=1$.  We define $a_j^*(x)=\chi_j(x)a(x',\gamma_j(x'))$ when $x\in U_j$, $j=1,\dots,N$. Then $a^*=\sum_{j=1}^N a_j^*$ is the desired extension of $a$. 

Let $\mu_1,\mu_2\in \textrm{Lip}(\overline{D})$ be two Lipschitz functions. Then we have
\begin{equation}
\label{eq_6_8}
\begin{aligned}
&(\mu_1\p_\nu u_1^-,v^-)_{(H^{-1/2},H^{1/2})(\p D)}=\int_D \mu_1 \nabla u_1^-\cdot \nabla \overline{v^-}dx\\
&+\int_D (\nabla \mu_1\cdot\nabla u_1^-+\mu_1\Delta u_1^-)\overline{v^-}dx
=\int_D \mu_1 \nabla u_1^-\cdot \nabla \overline{v^-}dx + \int_D (\nabla \mu_1\cdot\nabla u_1^-)\overline{v^-}dx\\
 & +\int_D \mu_1(-2iA_1^-\cdot\nabla u_1^--i(\nabla\cdot A_1^-)u_1^-+ ((A_1^-)^2+q_1^-)u_1^-) \overline{v^-}dx.
\end{aligned}
\end{equation}
Using the fact that 
\[
\mathcal{L}_{-A_2^-,q_2^-}\overline{v^-}=0\quad \textrm{in}\quad D,
\]
we also have
\begin{equation}
\label{eq_6_9}
\begin{aligned}
&(\mu_2u_1^-,\p_\nu v^-)_{(H^{1/2},H^{-1/2})(\p D)}=\int_D \mu_2 \nabla u_1^-\cdot \nabla \overline{v^-}dx\\
&+\int_D (\nabla\mu_2\cdot\nabla \overline{v^-}+\mu_2\Delta \overline{v^-})u_1^-dx=\int_D \mu_2 \nabla u_1^-\cdot \nabla \overline{v^-}dx
+ \int_D (\nabla\mu_2\cdot\nabla \overline{v^-})u_1^-dx\\
&+\int_D\mu_2(2iA_2^-\cdot\nabla \overline{v^-}+i(\nabla\cdot A_2^-)\overline{v^-}+ ((A_2^-)^2+q_2^-)\overline{v^-}) u_1^-dx.
\end{aligned}
\end{equation}
Letting 
\[
\mu_1=\bigg(\frac{b_1}{b_2}-1\bigg)^*,\quad \mu_2=\bigg(\frac{a_1}{a_2}-1\bigg)^*,
\]
we conclude from \eqref{eq_6_7}, using \eqref{eq_6_8}, \eqref{eq_6_9}, as well as an integration by parts, 
\begin{equation}
\label{eq_6_10}
\begin{aligned}
&\int_D \bigg( \bigg(\frac{b_1}{b_2}\bigg)^*- \bigg(\frac{a_1}{a_2}\bigg)^*\bigg)\nabla u_1^-\cdot \nabla\overline{v^-}dx= 
\int_D \bigg(\nabla\bigg(\frac{a_1}{a_2}\bigg)^* \cdot \nabla \overline{v^-}\bigg) u_1^-dx\\
&-\int_D \bigg(\nabla\bigg(\frac{b_1}{b_2}\bigg)^* \cdot \nabla u_1^-\bigg) \overline{v^-}dx+ \bigg( \frac{1}{b_2a_2}(c_2a_1-c_1a_2)u_1^-,v^-\bigg)_{L^2(\p D)}\\
&+\int_D \bigg(\frac{a_1}{a_2}\bigg)^*(2iA_2^-\cdot\nabla \overline{v^-}+i(\nabla\cdot A_2^-)\overline{v^-}+ ((A_2^-)^2+q_2^-)\overline{v^-}) u_1^-dx\\
&-\int_D   \bigg(\frac{b_1}{b_2}\bigg)^*(-2iA_1^-\cdot\nabla u_1^--i(\nabla\cdot A_1^-)u_1^-+ ((A_1^-)^2+q_1^-)u_1^-) \overline{v^-}dx\\
&-\bigg(i\bigg(\frac{b_1}{b_2}(A_1^-\cdot\nu)+ \frac{a_1}{a_2}(A_2^-\cdot\nu)\bigg)u_1^-,v^-\bigg)_{L^2(\p D)}.
\end{aligned}
\end{equation}
The identity \eqref{eq_6_10} is valid for any $(u_1^+,u_1^-)\in H^1(\Omega\setminus\overline{D})\times H^1(D)$, which satisfies the transmission problem \eqref{eq_6_1} with $j=1$ and such that $\supp(u_1^+|_{\p \Omega})\subset \gamma$, and any $v\in H^1(D)$ satisfies \eqref{eq_6_5_1}. 

By Lemma \ref{lem_density}, \eqref{eq_6_10} can be extended to all $u_1^-\in H^1(D)$ such that 
\begin{equation}
\label{eq_6_11}
\mathcal{L}_{A_1^-,q_1^-}u_1^-=0\quad\textrm{on}\quad D.
\end{equation}
We shall use \eqref{eq_6_10}  for fundamental solutions of the equation \eqref{eq_6_5_1} and \eqref{eq_6_11} with poles outside $\overline{D}$. 

Next, we would like to show that 
\begin{equation}
\label{eq_6_12_0}
\rho:=\frac{b_1}{b_2}-\frac{a_1}{a_2}=0\quad\textrm{on}\quad \p D.
\end{equation}
If there is a point $x_0\in\p D$ such that $\rho(x_0)\ne 0$, then without loss of generality we may assume that $\rho(x_0)>0$. Hence, there is an open ball, centered at $x_0$, such that 
\[
\rho^*>0 \quad\textrm{on}\quad \overline{B\cap D}. 
\]
Define 
\[
x_\delta=x_0+\delta\nu(x_0),
\]
for $\delta>0$ small so that $x_\delta\in B$.

Let us extend $A_j^-$ and $q_j^-$, $j=1,2$, to the whole of $\R^n$ so that the extensions, which we denote by the same letters, satisfy $A_j^-\in W^{1,\infty}(\R^n)$ and $q_j^-\in L^\infty(\R^n)$. Let 
$G(x,y)$ and $\tilde G(x,y)$ be fundamental solutions of the operators 
$\mathcal{L}_{A_1^-,q_1^-}$ and $\mathcal{L}_{A_2^-, q_2^-}$, respectively, which satisfy  \eqref{eq_fund_sol} and \eqref{eq_integral}. Then we set
\begin{equation}
\label{eq_6_12}
u_1^-=G(\cdot,x_\delta),\quad v^-=\tilde G(\cdot,x_\delta).
\end{equation}
By \eqref{eq_prop_2}, we get
\begin{align*}
\nabla u_1^-=\nabla G_0(\cdot,x_\delta)+R_1(\cdot,x_\delta),\quad
\nabla v^-=\nabla G_0(\cdot,x_\delta)+R_2(\cdot,x_\delta),
\end{align*}
where $G_0$ is the fundamental solution of $-\Delta$, given by \eqref{eq_G_0}, and
\[
R_j=\mathcal{O}(|x-x_\delta|^{2-n}), \quad \textrm{as}\quad \delta\to 0,\quad j=1,2.
\]
For $u_1^-$ and $v^-$, given by \eqref{eq_6_12}, we write
\[
\int_D \bigg( \bigg(\frac{b_1}{b_2}\bigg)^*- \bigg(\frac{a_1}{a_2}\bigg)^*\bigg)\nabla u_1^-\cdot \nabla\overline{v^-}dx=\int_{B\cap D}\bigg( \bigg(\frac{b_1}{b_2}\bigg)^*- \bigg(\frac{a_1}{a_2}\bigg)^*\bigg)|\nabla G_0|^2dx+ I_0,
\]
where 
\begin{equation}
\label{eq_6_13}
\begin{aligned}
I_0:=&\int_{D\setminus \overline{B\cap D}} \bigg( \bigg(\frac{b_1}{b_2}\bigg)^*- \bigg(\frac{a_1}{a_2}\bigg)^*\bigg)\nabla u_1^-\cdot \nabla\overline{v^-}dx\\
&+
\int_{B\cap D}\bigg( \bigg(\frac{b_1}{b_2}\bigg)^*- \bigg(\frac{a_1}{a_2}\bigg)^*\bigg)(\nabla G_0\cdot \overline{R_2}+ R_1\cdot \nabla G_0 + R_1\cdot \overline{R_2})dx. 
\end{aligned}
\end{equation}
By \eqref{eq_D_4}, we have
\begin{equation}
\label{eq_6_14}
\frac{1}{C}\delta^{2-n}\le \int_{B\cap D}\bigg( \bigg(\frac{b_1}{b_2}\bigg)^*- \bigg(\frac{a_1}{a_2}\bigg)^*\bigg)|\nabla G_0|^2dx,\quad \textrm{as}\quad \delta\to 0. 
\end{equation}

Substituting $u_1^-$ and $v^-$, given by \eqref{eq_6_12}, into \eqref{eq_6_10} and moving $I_0$ to the right hand side, we shall show that the absolute value of the right hand side of \eqref{eq_6_10} is bounded by $C\delta^{1-n/2}\mu_n(\delta)$, where $\mu_n(\delta)$ is defined by \eqref{eq_mu_n}.  The first integral in $I_0$, which is defined by \eqref{eq_6_13} is bounded because $\dist(x_\delta,\p (D\setminus \overline{B\cap D}))\ge 1/C$ for some $C>0$.  
Similarly to \eqref{eq_I_0_2} and \eqref{eq_I_0_3}, for the second integral in $I_0$, and hence, for $I_0$, we have
\begin{equation}
\label{eq_6_15}
|I_0|\le C\delta^{1-n/2}\mu_n(\delta),\quad \textrm{as}\quad \delta\to 0. 
\end{equation}
Furthermore, it follows from \eqref{eq_B_0}, \eqref{eq_B_0_1} and \eqref{eq_B_0_3} that  the absolute value of right hand side of \eqref{eq_6_10} is bounded by
\begin{equation}
\label{eq_6_16}
\begin{aligned}
C(\|\nabla v^-\|_{L^2(D)}\|u_1^-\|_{L^2(D)}+\|\nabla u_1^-\|_{L^2(D)}\|v^-\|_{L^2(D)}+ \|u_1^-\|_{L^2(D)}\|v^-\|_{L^2(D)}\\
+\|u_1^-\|_{L^2(\p D)}\|v^-\|_{L^2(\p D)})\le C(\delta^{1-n/2}\tau_n(\delta)+(\tau_n(\delta))^2+(\mu_n(\delta))^2)\le C\delta^{1-n/2}\mu_n(\delta).
\end{aligned}
\end{equation}
Hence, using \eqref{eq_6_14}, \eqref{eq_6_15} and \eqref{eq_6_16}, we conclude from \eqref{eq_6_10}  that 
\[
\frac{1}{C}\delta^{2-n}\le \int_{B\cap D}\bigg( \bigg(\frac{b_1}{b_2}\bigg)^*- \bigg(\frac{a_1}{a_2}\bigg)^*\bigg)|\nabla G_0|^2dx\le    C\delta^{1-n/2}\mu_n(\delta),\quad \textrm{as}\quad \delta\to 0. 
\]
This contradiction shows \eqref{eq_6_12_0}. 

By the definition of the extension,  \eqref{eq_6_12_0} implies that 
\begin{equation}
\label{eq_6_17}
\bigg(\frac{b_1}{b_2}\bigg)^*- \bigg(\frac{a_1}{a_2}\bigg)^*=0\quad\textrm{on}\quad D.
\end{equation}

Next we shall show that 
\begin{equation}
\label{eq_6_17_0}
\beta:= c_1a_2-c_2a_1=0\quad\textrm{on}\quad \p D.
\end{equation}
As above, if $\beta\ne 0$, there is a point $x_0\in \p D$ such that $\beta(x_0)\ne 0$ and an open ball $B$ in $\R^n$, centered at $x_0$ such that without loss of generality $\beta^*>0$ on $\overline{B\cap D}$.  
Also as before, we define $x_\delta=x_0+\delta \nu (x_0)$ for $\delta >0$ small so that $x_\delta\in B$. 

We set 
\[
\Gamma(B):=\p D\cap B, 
\]
Substituting \eqref{eq_6_17} into \eqref{eq_6_10} and integrating by parts, 
we get 
\begin{equation}
\label{eq_6_18}
\begin{aligned}
&\bigg( \frac{1}{b_2a_2}(c_1a_2-c_2a_1)u_1^-,v^-\bigg)_{L^2(\Gamma(B))}
= \bigg( \frac{1}{b_2a_2}(c_2a_1-c_1a_2)u_1^-,v^-\bigg)_{L^2(\p D\setminus\overline{\Gamma(B)})}\\
&
+\int_D \bigg( \nabla\bigg(\frac{a_1}{a_2}\bigg)^* +i \bigg(\frac{a_1}{a_2}\bigg)^* (A_2^--A_1^-)\bigg)\cdot  ( u_1^-\nabla \overline{v^-} -\overline{v^-}\nabla u_1^-)dx\\
&+\int_D \bigg( - i\nabla \bigg(\frac{a_1}{a_2}\bigg)^*\cdot (A_2^-+A_1^-)+ \bigg(\frac{a_1}{a_2}\bigg)^*((A_2^-)^2+q_2^--(A_1^-)^2-q_1^-) \bigg) u_1^-\overline{v^-}dx.
\end{aligned}
\end{equation}
The idea in the above computation is to make the gradients of $u_1^-$ and $v^-$  occur only in the expression $u_1^-\nabla \overline{v^-} -\overline{v^-}\nabla u_1^-$, but not anywhere else.  
 
For $u_1^-$ and $v^-$, given by \eqref{eq_6_12}, we conclude from \eqref{eq_prop_1} and \eqref{eq_prop_2} that  
\begin{equation}
\label{eq_6_19_0}
u_1^-=G_0(\cdot,x_\delta)+r_1(\cdot,x_\delta),\quad v^-=G_0(\cdot,x_\delta)+r_2(\cdot,x_\delta),
\end{equation}
 where $G_0$ is a fundamental solution of $-\Delta$, given by \eqref{eq_G_0}, and 
\begin{equation}
\label{eq_6_19_1}
 \begin{aligned}
 &|r_j(x,x_\delta)|\le C \begin{cases} \log\frac{1}{|x-x_\delta|}, & n=3,\\
|x-x_\delta|^{3-n}, & n\ge 4,
\end{cases}\\
&|\nabla r_j(x,x_\delta)|\le C|x-x_\delta|^{2-n},\quad n\ge 3,
\quad\textrm{as}\quad \delta\to 0,\quad j=1,2.
 \end{aligned}
 \end{equation}
 For $u_1^-$ and $v^-$, given by \eqref{eq_6_12}, we write
 \[
 \bigg( \frac{1}{b_2a_2}(c_1a_2-c_2a_1)u_1^-,v^-\bigg)_{L^2(\Gamma(B))}=\int_{\Gamma(B)}\frac{1}{b_2a_2}(c_1a_2-c_2a_1)|G_0(x,x_\delta)|^2dS +I_1,
 \]
 where 
 \begin{equation}
\label{eq_6_19}
 I_1:=\int_{\Gamma(B)}\frac{1}{b_2a_2}(c_1a_2-c_2a_1) (G_0(\overline{r_2}+r_1)+r_1\overline{r_2}) dS.
 \end{equation}
 By \eqref{eq_D_7}, we have
 \begin{equation}
\label{eq_6_20}
 \frac{1}{C}\sigma_n(\delta)\le \frac{1}{C}\int_{\Gamma(B)}\frac{1}{|x-x_\delta|^{2(n-2)}}dS \le  \int_{\Gamma(B)}\frac{1}{b_2a_2}(c_1a_2-c_2a_1)|G_0(x,x_\delta)|^2dS,
 \end{equation}
 where 
 \[
 \sigma_n(\delta):=\begin{cases}
 \log\frac{1}{\delta}, & n=3,\\
 \delta^{3-n}, & n\ge 4.
 \end{cases}
 \]
 Notice that for all $n\ge 3$,
 \[
  (\tau_n(\delta))^2=o(\sigma_n(\delta)),\quad \textrm{as}\quad\delta\to 0.
 \]

 Since $\dist(x_\delta, \p D\setminus\overline{\Gamma(B)})\ge 1/C$ for some $C>0$, for $u_1^-$ and $v^-$, given by \eqref{eq_6_12}, we get
\begin{equation}
\label{eq_6_21}
\bigg|\bigg( \frac{1}{b_2a_2}(c_2a_1-c_1a_2)u_1^-,v^-\bigg)_{L^2(\p D\setminus\overline{\Gamma(B)})}\bigg|\le C. 
\end{equation}
  By \eqref{eq_6_19_0}, we have
 \begin{align*}
  u_1^-\nabla \overline{v^-} -\overline{v^-}\nabla u_1^-=G_0\nabla(\overline{r_2}-r_1)+(r_1-\overline{r_2})\nabla G_0+r_1\nabla \overline{r_2}-\overline{r_2}\nabla r_1.
 \end{align*}
  Therefore, using \eqref{eq_6_19_1}, we get
 \begin{align*}
  |u_1^-\nabla \overline{v^-} -\overline{v^-}\nabla u_1^-|\le 
  C\begin{cases} |x-x_\delta|^{-2}\log\frac{1}{|x-x_\delta|},& n=3,\\
  |x-x_\delta|^{2(2-n)}, & n\ge 4. 
  \end{cases}
 \end{align*}
 Hence,  using \eqref{eq_D_3} and \eqref{eq_B_-1_2}, we have
 \begin{equation}
\label{eq_6_22}
  \bigg|\int_D \bigg( \nabla\bigg(\frac{a_1}{a_2}\bigg)^* +i \bigg(\frac{a_1}{a_2}\bigg)^* (A_2^--A_1^-)\bigg)\cdot  ( u_1^-\nabla \overline{v^-} -\overline{v^-}\nabla u_1^-)dx \bigg| \le C
  (\tau_n(\delta))^2.
 \end{equation}

Furthermore, for $u_1^-$ and $v^-$, given by \eqref{eq_6_12}, using \eqref{eq_B_0_1}, we obtain that 
\begin{equation}
\label{eq_6_23}
\begin{aligned}
\bigg|\int_D \bigg( - i\nabla \bigg(\frac{a_1}{a_2}\bigg)^*\cdot (A_2^-+A_1^-)+ \bigg(\frac{a_1}{a_2}\bigg)^*((A_2^-)^2+q_2^--(A_1^-)^2-q_1^-) \bigg) u_1^-\overline{v^-}dx\bigg|&\\
\le C(\tau_n(\delta))^2.&
\end{aligned}
\end{equation}

Using \eqref{eq_6_19_1}, \eqref{eq_D_5} and \eqref{eq_B_12}, for the integral $I_1$, given by \eqref{eq_6_19}, we get  
\begin{equation}
\label{eq_6_24}
 |I_1|\le C\begin{cases} \int_{\Gamma(B)}\frac{1}{|x-x_\delta|}\log\frac{1}{|x-x_\delta|}dS, & n=3,\\
 \int_{\Gamma(B)}\frac{dS}{|x-x_\delta|^{2n-5}}, & n\ge 4. 
 \end{cases}\le C(\tau_n(\delta))^2. 
\end{equation}
We conclude from  \eqref{eq_6_18} with the help of \eqref{eq_6_20}, \eqref{eq_6_21}, \eqref{eq_6_22}, \eqref{eq_6_23}, and \eqref{eq_6_24} that 
\[
 \frac{1}{C}\sigma_n(\delta)\le   \int_{\Gamma(B)}\frac{1}{b_2a_2}(c_1a_2-c_2a_1)|G_0(x,x_\delta)|^2dS\le C(\tau_n(\delta))^2,\quad \textrm{as}\quad \delta\to 0.
\]
This contradiction shows that \eqref{eq_6_17_0}.

Hence, \eqref{eq_6_12_0},  \eqref{eq_6_17_0}, and the fact that $a_jb_j=1$ on $\overline{D}$ imply that 
\begin{equation}
\label{eq_6_25}
a_1=  a_2,\quad b_1= b_2,\quad c_1= c_2, \quad \textrm{on}\quad \p D. 
\end{equation}

\subsection{Recovery of the magnetic  and electric potentials}

\label{sec_rec_potentials}

In this subsection we shall assume that $\p D$ is of class $C^{1,1}$.  Then using \eqref{eq_6_25}, we get
\begin{equation}
\label{eq_7_0_0_1}
\mathcal{C}_\gamma(A^+,q^+,A_2^-,q_2^-,a_2,b_2,c_2; D) = \mathcal{C}_\gamma(A^+,q^+,A_1^-,q_1^-, a_2, b_2,c_2; D). 
\end{equation}
Let $(u_1^+,u_1^-)\in H^1(\Omega\setminus\overline{D})\times H^1(D)$ satisfy the following transmission problem with $j=1$,
\begin{equation}
\label{eq_7_0_0_2}
\begin{aligned}
&\mathcal{L}_{A^+,q^+}u_j^+=0\quad \textrm{in}\quad \Omega\setminus\overline{D},\\
&\mathcal{L}_{A_j^-,q_j^-}u_j^-=0\quad \textrm{in}\quad D,\\
&u_j^+=a_2u_j^-\quad \textrm{on}\quad \p D,\\
(&\p_\nu + iA^+\cdot\nu) u_j^+=b_2(\p_\nu+i A_j^-\cdot \nu)u_j^-+c_2u_j^- \quad \textrm{on}\quad \p D,
\end{aligned}
\end{equation}
such that $\supp(u_1^+|_{\p \Omega})\subset \gamma$. It follows from \eqref{eq_7_0_0_1} that there is  $(u_2^+,u_2^-)\in H^1(\Omega\setminus\overline{D})\times H^1(D)$, which satisfies the transmission problem  \eqref{eq_7_0_0_2}
 with $j=2$, and such that 
 \begin{align*}
 &u_1^+=u_2^+\quad\textrm{on}\quad\p\Omega,\quad\supp(u_2^+|_{\p \Omega})\subset\gamma,\\
&(\p_\nu +iA^+\cdot\nu)u_1^+=(\p_\nu +i A^+\cdot\nu)u_2^+\quad \textrm{on}\quad\gamma. 
 \end{align*}
 By unique continuation, we have
 \[
 u_1^+=u_2^+,\quad \p_\nu u_1^+=\p_\nu u_2^+\quad\textrm{on}\quad \p D, 
 \]
and therefore, using the transmission conditions in \eqref{eq_7_0_0_2}, we obtain that 
 \begin{equation}
\label{eq_7_0_0_3}
 u_1^-=u_2^-,\quad (\p_\nu +i A_1^-\cdot\nu )u_1^-=(\p_\nu +i A_2^-\cdot\nu )u_2^-\quad\textrm{on}\quad \p D.
 \end{equation}
 We have
\begin{equation}
\label{eq_7_0_0_4}
 \begin{aligned}
\mathcal{L}_{A_2^-,q_2^-}(u_2^--u_1^-)=& 2 i(A_2^-- A_1^-)\cdot\nabla u_1^-+(i\nabla\cdot(A_2^-- A_1^-)\\
&+( A_1^-)^2-(A_2^-)^2+q_1^--q_2^-)u_1^-\quad \textrm{in}\quad D. 
\end{aligned}
\end{equation}
By the second Green formula \eqref{eq_second_green}, using \eqref{eq_7_0_0_3}, we get 
\begin{equation}
\label{eq_7_0_0_4_1}
\begin{aligned}
(\mathcal{L}_{A_2^-,q_2^-}(u_2^--u_1^-),v^-)_{L^2(D)}=(i(A_2^-- A_1^-)\cdot\nu u_1^-,v^-)_{L^2(\p D)},
\end{aligned}
\end{equation}
where $v^-\in H^1(D)$ is a solution to 
\begin{equation}
\label{eq_7_0_0_5}
\mathcal{L}_{A_2^-,q_2^-}v^-=0\quad \textrm{in}\quad D.
\end{equation}
Substituting  \eqref{eq_7_0_0_4} into \eqref{eq_7_0_0_4_1}, and integrating by parts, we have
\begin{equation}
\label{eq_7_0_new}
\begin{aligned}
&\int_Di  (A_2^--A_1^-)\cdot  ( \overline{v^-}\nabla u_1^--u_1^-\nabla \overline{v^-} )dx\\
&+\int_D ((A_1^-)^2 -(A_2^-)^2 + q_1^- - q_2^-) u_1^-\overline{v^-}dx=0,
\end{aligned}
\end{equation}
which is valid for any $u_1^-\in H^1(D)$ and $v^-\in H^1(D)$, satisfying $\mathcal{L}_{A_1^-,q_1^-}u_1=0$ in $D$ and \eqref{eq_7_0_0_5}, respectively. Here we have used Lemma \ref{lem_density}.

By  \cite[Theorem 5.8]{Salo_2004}, there is $\psi_j\in C^{1,1}(\overline{D},\R)$ which satisfies $\psi_j=0$ on $\p D$ and $\p_\nu  \psi_j=- A_j^-\cdot\nu$ on $\p D$, $j=1,2$.  It follows from \eqref{eq_int_1} that $ A_j^-$ can be replaced by $A_j^-+\nabla \psi_j$, $j=1,2$.  Hence, we may and shall assume that the normal components of $A_j^-$ satisfy $A_j^-\cdot\nu=0$ on $\p D$, $j=1,2$.  Then by Proposition \ref{prop_boundary_rec} below applied to each connected component of $D$, we conclude form \eqref{eq_7_0_new}
that $A_1^-=A_2^-$ on $\p D$.

Let $B\subset \R^n$ be an open ball such that $D\subset\subset \Omega\subset\subset B$. As the boundary of $\Omega$ is connected,  we get that $B\setminus\overline{D}$ is connected. Since $A_1^-=A_2^-$ on $\p D$, we can extend $ A_1^-$ and $A_2^-$ to $B$ so that the extensions, which we shall denote by the same letters, agree on $B\setminus \overline{D}$, have compact support, and satisfy $ A_1^-, A_2^-\in W^{1,\infty}(B)$. We also extend $q_j^-$ to $B$ so that $q_j^-\in L^\infty(B)$ and $q_j^-=0$ on $B\setminus\overline{ D}$, $j=1,2$.  

Hence, \eqref{eq_7_0_new} yields that 
\begin{equation}
\label{eq_7_0_new_1}
\begin{aligned}
&\int_Bi  (A_2^-- A_1^-)\cdot  ( \overline{v^-}\nabla u_1^--u_1^-\nabla \overline{v^-} )dx\\
&+\int_B (( A_1^-)^2 -(A_2^-)^2 + q_1^- - q_2^-) u_1^-\overline{v^-}dx=0,
\end{aligned}
\end{equation}
for any $u_1^-\in H^1(B)$ and $v^-\in H^1(B)$ , which solve 
\[
\mathcal{L}_{ A_1^-,q_1^-}u_1^-=0\quad\textrm{in}\quad B, \quad \mathcal{L}_{A_2^-,q_2^-}v^-=0\quad\textrm{in}\quad B. 
\]
Notice that  identity \eqref{eq_7_0_new_1}  is exactly what one encounters when solving the inverse boundary value problem for the magnetic Schr\"odinger operator. 
The next step is therefore to construct complex geometric optics solutions for the magnetic Schr\"odinger operator on $B$. Using the method of Carleman estimates, complex geometric optics solutions were constructed in 
\cite{DKSU_2007} for $C^2$--magnetic potentials, and in \cite{Knu_Salo_2007}, the construction was generalized to less regular magnetic potentials, including the Lipschitz continuous case. We refer also to 
\cite{KrupLassasUhlmann_slab}, where this construction was reviewed in the latter case. Substituting these complex geometric optics solutions into \eqref{eq_7_0_new_1} and arguing as in \cite{DKSU_2007, NakSunUlm_1995,  Salo_2004, Sun_1993}, see also \cite[Theorem 1.1]{KrupLassasUhlmann_slab}, we conclude that there is a function $\psi\in C^{1,1}(\overline{B},\R)$ such that  $\psi|_{\p B}=0$ and  
\begin{equation}
\label{eq_7_3_5sec}
A_2^-- A_1^-=\nabla \psi\quad \textrm{in}\quad B.
\end{equation}
Since the potentials $A_2^-$ and  $ A_1^-$ agree on the connected set $B\setminus\overline{D}$, it follows that $\psi=0$ on $\p D$, and therefore, using \eqref{eq_int_1}, we can assume in what follows that $A_2^-= A_1^-$ in $D$. Going back to \eqref{eq_7_0_new} and arguing as in  \cite{DKSU_2007, NakSunUlm_1995,  Salo_2004, Sun_1993}, see also \cite[Theorem 1.1]{KrupLassasUhlmann_slab},  we conclude that 
\begin{equation}
\label{eq_7_4_5sec}
q_1^-=q_2^-\quad \textrm{in}\quad D.
\end{equation}
The proof of Theorem \ref{thm_main_sa} is complete. 

\begin{rem}
Assuming that zero is not a Dirichlet eigenvalue of the operator $\mathcal{L}_{A_j^-,q_j^-}$ in $D$, $j=1,2$, using Lemma \ref{lem_density} together with a priori estimates for the Dirichlet problem for $\mathcal{L}_{A_j^-,q_j^-}$ in $D$, we get  $\mathcal{C}(A_1^-,q_1^-)=\mathcal{C}(A_2^-,q_1^-)$. Here we write for $j=1,2$,
\[
\mathcal{C}(A_j^-,q_j^-)=\{(w|_{\p D}, (\p_\nu + i A_j^-\cdot\nu)w|_{\p D}):w\in H^1(D), \mathcal{L}_{A_j^-,q_j^-}w=0\textrm{ in }D\}.
\]
 In this case we can conclude that \eqref{eq_7_3_5sec} and  \eqref{eq_7_4_5sec} hold by appealing directly to the results of \cite{Salo_2004}.
\end{rem}

\section{Boundary reconstruction of the magnetic potential} 

\label{sec_boundary_rec}

When recovering the magnetic potential in Theorem \ref{thm_main_sa}, an important step consists in determining the boundary values of the tangential component of the magnetic potential. The purpose of this section is to carry out this step by adapting the method of  \cite{Brown_Salo_2006}.  Compared with the latter work, here we do not assume that the Dirichlet problem for the magnetic Schr\"odinger operator is well-posed. 

Let $\Omega\subset\R^n$, $n\ge 3$, be a bounded domain  in $\R^n$ with $C^1$ boundary.  Notice that such a regularity of the boundary is important in the method of \cite{Brown_Salo_2006}.

To circumvent the difficulty related to the fact that zero may be a Dirichlet eigenvalue  we shall require a solvability result for the magnetic Schr\"odinger operator, which is based on a Carleman estimate with a gain of two derivatives, obtained in  \cite{SaloTzou_2009}.  We have learned of the idea of using a Carleman estimate to handle the case when zero is a Dirichlet eigenvalue from 
the work \cite{SaloTzou_2009} on the Dirac operator. 

\begin{prop}  \cite{SaloTzou_2009}. Let $\varphi(x)=\alpha\cdot x$, $\alpha\in \R^n$,  $|\alpha|=1$,  and let  $\varphi_{\varepsilon}=\varphi+\frac{h}{\varepsilon}\frac{\varphi^2}{2}$
be a convexification of $\varphi$. Then for $0<h\ll \varepsilon\ll 1$ and $s\in \R$,
\begin{equation}
\label{eq_7_5}
\frac{h}{\sqrt{\varepsilon}}\|u\|_{H^{s+2}_{\emph{\textrm{scl}}}}\le C\|e^{\varphi_{\varepsilon}/h}(-h^2\Delta)e^{-\varphi_{\varepsilon}/h}u\|_{H^s_{\emph{\textrm{scl}}}},
\end{equation}
for all $u\in C^\infty_0(\Omega)$. 
\end{prop}
Here 
\[
\|u\|_{H^s_{\textrm{scl}}}=\|\langle h D\rangle^s u \|_{L^2}, \quad \langle \xi\rangle=(1+|\xi|^2)^{1/2},
\]
is the natural semiclassical norm in the Sobolev space $H^s(\R^n)$.

The following result concerns  a similar Carleman estimate for the magnetic Schr\"odinger operator $\mathcal{L}_{A,q}$ with $s=-1$.  
\begin{prop}
Let $A\in W^{1,\infty}(\Omega,\C^n)$,  $q\in L^\infty(\Omega,\C)$, and  let $\varphi(x)=\alpha\cdot x$, $\alpha\in \R^n$,  $|\alpha|=1$. Then for $h>0$ small enough,
\begin{equation}
\label{eq_7_5_mag}
h\|u\|_{H^{1}_{\textrm{scl}}}\le C\|e^{\varphi/h} (h^2\mathcal{L}_{A,q})e^{-\varphi/h} u\|_{H^{-1}_{\textrm{scl}}},
\end{equation}
for all $u\in C^\infty_0(\Omega)$. 

\end{prop}

\begin{proof}
Let us write
\[
\mathcal{L}_{A,q}=-\Delta +\tilde A\cdot \nabla + \tilde q,
\]
where
\[
\tilde A:=-2i A\in W^{1,\infty}(\Omega,\C^n),\quad \tilde q=-i(\nabla\cdot A)+A^2+q\in L^\infty(\Omega,\C). 
\]
We have
\begin{equation}
\label{eq_7_6}
\|h^2\tilde qu\|_{H^{-1}_{\textrm{scl}}}\le h^2\|\tilde q\|_{L^\infty}\|u\|_{H^1_{\textrm{scl}}},
\end{equation}
and
\[
e^{\varphi_{\varepsilon}/h}(h^2\tilde A\cdot\nabla)e^{-\varphi_{\varepsilon}/h}=h^2 \tilde A\cdot \nabla -h\tilde A\cdot\nabla\varphi_{\varepsilon}.
\]
Let $0<\varepsilon\ll 1$ be independent of $h$. Then for $h$ small enough, 
\begin{equation}
\label{eq_7_7}
\|h(\tilde A\cdot\nabla\varphi_{\varepsilon}) u\|_{H^{-1}_{\textrm{scl}}}\le h\bigg\|\bigg(\tilde A\cdot \bigg(1+\frac{h}{\varepsilon}\varphi\bigg)\nabla\varphi \bigg)u\bigg\|_{L^2}\le Ch\|u\|_{H^1_{\textrm{scl}}}.
\end{equation}
We have
\[
h^2\tilde A\cdot \nabla u=h(h\nabla)(\tilde A u)-h^2(\nabla \cdot\tilde A)u.
\]
Since the operator $h\nabla$ maps $L^2(\Omega)\to H^{-1}_{\textrm{scl}}(\Omega)$, we get
\begin{equation}
\label{eq_7_8}
\|h(h\nabla)(\tilde A u)\|_{H^{-1}_{\textrm{scl}}}\le hC\|\tilde Au\|_{L^2}\le Ch\|u\|_{H^1_{\textrm{scl}}},
\end{equation}
and 
\begin{equation}
\label{eq_7_9}
\|h^2(\nabla \cdot\tilde A)u\|_{H^{-1}_{\textrm{scl}}}\le Ch^2\|u\|_{H^1_{\textrm{scl}}}.
\end{equation}
Fixing $\varepsilon>0$ small enough, and combining \eqref{eq_7_5},\eqref{eq_7_6}, \eqref{eq_7_7}, \eqref{eq_7_8}, and \eqref{eq_7_9}, we obtain that
\begin{align*}
h\|u\|_{H^{1}_{\textrm{scl}}}&\le C\|e^{\varphi/h} e^{\varphi^2/(2\varepsilon)}(h^2\mathcal{L}_{A,q})e^{-\varphi/h} e^{-\varphi^2/(2\varepsilon)}u\|_{H^{-1}_{\textrm{scl}}}. 
\end{align*}
The claim follows. 
\end{proof}

The formal $L^2$-adjoint of the operator $\mathcal{L}_{\varphi}=e^{\varphi/h} (h^2\mathcal{L}_{A,q})e^{-\varphi/h}$ is given by
$\mathcal{L}^*_{\varphi}=e^{-\varphi/h} (h^2\mathcal{L}_{\overline{A},\overline{q}})e^{\varphi/h}$.  The estimate \eqref{eq_7_5_mag} also holds  for the formal adjoint $\mathcal{L}^*_{\varphi}$.

Using the Hahn-Banach theorem, one can convert the Carleman estimate \eqref{eq_7_5_mag} for $\mathcal{L}^*_{\varphi}$ into the following solvability result.  We refer to \cite{Ken_Sjo_Uhl_2007} and \cite{KrupLassasUhlmann} for such an argument. 

\begin{prop}
\label{prop_solvability_ap}
Let $A\in W^{1,\infty}(\Omega,\C^n)$, $q\in L^\infty(\Omega,\C)$, and  let $\varphi(x)=\alpha\cdot x$, $\alpha\in \R^n$,  $|\alpha|=1$.  If $h>0$ small enough, then for any $v\in H^{-1}(\Omega)$, there is a solution $u\in H^1(\Omega)$ of the equation 
\[
e^{\varphi/h} (h^2\mathcal{L}_{A,q})e^{-\varphi/h}u=v\quad\textrm{in}\quad \Omega,
\]
which satisfies
\[
\|u\|_{H^1_{\emph{\textrm{scl}}}(\Omega)}\le \frac{C}{h}\|v\|_{H^{-1}_{\emph{\textrm{scl}}}(\Omega)}.
\]

\end{prop}

Here 
\[
\|v\|_{H^{-1}_{\textrm{scl}}(\Omega)}=\sup_{w\in C^\infty_0(\Omega)}\frac{|(w,v)_{H^1_0,H^{-1}}|}{\|w\|_{H^1_{\textrm{scl}}(\Omega)}}. 
\]

The following proposition is an extension of the result of \cite{Brown_Salo_2006}. Notice that here we do not assume the well-posedness of the Dirichlet problem for the magnetic Schr\"odinger operator.   

\begin{prop} 
\label{prop_boundary_rec}
Let $\Omega\subset \R^n$, $n\ge 3$, be a bounded domain with $C^1$ boundary, and 
let $A_j\in W^{1,\infty}(\Omega,\C^n)$ and $q_j\in L^\infty(\Omega,\C)$, $j=1,2$.  Assume that the identity  
\begin{equation}
\label{eq_8_1_identity}
\begin{aligned}
\int_\Omega i  (A_2- A_1)\cdot  ( u_1\nabla \overline{u_2} -\overline{u_2}\nabla u_1)dx
+\int_\Omega ((A_2)^2 -( A_1)^2+  q_2-q_1 ) u_1\overline{u_2}dx=0,
\end{aligned}
\end{equation}
holds for any $u_1\in H^1(\Omega)$ and $u_2\in H^1(\Omega)$, satisfying 
\[
\mathcal{L}_{A_1,q_1} u_1=0\quad \textrm{in}\quad \Omega,\quad  \mathcal{L}_{\overline{A_2},\overline{q_2}} u_2=0\quad \textrm{in}\quad \Omega.
\]
Then 
\[
\tau\cdot (A_2- A_1)(x_0)=0,
\]
for all points $x_0\in \p \Omega$ and all unit tangent vectors $\tau\in T_{x_0}(\p \Omega)$. 
\end{prop}

\begin{proof}
We shall follow closely \cite{Brown_Salo_2006}. As $\Omega$ is a $C^1$-domain, it has a defining function $\rho\in C^1(\R^n,\R)$ such that $\Omega=\{x\in \R^n:\rho(x)>0\}$, $\p \Omega=\{x:\rho(x)=0\}$, and 
$\nabla \rho$ does not vanish on $\p \Omega$. We fix $x_0\in \p \Omega$, and a unit vector $\tau$, which is tangent to $\p \Omega$. We normalize $\rho$ so that $\nabla \rho(x_0)=-\nu(x_0)$ where $\nu$ is the unit outer normal to $\p \Omega$. 
By an affine change of coordinates we may assume that $x_0$ is the origin and $\nu(x_0)=-e_n$, and therefore, $\nabla \rho(0)=e_n$. 

Let $\omega(t)$, $t\ge 0$, be a modulus of continuity for $\nabla \rho$, which is  a strictly  increasing continuous function, such that  
$\omega(0)= 0$. 
Let $\eta\in C^\infty_0(\R^n,\R)$ be a function such that $\supp (\eta)\subset B(0,1/2)$, and
\[
\int_{\R^{n-1}}\eta(x',0)^2dx'=1,
\]
where $B(0,1/2)$ is a ball of radius $1/2$, centered at $0$, and $x'=(x_1,\dots,x_{n-1})$.
We set $\eta_M(x)=\eta(Mx',M\rho(x))$, for $M>0$. Hence, for $M>0$ large enough, $\supp(\eta_M)\subset B(0,1/M)$. 
Following \cite{Brown_Salo_2006}, for $N>0$, we define $v_0$ by
\begin{equation}
\label{eq_7_1}
v_0(x)=\eta_M(x)e^{N(i\tau\cdot x-\rho(x))}.
\end{equation}
The function $v_0$ is of class $C^1$ with $\supp(v_0)\subset B(0,1/M)$.  Following   \cite{Brown_Salo_2006}, we relate the parameters $N$ and $M$  by the equation 
\begin{equation}
\label{eq_7_2}
M^{-1}\omega(M^{-1})=N^{-1}. 
\end{equation}
As $\omega$ is strictly increasing, the equation  \eqref{eq_7_2} has exactly one solution $N$ for each $M$. Since $\omega(t)\to 0$ as $t\to +0$, there is $M_0$ such that $\omega(M^{-1})<1$ for $M>M_0$. We shall assume that $M>M_0$ and therefore, $N>M$.

Let $v_1\in H^1(\Omega)$ be the solution to the following Dirichlet problem for the Laplacian,
\begin{align*}
-\Delta v_1&=\Delta v_0, \quad\textrm{in}\quad \Omega,\\
v_1|_{\p \Omega}&=0.   
\end{align*}
We shall need the following estimates, obtained in  \cite{Brown_Salo_2006},
\begin{equation}
\label{eq_7_3}
\|v_0\|_{L^2(\Omega)}\le CM^{(1-n)/2}N^{-1/2},
\end{equation}
\begin{equation}
\label{eq_7_4}
\|v_1\|_{L^2(\Omega)}\le CM^{(1-n)/2}N^{-1/2},
\end{equation}
\begin{equation}
\label{eq_7_13}
\lim_{M\to \infty} M^{n-1}N\int_\Omega e^{-2N\rho(x)}\eta_M^2(x)dx=\frac{1}{2}\int_{\R^{n-1}}\eta^2(x',0)dx'=\frac{1}{2},
\end{equation}
\begin{equation}
\label{eq_7_14}
\bigg|\int_\Omega e^{-2N\rho(x)}\eta_M^2(x)dx\bigg|\le CM^{1-n}N^{-1},
\end{equation}
and
\begin{equation}
\label{eq_7_15}
\|\nabla v_1\|_{L^2(\Omega)}\le C\omega(M^{-1})N^{1/2}M^{(1-n)/2}.
\end{equation}

Next we would like to show the  existence of a solution $u_1\in H^1(\Omega)$ to the magnetic Schr\"odinger operator
\begin{equation}
\label{eq_7_10}
\mathcal{L}_{A_1,q_1}u_1=0\quad \textrm{in}\quad \Omega,
\end{equation}
of the form 
\begin{equation}
\label{eq_7_11}
u_1=v_0+v_1+r_1,
\end{equation}
with 
\begin{equation}
\label{eq_7_12}
\|r_1\|_{H^1(\Omega)}\le C\|v_0+v_1\|_{L^2(\Omega)}\le CM^{(1-n)/2}N^{-1/2}.
\end{equation}
To that end, plugging \eqref{eq_7_11} into \eqref{eq_7_10}, we obtain that
\[
\mathcal{L}_{A_1,q_1} r_1=2iA_1\cdot\nabla(v_0+v_1)+(i\nabla \cdot A_1-(A_1)^2-q_1)(v_0+v_1)\quad \textrm{in}\quad \Omega.
\]
Applying Proposition \ref{prop_solvability_ap} with $h>0$ small but fixed, we conclude the existence of $r_1\in H^1(\Omega)$ such that
\[
\|r_1\|_{H^1(\Omega)}\le C\|2iA_1\cdot\nabla(v_0+v_1)+(i\nabla \cdot A_1-(A_1)^2-q_1)(v_0+v_1)\|_{H^{-1}(\Omega)}.
\]
Let $\psi\in H^1_0(\Omega)$. Then
\begin{align*}
|(2iA_1\cdot\nabla&(v_0+v_1)+(i\nabla \cdot A_1-(A_1)^2-q_1)(v_0+v_1),\psi)_{(H^{-1},H^{1}_0)(\Omega)}|\\
&\le |(\nabla\cdot(2iA_1(v_0+v_1)) ,\psi)_{(H^{-1},H^{1}_0)(\Omega)}|+ C\|v_0+v_1\|_{L^2(\Omega)}\|\psi\|_{L^2(\Omega)}\\
&\le C\|v_0+v_1\|_{L^2(\Omega)}\|\psi\|_{H^1(\Omega)},
\end{align*}
which implies \eqref{eq_7_12}.

Similarly, 
let 
\begin{equation}
\label{eq_7_16}
u_2=v_0+v_1+r_2,
\end{equation}
where $r_2\in H^1(\Omega)$ satisfies \eqref{eq_7_12}, be a solution of $\mathcal{L}_{\overline{A_2},\overline{q_2}}u_2=0$ in $\Omega$. 

The next step is to substitute $u_1$ and $u_2$, given by \eqref{eq_7_11} and \eqref{eq_7_16} into the identity \eqref{eq_8_1_identity}, multiply it by $M^{n-1}$ and compute the limit as $M\to \infty$. 
To that end we have
\[
\nabla v_0=(\nabla\eta_M(x)+\eta_M(x)N(i\tau-\nabla\rho(x))e^{N(i\tau\cdot x-\rho(x))},
\]
and 
\[
v_0\nabla\overline{v_0}-\overline{v_0}\nabla v_0=-2i\eta_M^2(x)e^{-2N\rho(x)}N\tau.
\]
Thus, by \eqref{eq_7_13} and \eqref{eq_7_14}, we get
\begin{equation}
\label{eq_7_18}
\begin{aligned}
\lim_{M\to \infty} &M^{n-1}\int_\Omega (A_2-A_1)\cdot (v_0\nabla\overline{v_0}-\overline{v_0}\nabla v_0)dx\\
&=-2i  ((A_2-A_1)(0)\cdot\tau)\lim_{M\to \infty}M^{n-1}N\int_\Omega \eta^2_M(x)e^{-2N\rho(x)}dx\\
&-2i\lim_{M\to \infty}M^{n-1}N\int_\Omega ((A_2-A_1)(x)-(A_2-A_1)(0))\cdot\tau \eta^2_M(x)e^{-2N\rho(x)}dx\\
&=-i(A_2-A_1)(0)\cdot\tau.
\end{aligned}
\end{equation}
Now \eqref{eq_7_3}, \eqref{eq_7_4} and \eqref{eq_7_12} imply that
\begin{equation}
\label{eq_7_17}
\|u_j\|_{L^2(\Omega)}\le CM^{(1-n)/2}N^{-1/2},\quad j=1,2,
\end{equation}
where $u_1$ and $u_2$ are given by \eqref{eq_7_11} and \eqref{eq_7_16}, respectively.  Using \eqref{eq_7_15} and \eqref{eq_7_17}, we get
\begin{equation}
\label{eq_7_19}
\begin{aligned}
M^{n-1}\bigg|\int_\Omega& (A_2-A_1)\cdot (u_1\nabla\overline{v_1}-\overline{u_2}\nabla v_1)dx\bigg|\\
&\le CM^{n-1}(\|u_1\|_{L^2(\Omega)}+ \|u_2\|_{L^2(\Omega)})\|\nabla v_1\|_{L^2(\Omega)}
\le C\omega(M^{-1}).
\end{aligned}
\end{equation}
By \eqref{eq_7_12} and \eqref{eq_7_17}, we obtain that 
\begin{equation}
\label{eq_7_20}
\begin{aligned}
M^{n-1}\bigg|\int_\Omega &(A_2-A_1)\cdot (u_1\nabla\overline{r_2}-\overline{u_2}\nabla r_1)dx\bigg|
\le CM^{n-1}(\|u_1\|_{L^2(\Omega)}\|\nabla r_2\|_{L^2(\Omega)}\\
&+ \|u_2\|_{L^2(\Omega)}\|\nabla r_1\|_{L^2(\Omega)})
\le CN^{-1}=CM^{-1}\omega(M^{-1}).
\end{aligned}
\end{equation}

Furthermore, since $v_1|_{\p \Omega}=0$,    we have
\[
\int_\Omega (A_2-A_1)\cdot ((v_1+r_1)\nabla\overline{v_0}-(\overline{v_1}+\overline{r_2})\nabla v_0)dx=I_1+I_2,
\]
where
\begin{align*}
I_1:=&
-\int_\Omega ((\nabla\cdot (A_2-A_1))(v_1+r_1)\overline{v_0} +(A_2-A_1)\cdot (\nabla v_1+\nabla r_1) \overline{v_0})dx\\
&+\int_\Omega ((\nabla\cdot (A_2-A_1))(\overline{v_1}+\overline{r_2}) v_0+(A_2-A_1)\cdot (\nabla \overline{v_1}+\nabla\overline{r_2}) v_0 )dx,
\end{align*}
and 
\begin{align*}
I_2:=\int_{\p \Omega}  ((A_2-A_1)\cdot\nu)  r_1\overline{v_0} dS- \int_{\p \Omega} ((A_2-A_1)\cdot \nu) \overline{r_2} v_0 dS.
\end{align*}
It follows from \eqref{eq_7_3},  \eqref{eq_7_4}, \eqref{eq_7_12} and \eqref{eq_7_15} that
\begin{equation}
\label{eq_7_21}
\begin{aligned}
M^{n-1}|I_1|&\le CM^{n-1}(\|v_1\|_{L^2(\Omega)}+\|\nabla v_1\|_{L^2(\Omega)}+\|r_1\|_{H^1(\Omega)}+\|r_2\|_{H^1(\Omega)})\|v_0\|_{L^2(\Omega)}\\
&\le C(N^{-1}+\omega(M^{-1}))\le C\omega(M^{-1}).
\end{aligned}
\end{equation}
A direct computation shows that
\[
\|v_0\|_{L^2(\p \Omega)}\le CM^{(1-n)/2}.
\]
This together with the trace theorem and \eqref{eq_7_12} implies that
\begin{equation}
\label{eq_7_21_1}
\begin{aligned}
M^{n-1}|I_2|&\le CM^{n-1}(\|r_1\|_{H^1(\Omega)}+ \|r_2\|_{H^1(\Omega)})\|v_0\|_{L^2(\p \Omega)}\le CN^{-1/2}\\
&=CM^{-1/2}\sqrt{\omega(M^{-1})}.
\end{aligned}
\end{equation}

Using \eqref{eq_7_17}, we get
\begin{equation}
\label{eq_7_22}
\begin{aligned}
M^{n-1}&\bigg|\int_\Omega ( (A_2)^2-(A_1)^2+ q_2-q_1 ) u_1\overline{u_2}dx\bigg|\\
&\le  CM^{n-1}\|u_1\|_{L^2(\Omega)}\|u_2\|_{L^2(\Omega)}\le CN^{-1}=CM^{-1}\omega(M^{-1}). 
\end{aligned}
\end{equation}

Hence, it follows from \eqref{eq_8_1_identity} together with \eqref{eq_7_18},  \eqref{eq_7_19},  \eqref{eq_7_20},  \eqref{eq_7_21},  \eqref{eq_7_21_1}, and  \eqref{eq_7_22} as $M\to \infty$ that 
$(A_2-A_1)(0)\cdot\tau=0$. 
This completes the proof.

\end{proof}

\section{Transmission scattering problem}

\label{sec_scat_prob}

\subsection{Direct scattering problem}

\label{subsec_direct_scat}

Let $D\subset\R^n$, $n\ge 3$, be a bounded open set with Lipschitz boundary such that $D^+:=\R^n\setminus\overline{D}$ is connected.   Set also $D^-=D$, and let  $A^\pm\in W^{1,\infty}(D^\pm,\C^n)$,  $q^\pm\in L^\infty(D^\pm,\C)$, $a,b\in C^{1,1}(\overline{D},\R)$,  $c\in C(\overline{D},\R)$, $f^\pm\in \tilde H^{-1}(D^\pm)$, $g_0\in H^{1/2}(\p D)$, and $g_1\in H^{-1/2}(\p D)$.  In what follows, we shall assume that $A^+$, $q^+$ and $f^+$ are compactly supported. 

Let $k>0$ and for $(u^+,u^-)\in H^1_{\textrm{loc}}(D^+)\times H^1(D)$, we consider the following  inhomogeneous transmission problem,
\begin{equation}
\label{eq_sc_inhom_tr}
\begin{aligned}
&(\mathcal{L}_{A^+,q^+}(x,D_x)-k^2)u^+=f^+\quad \textrm{in}\quad D^+,\\
&(\mathcal{L}_{A^-,q^-}(x,D_x)-k^2)u^-=f^-\quad\textrm{in}\quad D^-,\\
&u^+=au^-+g_0\quad\textrm{on}\quad \p D,\\
&(\p_\nu+iA^+\cdot\nu)u^+=b(\p_\nu+iA^-\cdot\nu)u^-+cu^-+g_1\quad\textrm{on}\quad \p D,\\
&(\p_r-ik)u^+=o(r^{-(n-1)/2}), \quad\textrm{as} \quad r=|x|\to\infty,
\end{aligned}
\end{equation}
and the corresponding homogeneous transmission problem,
\begin{equation}
\label{eq_sc_hom_tr}
\begin{aligned}
&(\mathcal{L}_{A^+,q^+}(x,D_x)-k^2)u^+=0\quad \textrm{in}\quad D^+,\\
&(\mathcal{L}_{A^-,q^-}(x,D_x)-k^2)u^-=0\quad\textrm{in}\quad D^-,\\
&u^+=au^-\quad\textrm{on}\quad \p D,\\
&(\p_\nu+iA^+\cdot\nu)u^+=b(\p_\nu+iA^-\cdot\nu)u^-+cu^-\quad\textrm{on}\quad \p D,\\
&(\p_r-ik)u^+=o(r^{-(n-1)/2}), \quad\textrm{as} \quad r=|x|\to\infty. 
\end{aligned}
\end{equation}

Following \cite{Valdivia_2004}, in order to study the solvability of the inhomogeneous transmission problem \eqref{eq_sc_inhom_tr}  we shall use the Lax-Phillips method, see \cite{Isakov_book, KrupLassasUhlmann_slab, Lax_Phillips},  and to that end, we shall need the following assumption. 

\begin{itemize}
\item[\textbf{(B)}] If $(u^+,u^-)\in H^1_{\textrm{loc}}(D^+)\times H^1(D)$ solves the homogeneous transmission problem \eqref{eq_sc_hom_tr}  then $(u^+,u^-)=0$ in $D^+\times D^-$.  
\end{itemize}

The following result shows that the assumption (B) is satisfied under some suitable conditions on the potentials and the transmission coefficients. Notice that these conditions are similar to those occurring in \cite{Valdivia_2004}, when studying the homogeneous transmission problem for the Schr\"odinger operator without a magnetic potential.  

\begin{prop}
\label{prop_uni_9_1}
Let $A^\pm$ be real-valued, $\emph{\Im} q^\pm\le 0$ in $D^\pm$, $a,b>0$ on $\overline{D}$, and $ab$ is constant on each connected component of  $\overline{D}$. Then the assumption \emph{(B)} is satisfied. 

\end{prop}

\begin{proof}

Let $(u^+,u^-)\in H^1_{\textrm{loc}}(D^+)\times H^1(D)$ be a solution of the homogeneous transmission problem \eqref{eq_sc_hom_tr}. Then using the second equation in \eqref{eq_sc_hom_tr} and the Green formula \eqref{eq_first_green}, we get
\begin{equation}
\label{eq_8_1}
\begin{aligned}
0=&\int_{D}ab\mathcal{L}_{A^-,q^-}u^-\overline{u^-}dx-\int_D abk^2|u^-|^2dx\\
=&\int_D ab|\nabla u^-|^2dx+\int_D\overline{u^-}\nabla u^- \cdot\nabla(ab)dx +i\int_D A^-\cdot(u^-\nabla\overline{u^-}-\overline{u^-}\nabla u^-)ab dx\\
&+i\int_{D} (A^-\cdot\nabla (ab))|u^-|^2dx+ \int_{D} ((A^-)^2+q^--k^2)ab|u^-|^2dx\\
&-(b(\p_\nu+iA^-\cdot\nu)u^-,au^-)_{(H^{-1/2},H^{1/2})(\p D)}.
\end{aligned}
\end{equation}

Let $R>0$ be large so that 
\[
\supp(A^+), \supp(q^+)\subset B_R, \quad \overline{D}\subset B_R.
\]
Here $B_R$ denotes the open ball, centered at the origin with radius $R$.   Using the first equation in \eqref{eq_sc_hom_tr} and  the Green formula \eqref{eq_first_green}, we have 
\begin{equation}
\label{eq_8_2}
\begin{aligned}
0=&\int_{B_R\setminus\overline{D}}\mathcal{L}_{A^+,q^+}u^+\overline{u^+}dx-\int_{B_R\setminus\overline{D}} k^2|u^+|^2dx=\int_{B_R\setminus\overline{D}} |\nabla u^+|^2dx\\
&+i\int_{B_R\setminus\overline{D}} A^+\cdot(u^+\nabla\overline{u^+}-\overline{u^+}\nabla u^+)dx+ \int_{B_R\setminus\overline{D}}((A^+)^2+q^+-k^2)|u^+|^2dx\\
&+((\p_\nu+iA^+\cdot\nu)u^+,u^+)_{(H^{-1/2},H^{1/2})(\p D)} - (\p_\nu u^+,u^+)_{(H^{-1/2},H^{1/2})(\p B_R)}.
\end{aligned}
\end{equation}

Adding \eqref{eq_8_1} and \eqref{eq_8_2}, using the assumptions of the proposition and the transmission conditions in \eqref{eq_sc_hom_tr}, and taking the imaginary part, we obtain that
\begin{equation}
\label{eq_8_3}
\Im (\p_\nu u^+,u^+)_{(H^{-1/2},H^{1/2})(\p B_R)}=\int_{B_R\setminus\overline{D}}\Im q^+ |u^+|^2dx+ \int_{D} ab\Im q^-|u^-|^2dx\le 0. 
\end{equation}

By the choice of the ball $B_R$, we have that $u^+$ satisfies the equation $(-\Delta -k^2) u^+=0$ in $\R^n\setminus\overline{B_R}$, and the Sommerfeld radiation condition. Then $u^+$ has the following asymptotic behavior, 
\begin{equation}
\label{eq_8_4}
u^+(x)=a(\theta)\frac{e^{ik|x|}}{|x|^{(n-1)/2}}+\mathcal{O}\bigg(\frac{1}{|x|^{(n+1)/2}}\bigg),\quad \theta=\frac{x}{|x|},
\end{equation}
as $|x|\to \infty$, see \cite{Col_Kress_book, Odell_2006}.  Substituting \eqref{eq_8_4} into \eqref{eq_8_3}, we get
\begin{equation}
\label{eq_8_5}
\begin{aligned}
\Im \int_{|x|=R} \bigg(ik\frac{|a(\theta)|^2}{|x|^{n-1}}+&\mathcal{O}\bigg(\frac{1}{|x|^n}\bigg)\bigg)dS_R\\
&=\Im \int_{|x|=1} \bigg(ik |a(\theta)|^2+\mathcal{O}\bigg(\frac{1}{R}\bigg)\bigg)dS_1\le 0,
\end{aligned}
\end{equation}
where $dS_R$ and $dS_1$ are the surface measures on the spheres $|x|=R$ and $|x|=1$, respectively.  
Letting $R\to \infty$ in \eqref{eq_8_5}, we obtain that 
\[
\int_{|x|=1} |a(\theta)|^2dS_1=0,
\]
and therefore, $a(\theta)=0$. By Rellich's theorem, $u^+=0$ in $\R^n\setminus\overline{B_R}$, see \cite{Horm_1973}. 
As $u^+$ satisfies the first  equation in \eqref{eq_sc_hom_tr} and  $D^+$ is connected, by unique continuation,  $u^+=0$ in $D^+$.  The transmission conditions in \eqref{eq_sc_hom_tr} imply that $u^-=0$ and $\p_\nu u^-=0$ on $\p D$, and therefore, using the second equation in \eqref{eq_sc_hom_tr} and unique continuation, we get $u^-=0$ in $D$. The proof is complete.

\end{proof}

In the following result, we establish the existence of solutions to the transmission scattering problem \eqref{eq_sc_inhom_tr}  in the non-selfadjoint case. 

\begin{prop}

\label{prop_8_2}

Let $k>0$,  the assumption \emph{(B)} be satisfied, and let $a,b>0$ on $\overline{D}$. Then for any $f^+\in \tilde H^{-1}(D^+)$ with compact support, $f^-\in \tilde H^{-1}(D^-)$,  $g_0\in H^{1/2}(\p D)$, and $g_1\in H^{-1/2}(\p D)$, the inhomogeneous transmission problem \eqref{eq_sc_inhom_tr}
has a unique solution $(u^+,u^-)\in H^1_{\emph{\textrm{loc}}}(D^+)\times H^1(D)$. 

\end{prop}

\begin{proof}

Let $R>0$ be large and $S>R$ so that $\supp(A^+), \supp(q^+)\subset B_R$,  $\overline{D}\subset B_R$, and $\supp(f^+)\subset \overline{B_S}$.  Let $z\in \C$ be such that $\Im z\ne 0$ and the following homogeneous transmission problem 
\begin{equation}
\label{eq_sc_hom_1}
\begin{aligned}
&(\mathcal{L}_{A^+,q^+}-z)u^+=0\quad \textrm{in}\quad B_S\setminus\overline{D},\\
&(\mathcal{L}_{A^-,q^-}-z)u^-=0\quad\textrm{in}\quad D,\\
&u^+=au^-\quad\textrm{on}\quad \p D,\\
&(\p_\nu+iA^+\cdot\nu)u^+=b(\p_\nu+iA^-\cdot\nu)u^-+cu^-\quad\textrm{on}\quad \p D,\\
&u^+=0\quad\textrm{on}\quad \p B_S,
\end{aligned}
\end{equation}
has only the trivial solution. The existence of such $z$ follows from the fact that the transmission problem \eqref{eq_sc_hom_1} has only the trivial solution if and only if the following transmission problem 
\begin{equation}
\label{eq_sc_hom_2}
\begin{aligned}
&(\mathcal{L}_{A^+,q^+}-z)w^+=0\quad \textrm{in}\quad B_S\setminus\overline{D},\\
&b\mathcal{L}_{A^-,q^-}(a^{-1}w^-)-zba^{-1}w^-=0\quad\textrm{in}\quad D,\\
&w^+=w^-\quad\textrm{on}\quad \p D,\\
&(\p_\nu+iA^+\cdot\nu)w^+=b(\p_\nu+iA^-\cdot\nu)(a^{-1}w^-)+ca^{-1}w^-\quad\textrm{on}\quad \p D,\\
&w^+=0\quad\textrm{on}\quad \p B_S,
\end{aligned}
\end{equation}
has only the trivial solution. Furthermore, $(w^+,w^-)\in H^1(B_S\setminus\overline{D})\times H^1(D)$ solves the transmission problem \eqref{eq_sc_hom_2} if and only if 
\[
\Phi_e(w,v):=\Phi(w,v)-z\int_{B_S\setminus\overline{D}} w^+\overline{v^+}dx-z\int_{D} ba^{-1} w^-\overline{v^-}dx=0,
\]
for any $v\in H^1_0(B_S)$.  Here the sesquilinear form $\Phi$ is given by \eqref{eq_form_main_var} with $\Omega$ replaced by $B_S$.  It follows that the form $\Phi_e:H^1_0(B_S)\times H^1_0(B_S)\to \C$ is bounded, and \eqref{eq_coercivity_estim} implies that  for $\Re z<0$ with $|\Re z|$ large enough, 
\[
\Re \Phi_e(w,w)\ge c\|w\|^2_{H^1(B_S)}, \quad c>0,
\]
for all $w\in H^1_0(B_S)$.  Thus, the bounded linear operator $\mathcal{B}:H^1_0(B_S)\to H^{-1}(B_S)$, defined by 
$(\mathcal{B}w,v)_{H^{-1}(B_S),H^1_0(B_S)}=\Phi_e(w,v)$ for $w,v\in H^1_0(B_S)$, has a bounded inverse for $z\in \C$ such that $\Re z<0$ and $|\Re z|$ is large enough, see \cite[Lemma 2.32]{McLean_book}.  
Hence, for such $z$, both transmission problems \eqref{eq_sc_hom_1} and \eqref{eq_sc_hom_2} have only the trivial solution, and therefore, there exists $z\in \C$ with $\Im z\ne 0$ so that the homogeneous transmission problems \eqref{eq_sc_hom_1} has only the trivial solution.  

Let us fix a choice of such a $z\in \C$. Then the inhomogeneous transmission problem,
\begin{align*}
&(\mathcal{L}_{A^+,q^+}-z)u^+=h^+\quad \textrm{in}\quad B_S\setminus\overline{D},\\
&(\mathcal{L}_{A^-,q^-}-z)u^-=h^-\quad\textrm{in}\quad D,\\
&u^+=au^-+p_0\quad\textrm{on}\quad \p D,\\
&(\p_\nu+iA^+\cdot\nu)u^+=b(\p_\nu+iA^-\cdot\nu)u^-+cu^-+p_1\quad\textrm{on}\quad \p D,\\
&u^+=p \quad\textrm{on}\quad \p B_S,
\end{align*}
has a unique solution $(u^+,u^-)\in H^1(B_S\setminus\overline{D})\times H^1(D)$ for any $h^+\in\tilde H^{-1}(B_S\setminus\overline{D})$, $h^-\in \tilde H^{-1}(D)$, $p_0\in H^{1/2}(\p D)$, $p_1\in H^{-1/2}(\p D)$, and $p\in H^{1/2}(\p B_S)$.

When solving the inhomogeneous transmission problem \eqref{eq_sc_inhom_tr}, we can always assume that $g_0=0$ and $g_1=0$. Indeed, let $(U^+,U^-)\in H^1(B_S\setminus \overline{D})\times H^1(D)$ be the solution to the problem 
\begin{align*}
&(\mathcal{L}_{A^+,q^+}-z)U^+=0\quad \textrm{in}\quad B_S\setminus\overline{D},\\
&(\mathcal{L}_{A^-,q^-}-z)U^-=0\quad\textrm{in}\quad D,\\
&U^+=aU^-+g_0\quad\textrm{on}\quad \p D,\\
&(\p_\nu+iA^+\cdot\nu)U^+=b(\p_\nu+iA^-\cdot\nu)U^-+cU^-+g_1\quad\textrm{on}\quad \p D,\\
&U^+=0 \quad\textrm{on}\quad \p B_S,
\end{align*}
and set
\[
\tilde U^+=\begin{cases} U^+, & B_S\setminus\overline{D},\\
0, & \R^n\setminus\overline{B_S}, 
\end{cases}\in H^1(\R^n\setminus\overline{D}).
\]
If $(u^+,u^-)\in H^1_{\textrm{loc}}(D^+)\times H^1(D^-)$ is a solution to the problem,
\begin{equation}
\label{eq_8_6}
\begin{aligned}
&(\mathcal{L}_{A^+,q^+}-k^2)u^+=f^++\tilde  f^+\quad \textrm{in}\quad D^+,\\
&(\mathcal{L}_{A^-,q^-}-k^2)u^-=f^-+\tilde f^-\quad\textrm{in}\quad D^-,\\
&u^+=au^-\quad\textrm{on}\quad \p D,\\
&(\p_\nu+iA^+\cdot\nu)u^+=b(\p_\nu+iA^-\cdot\nu)u^-+cu^-\quad\textrm{on}\quad \p D,\\
&(\p_r-ik)u^+=o(r^{-(n-1)/2}), \quad\textrm{as} \quad r=|x|\to\infty,
\end{aligned}
\end{equation}
with  $ \tilde f^+=-(\mathcal{L}_{A^+,q^+}-k^2)\tilde U^+$ and $\tilde f^-=(k^2-z)U^-$, then $(u^++\tilde U^+, u^-+U^-)$ solves the problem \eqref{eq_sc_inhom_tr}.  Here $\tilde f^+\in \tilde H^{-1}(B_S\setminus\overline{D})$. This follows from the fact that $(\mathcal{L}_{A^+,q^+}-z)\tilde U^+\in H^{-1}(\R^n\setminus\overline{D})$ and $\supp((\mathcal{L}_{A^+,q^+}-z)\tilde U^+)\subset\p B_S$, and therefore, $(\mathcal{L}_{A^+,q^+}-z)\tilde U^+$ can be extended to an element of $H^{-1}(\R^n)$.

In what follows we shall assume that $g_0=g_1=0$ in \eqref{eq_sc_inhom_tr}. 
In order to show the existence of a solution of the problem \eqref{eq_sc_inhom_tr},  we shall use the Lax-Phillips method, see \cite{Isakov_book, Lax_Phillips}. To that end let $\phi\in C^\infty_0(B_S)$, $0\le \phi\le 1$, and $\phi=1$ in $\overline{B_R}$. Let $h^+\in \tilde H^{-1}(B_S\setminus\overline{D})$, $h^-\in \tilde H^{-1}(D)$,  and let 
$(w^+,w^-)\in H^1(B_S\setminus \overline{D})\times H^1(D)$ be the solution to the problem 
\begin{equation}
\label{eq_8_7_0}
\begin{aligned}
&(\mathcal{L}_{A^+,q^+}-z)w^+=h^+\quad \textrm{in}\quad B_S\setminus\overline{D},\\
&(\mathcal{L}_{A^-,q^-}-z)w^-=h^-\quad\textrm{in}\quad D,\\
&w^+=aw^-\quad\textrm{on}\quad \p D,\\
&(\p_\nu+iA^+\cdot\nu)w^+=b(\p_\nu+iA^-\cdot\nu)w^-+cw^-\quad\textrm{on}\quad \p D,\\
&w^+=0 \quad\textrm{on}\quad \p B_S.
\end{aligned}
\end{equation}
Let $v\in H^1_{\textrm{loc}}(D^+)$ be the unique solution of the Dirichlet problem,
\begin{equation}
\label{eq_8_7_1}
\begin{aligned}
&(-\Delta-k^2)v=h^+\quad \textrm{in}\quad D^+,\\
& v=0 \quad \textrm{on}\quad \p D,\\
&(\p_r-ik)v=o(r^{-(n-1)/2}), \quad\textrm{as} \quad r=|x|\to\infty,
\end{aligned}
\end{equation}
see \cite[Theorem 9.11]{McLean_book}. 

We look for a solution to \eqref{eq_sc_inhom_tr}  in the form,
\begin{equation}
\label{eq_8_7}
u^+=\phi w^++(1-\phi)v,\quad u^-=w^-. 
\end{equation}
It is clear that $(u^+,u^-)$, given by \eqref{eq_8_7}, solves the transmission problem \eqref{eq_sc_inhom_tr}, if 
\begin{align*}
(\mathcal{L}_{A^+,q^+}-k^2)(\phi w^++(1-\phi)v)&=h^++\phi(z-k^2)w^+ +[\mathcal{L}_{A^+,q^+},\phi](w^+-v)\\
&= f^+\quad \textrm{in}\quad D^+,\\
(\mathcal{L}_{A^-,q^-}-k^2) w^-=h^-+(z-k^2)w^-&=f^-\quad \textrm{in}\quad D^-.
\end{align*}

Thus, given $(f^+,f^-)\in \tilde H^{-1}(B_S\setminus\overline{D})\times \tilde H^{-1}(D)$, we would like to find $(h^+,h^-)\in \tilde H^{-1}(B_S\setminus\overline{D})\times \tilde H^{-1}(D)$ such that 
\begin{equation}
\label{eq_8_8}
(h^+,h^-)+T(h^+,h^-)=(f^+,f^-),
\end{equation}
where 
\begin{align*}
&T:\tilde H^{-1}(B_S\setminus\overline{D})\times \tilde H^{-1}(D)\to \tilde H^{-1}(B_S\setminus\overline{D})\times \tilde H^{-1}(D),\\
&T(h^+,h^-)= \big(\phi(z-k^2)w^+ +[\mathcal{L}_{A^+,q^+},\phi](w^+-v), (z-k^2)w^-\big). 
\end{align*}

Let us first check that the operator $T$ is compact. Since $(w^+,w^-)$ is the unique solution to the transmission problem \eqref{eq_8_7_0}, the estimate \eqref{eq_energy_estimates_unique} implies that the map
\[
\tilde H^{-1}(B_S\setminus\overline{D})\times \tilde H^{-1}(D) \to H^{1}(B_S\setminus\overline{D})\times H^1(D), \quad (h^+,h^-)\mapsto (w^+,w^-),
\]
is continuous.  As $v$ is the unique solution to \eqref{eq_8_7_1}, the map 
\[
\tilde H^{-1}(B_S\setminus\overline{D})\to H^{1}_{\textrm{loc}}(D^+), \quad h^+\mapsto v,
\]
is continuous. Furthermore, the commutator is given by
\[
[\mathcal{L}_{A^+,q^+},\phi]=-2\nabla\phi\cdot\nabla-\Delta\phi - 2i A\cdot\nabla \phi,
\]
and therefore, 
\[
T: \tilde H^{-1}(B_S\setminus\overline{D})\times \tilde H^{-1}(D) \to L^2(B_S\setminus\overline{D})\times H^1(D)\hookrightarrow \tilde H^{-1}(B_S\setminus\overline{D})\times \tilde H^{-1}(D),
\]
which shows the compactness of the operator $T$.  Here we have also used that the boundary of $D$ is Lipschitz. 

Thus, the operator $I+T$ is Fredholm of index zero and therefore, to show the existence of a solution of  \eqref{eq_8_8}, it suffices to check that $(f^+,f^-)=(0,0)$ implies that $(h^+,h^-)=(0,0)$. 

Assume now that $(f^+,f^-)=(0,0)$. Then the assumption (B) implies that 
\begin{equation}
\label{eq_8_9}
\phi w^++(1-\phi)v=0\quad\textrm{in}\quad D^+,\quad w^-=0\quad\textrm{in}\quad D.
\end{equation} 
Furthermore, we get
\begin{equation}
\label{eq_8_9_1}
h^++\phi(z-k^2)w^+ +[\mathcal{L}_{A^+,q^+},\phi](w^+-v)=0 \quad\textrm{in}\quad D^+,\quad h^-=0\quad\textrm{in}\quad D.
\end{equation}

Let $\Sigma=\{x\in B_S: \phi=1\}$.  Then \eqref{eq_8_9} implies that $w^+=0$ in $\Sigma$, and therefore, it follows from \eqref{eq_8_9_1} that $h^+=0$ in $\Sigma$. 

Consider now the set  $\Sigma^c=\{x\in B_S: \phi\ne 1\}$.  By \eqref{eq_8_9}, we have $w^+=0$ in $B_R\setminus\overline{D}$, and  $v=\phi(v-w^+)$ in $D^+$. Using that $\supp(A^+),\supp(q^+)\subset B_R$, we get
\[
(-\Delta-z)(v-w^+)=(k^2-z)\phi (v-w^+)\quad\textrm{in}\quad B_S\setminus\overline{D}. 
\]
Furthermore, $w^+=0$ on $\p B_S\cup \p D$.  As $\phi=0$ near $\p B_S$, \eqref{eq_8_9} yields $v=0$ on $\p B_S$, and we know that $v=0$ on $\p D$.  Thus, we have
\begin{equation}
\label{eq_8_9_2}
\int_{B_S\setminus\overline{D}}(k^2-z)\phi |v-w^+|^2dx=\int_{B_S\setminus\overline{D}} (|\nabla (v-w^+)|^2-z|v-w^+|^2)dx.
\end{equation}
Taking the imaginary part in \eqref{eq_8_9_2}, we get 
\[
\int_{B_S\setminus\overline{D}}(\phi-1) |v-w^+|^2dx=0,
\]
and therefore, $v-w^+=0$ in $\Sigma^c$. It follows from \eqref{eq_8_9} that $w^+=0$ in $\Sigma^c$, and thus, \eqref{eq_8_9_1} yields that $h^+=0$ in $\Sigma^c$. Hence, $(h^+,h^-)=(0,0)$. 
 The proof is complete.

\end{proof}

Let $k>0$ and $\xi\in \mathbb{S}^{n-1}:=\{\xi\in\R^n:|\xi|=1\}$. Consider the following scattering transmission problem, 
\begin{equation}
\label{eq_8_10}
\begin{aligned}
&(\mathcal{L}_{A^+,q^+}-k^2)u^+=0\quad \textrm{in}\quad D^+,\\
&(\mathcal{L}_{A^-,q^-}-k^2)u^-=0\quad\textrm{in}\quad D^-,\\
&u^+=au^-\quad\textrm{on}\quad \p D,\\
&(\p_\nu+iA^+\cdot\nu)u^+=b(\p_\nu+iA^-\cdot\nu)u^-+cu^-\quad\textrm{on}\quad \p D,\\
& u^+(x;\xi,k)=e^{ikx\cdot\xi}+u_0^+(x;\xi,k),\\
&(\p_r-ik)u_0^+=o(r^{-(n-1)/2}), \quad\textrm{as} \quad r=|x|\to\infty. 
\end{aligned}
\end{equation}
We have the following consequence of Propositions \ref{prop_uni_9_1} and \ref{prop_8_2}. 

\begin{cor}
\label{cor_scatt_8_3}
Let $A^\pm\in W^{1,\infty}(D^\pm,\R^n)$, $q^\pm\in L^\infty(D^\pm,\C)$ be such that $\emph{\Im} q^\pm\le 0$ in $D^\pm$. Assume that $A^+$ and $q^+$ are compactly supported.  Let $a,b\in C^{1,1}(\overline{D},\R)$, $c\in C(\overline{D},\R)$, be such that $a,b>0$ in $\overline{D}$ and $ab$ is constant on each connected component of  $\overline{D}$.  Assume furthermore that $D^+$ is connected. Then the problem \eqref{eq_8_10} has a unique solution $(u^+,u^-)\in H^1_{\emph{\textrm{loc}}}(D^+)\times H^1(D^-)$. 
\end{cor}

\subsection{Inverse scattering problem. Proof of Theorem \ref{thm_main_2}}

\label{subsec_inv_scat}

In order to prove Theorem \ref{thm_main_2}, we shall follow  closely  the method of  \cite{Isakov_2008}.  
The key idea of \cite{Isakov_2008}  is  to  approximate  solutions of the transmission problem on a large ball by scattering solutions. 

Let us start with this approximation result. Let $k>0$ be fixed, $D\subset\R^n$, $n\ge 3$, be a  bounded open set with Lipschitz boundary such that $\R^n\setminus \overline{D}$ is connected. Let $A^+\in W^{1,\infty}(\R^n, \R^n)$, $q^+\in L^\infty(\R^n,\R)$ be compactly supported, and let  $A^-\in W^{1,\infty}(D,\R^n)$, and $q^-\in L^\infty(D,\R)$.  Let $a, b\in C^{1,1}(\overline{D},\R)$, and $c\in C(\overline{D},\R)$, be such that  $a, b>0$ on $\overline{D}$, $ab=1$ on $\overline{D}$.

Let $B\subset \R^n$ be an open ball such that $D\subset\subset B$, $\supp(A^+), \supp(q^+)\subset B$, the homogeneous transmission problem in $B$,
\begin{equation}
\label{eq_inv_sc_1}
\begin{aligned}
&(\mathcal{L}_{A^+,q^+}-k^2)w^+=0\quad \textrm{in}\quad B\setminus\overline{D},\\
&(\mathcal{L}_{A^-,q^-}-k^2)w^-=0\quad\textrm{in}\quad D,\\
&w^+=a w^-\quad\textrm{on}\quad \p D,\\
&(\p_\nu+iA^+\cdot\nu)w^+=b (\p_\nu+iA^-\cdot\nu)w^-+c w^-\quad\textrm{on}\quad \p D,
\end{aligned}
\end{equation}
with
\[
w^+=0\quad\textrm{on}\quad \p B,
\]
has only the trivial solution,  and furthermore, $k^2$ is not an eigenvalue of the Dirichlet Laplacian in $B$.  We refer to Corollary \ref{prop_domain_pert_2} for the existence of such a ball.

Setting 
\[
W(B):=\{w^+:(w^+,w^-)\in H^1(B\setminus\overline{D})\times H^1(D)\textrm{ satisfies }\eqref {eq_inv_sc_1}\},
\]
and 
\[
W_{\textrm{sc}}:=\{u^+: (u^+,u^-)\in H^1_{\textrm{loc}}(\R^n\setminus\overline{D})\times H^1(D)\textrm{ solves }\eqref{eq_8_10}\textrm{ with some }\xi\in \mathbb{S}^{n-1}\},
\]
we have the following Runge type approximation result.  The proof of this result follows closely \cite[Lemma 2.2]{Isakov_2008}. 

\begin{lem}
\label{lem_inv_sc_density}
The space $W_{\emph{\textrm{sc}}}$ is dense in $W(B)$ in the $H^1(B\setminus\overline{D})$--topology. 
\end{lem}

\begin{proof}
By the Hahn-Banach theorem, we need to show that  for any  $f^+\in \tilde H^{-1}(B\setminus\overline{D})$ such that 
\begin{equation}
\label{eq_inv_sc_2_0}
(f^+,u^+)_{(\tilde H^{-1},H^1)(B\setminus\overline{D})}=0,
\end{equation}
for any $u^+\in W_{\textrm{sc}}$, we  have
\begin{equation}
\label{eq_inv_sc_2_-1}
(f^+,w^+)_{(\tilde H^{-1},H^1)(B\setminus\overline{D})}=0,
\end{equation}
for any $w^+\in W(B)$.  Consider the following transmission problem,
\begin{equation}
\label{eq_inv_sc_2}
\begin{aligned}
&(\mathcal{L}_{-A^+,q^+}-k^2)\overline{v^+}=\overline{f^+}\quad \textrm{in}\quad \R^n\setminus\overline{D},\\
&(\mathcal{L}_{-A^-,q^-}-k^2)\overline{v^-}=0\quad\textrm{in}\quad D,\\
&\overline{v^+}=b^{-1}\overline{v^-}\quad\textrm{on}\quad \p D,\\
&(\p_\nu-iA^+\cdot\nu)\overline{v^+}=a^{-1}(\p_\nu-iA^-\cdot\nu)\overline{v^-}+ca^{-1}b^{-1}\overline{v^-}\quad\textrm{on}\quad \p D,\\
&(\p_r-ik)\overline{v^+}=o(r^{-(n-1)/2}), \quad\textrm{as} \quad r=|x|\to\infty. 
\end{aligned}
\end{equation}
It follows from Proposition \ref{prop_8_2} that  the problem \eqref{eq_inv_sc_2} has a unique solution $(\overline{v^+}, \overline{v^-})\in H^1_{\textrm{loc}}(\R^n\setminus\overline{D})\times H^1(D)$. 

Let $u^+\in W_{\textrm{sc}}$. Then by the second Green formula \eqref{eq_first_green_adjoint}, we get
\begin{equation}
\label{eq_inv_sc_3}
\begin{aligned}
((\mathcal{L}_{A^+,q^+}&-k^2)u^+, v^+)_{L^2(B\setminus\overline{D})}- ((\p_\nu+iA^+\cdot\nu)u^+,v^+)_{(H^{-1/2},H^{1/2})(\p D)}\\
&+ (\p_\nu u^+,v^+)_{(H^{-1/2},H^{1/2})(\p B)}=(u^+,f^+)_{(H^1,\tilde H^1)(B\setminus\overline{D})}\\
&- (u^+,(\p_\nu+iA^+\cdot\nu)v^+)_{(H^{1/2},H^{-1/2})(\p D)}+ (u^+,\p_\nu v^+)_{(H^{1/2},H^{-1/2})(\p B)},
\end{aligned}
\end{equation}
and 
\begin{equation}
\label{eq_inv_sc_4}
\begin{aligned}
((\mathcal{L}_{A^-,q^-}&-k^2)u^-, v^-)_{L^2(D)}+ ((\p_\nu+iA^-\cdot\nu)u^-,v^-)_{(H^{-1/2},H^{1/2})(\p D)}\\
&=(u^-, (\mathcal{L}_{A^-,q^-}-k^2) v^-)_{L^2(D)} +(u^-,(\p_\nu+iA^-\cdot\nu)v^-)_{(H^{1/2},H^{-1/2})(\p D)}.
\end{aligned}
\end{equation}
Adding \eqref{eq_inv_sc_3} and \eqref{eq_inv_sc_4}, and using the transmission conditions in \eqref{eq_8_10},  \eqref{eq_inv_sc_2}, and \eqref{eq_inv_sc_2_0}, we get 
\begin{equation}
\label{eq_inv_sc_5}
(\p_\nu u^+,v^+)_{(H^{-1/2},H^{1/2})(\p B)}- (u^+,\p_\nu v^+)_{(H^{1/2},H^{-1/2})(\p B)}=(u^+,f^+)_{(H^1,\tilde H^1)(B\setminus\overline{D})}=0. 
\end{equation}

We shall next show that $v^+=0$ outside the ball $B$. Indeed, $v^+$ and $u_0^+(x;\xi,k)=u^+(x;\xi,k)-e^{ikx\cdot\xi}$ satisfy the Helmholtz equation outside the ball $B$,
\begin{equation}
\label{eq_inv_sc_6}
(-\Delta-k^2)v^+=0\quad\textrm{in}\quad \R^n\setminus\overline{B},\quad  (-\Delta-k^2)u_0^+=0\quad\textrm{in}\quad \R^n\setminus\overline{B}.
\end{equation}
Let $R>0$ be so large that $B\subset B_R$, where $B_R$ is an open ball of radius $R$, centered at the origin. Multiplying the second equation in \eqref{eq_inv_sc_6}
by $\overline{v^+}$, and integrating over $B_R\setminus\overline{B}$, we obtain that
\begin{equation}
\label{eq_inv_sc_7}
\begin{aligned}
(u_0^+,&\p_\nu v^+)_{(H^{1/2},H^{-1/2})(\p B)}-(\p_\nu u^+_0,v^+)_{(H^{1/2},H^{-1/2})(\p B)}\\
&=(u_0^+,\p_\nu v^+)_{(H^{1/2},H^{-1/2})(\p B_R)}-(\p_\nu u^+_0,v^+)_{(H^{1/2},H^{-1/2})(\p B_R)}.
\end{aligned}
\end{equation}
As $u_0^+$ and $\overline{v^+}$ satisfy the Sommerfeld radiation condition and the Helmholtz equation outside $B$, cf. \eqref{eq_inv_sc_6},  they 
 have the following asymptotic behaviors 
\begin{equation}
\label{eq_inv_sc_7_1}
\begin{aligned}
&u^+_0(x;\xi,k)=a(\theta,\xi,k)\frac{e^{ik|x|}}{|x|^{(n-1)/2}}+\mathcal{O}\bigg(\frac{1}{|x|^{(n+1)/2}}\bigg),\quad \theta=\frac{x}{|x|},\\
&\overline{v^+}(x)=b(\theta,k)\frac{e^{ik|x|}}{|x|^{(n-1)/2}}+\mathcal{O}\bigg(\frac{1}{|x|^{(n+1)/2}}\bigg),
\end{aligned}
\end{equation}
as $|x|\to \infty$, see \cite{Col_Kress_book, Odell_2006}.  Substituting \eqref{eq_inv_sc_7_1} into the right hand side of \eqref{eq_inv_sc_7}, we get
\begin{equation}
\label{eq_inv_sc_8}
\begin{aligned}
(u_0^+,\p_\nu v^+)_{(H^{1/2},H^{-1/2})(\p B)}-(\p_\nu u^+_0,v^+)_{(H^{1/2},H^{-1/2})(\p B)}\\
=\int_{|x|=R} \mathcal{O}(|x|^{-n}) dS_R=\mathcal{O}(R^{-1}) \int_{|x|=1} dS_1,
\end{aligned}
\end{equation}
where $dS_R$ and $dS_1$ are the surface measures on the spheres $|x|=R$ and $|x|=1$, respectively. Letting $R\to \infty$ in \eqref{eq_inv_sc_8}, we conclude that 
\begin{equation}
\label{eq_inv_sc_9}
(u_0^+,\p_\nu v^+)_{(H^{1/2},H^{-1/2})(\p B)}-(\p_\nu u^+_0,v^+)_{(H^{1/2},H^{-1/2})(\p B)}=0.
 \end{equation}
 Substituting $u^+=e^{ikx\cdot\xi}+u^+_0$ into \eqref{eq_inv_sc_5}, and using \eqref{eq_inv_sc_9}, we get
 \begin{equation}
\label{eq_inv_sc_10}
 (e^{ikx\cdot\xi},\p_\nu v^+)_{(H^{1/2},H^{-1/2})(\p B)}-(\p_\nu e^{ikx\cdot\xi},v^+)_{(H^{1/2},H^{-1/2})(\p B)}=0.
 \end{equation}
 Recalling that $k^2$ is not an eigenvalue of the Dirichlet Laplacian in $B$, let $v_0\in H^1(B)$ be the unique solution to the problem,
 \begin{align*}
 (-\Delta-k^2)v_0&=0\quad \textrm{in}\quad B,\\
 v_0&=v^+\quad \textrm{on}\quad\p B. 
 \end{align*} 
 As $e^{ikx\cdot\xi}$ also satisfies the Helmholtz equation, integrating over $B$, we have
  \begin{equation}
\label{eq_inv_sc_11}
(\p_\nu e^{ikx\cdot\xi},v_0)_{(H^{1/2},H^{-1/2})(\p B)}= (e^{ikx\cdot\xi},\p_\nu v_0)_{(H^{1/2},H^{-1/2})(\p B)}.
 \end{equation}
 Combining \eqref{eq_inv_sc_10} and \eqref{eq_inv_sc_11}, we obtain that 
  \begin{equation}
\label{eq_inv_sc_12}
 (e^{ikx\cdot\xi}, \p_\nu v^+-\p_\nu v_0)_{(H^{1/2},H^{-1/2})(\p B)}=0,
 \end{equation}
 for all $\xi\in\mathbb{S}^{n-1}$. 
 
 Since $k^2$ is not a Dirichlet eigenvalue of the Laplacian in $B$, $\textrm{span}\{e^{ikx\cdot\xi}|_{\p B}: \xi\in\mathbb{S}^{n-1} \}$ is dense in $H^{1/2}(\p B)$, see \cite[Lemma 3.2]{Isakov_1990} and \cite[Lemma 3.5.3]{Isakov_book_source}.  
 Hence, it follows from \eqref{eq_inv_sc_12} that 
 \[
  \p_\nu v^+=\p_\nu v_0\quad \textrm{on}\quad \p B. 
 \]
 Let us define
 \[
 \tilde v=\begin{cases} v_0 & \quad \textrm{in}\quad B,\\
 v^+& \quad \textrm{in}\quad \R^n\setminus\overline{B}. 
 \end{cases}
 \]
 Then $\tilde v\in H^1_{\textrm{loc}}(\R^n)$ satisfies the Helmholtz equation,
 \[
 (-\Delta-k^2)\tilde v=0\quad \textrm{in}\quad \R^n,
 \]
  as well as the Sommerfeld radiation condition. Hence, $\tilde v=0$ in $\R^n$, see \cite{Col_Kress_book, Odell_2006}, and therefore $v^+=0$ in $ \R^n\setminus\overline{B}$. 
 
Now let $w^+\in W(B)$. Arguing as in the derivation of \eqref{eq_inv_sc_5} for $(w^+,w^-)$, satisfying \eqref{eq_inv_sc_1},  instead of $(u^+,u^-)$, we get 
\[
(w^+,f^+)_{(H^1,\tilde H^1)(B\setminus\overline{D})}=(\p_\nu w^+,v^+)_{(H^{-1/2},H^{1/2})(\p B)}- (w^+,\p_\nu v^+)_{(H^{1/2},H^{-1/2})(\p B)}=0. 
\]
 which shows \eqref{eq_inv_sc_2_-1}. The proof is complete. 
\end{proof}

Let $k>0$ be fixed and let $B\subset \R^n$ be an open ball such that $D_1,D_2\subset\subset B$, $\supp(A^+), \supp(q^+)\subset B$, both homogeneous transmission problems in $B$,
\begin{equation}
\label{eq_inv_sc_14}
\begin{aligned}
&(\mathcal{L}_{A^+,q^+}-k^2)w^+=0\quad \textrm{in}\quad B\setminus\overline{D_j},\\
&(\mathcal{L}_{A_j^-,q^-_j}-k^2)w^-=0\quad\textrm{in}\quad D_j,\\
&w^+=a_jw^-\quad\textrm{on}\quad \p D_j,\\
&(\p_\nu+iA^+\cdot\nu)w^+=b_j(\p_\nu+iA_j^-\cdot\nu)w^-+c_jw^-\quad\textrm{on}\quad \p D_j,
\end{aligned}
\end{equation}
with
\[
w^+=0\quad\textrm{on}\quad \p B,
\]
have only the trivial solutions, $j=1,2$, and furthermore, $k^2$ is not an eigenvalue of the Dirichlet Laplacian in $B$. 
Using Lemma \ref{lem_inv_sc_density} together with  Rellich's uniqueness theorem, we conclude, similarly to \cite{Isakov_2008}, that the sets of the Cauchy data for the transmission problems \eqref{eq_inv_sc_14} in the ball $B$ agree,
\[
\mathcal{C}_{\p B}(A^+,q^+,A_1^-,q_1^-,a_1,b_1,c_1;D_1)=\mathcal{C}_{\p B}(A^+,q^+,A_2^-,q_2^-,a_2,b_2,c_2;D_2). 
\]
Applying now Theorem \ref{thm_main_sa} to the transmission problems  \eqref{eq_inv_sc_14} in the ball $B$,  we get the claim of Theorem \ref{thm_main_2}. The proof of Theorem \ref{thm_main_2} is complete.

\begin{appendix}

\section{Unique continuation for elliptic second order operators from Lipschitz boundary}

\label{ap_UC}

The discussion in this section is essentially well-known and is presented here for the convenience of the reader, see also \cite{Behrndt_Roh}.

Let $\Omega\subset\R^n$, $n\ge 2$, be a bounded domain with Lipschitz boundary.  Let 
\begin{equation}
\label{eq_UC_2}
P=\sum_{j,k=1}^n \p_{x_j} a_{jk}\p_{x_k} + \sum_{j=1}^n b_j \p_{x_j}+c\quad\textrm{in}\quad \Omega,
\end{equation}
where $(a_{jk})_{1\le j,k\le n}$ is a real symmetric matrix with $a_{jk}$ being  Lipschitz continuous on $\overline{\Omega}$, such that 
\begin{equation}
\label{eq_UC_1}
\sum_{j,k=1}^n a_{jk}(x)\xi_j\xi_k\ge \eta |\xi|^2,\quad \textrm{for all}\quad x\in\overline{\Omega},\quad \xi\in \R^n, \quad \eta>0,
\end{equation}
and 
$b_j, c\in L^\infty(\Omega, \C)$. Let $u\in H^1(\Omega)$ satisfy $Pu=0$ in $\Omega$ and let $\nu$ be the almost everywhere defined outer unit normal to $\p \Omega$. Then the conormal derivative of $u$ is defined by
\[
\mathcal{B}_{\nu}u=\sum_{j,k=1}^n \nu_j (a_{jk}\p_{x_k} u)|_{\p \Omega}\in H^{-1/2}(\p \Omega).
\]

\begin{prop}
\label{prop_UC}
Let $\Omega\subset\R^n$, $n\ge 2$, be a bounded domain with Lipschitz boundary and let $\gamma\subset\p \Omega$ be an open non-empty subset of  the boundary of $\Omega$. 
If $u\in H^1(\Omega)$ satisfies $Pu=0$ in $\Omega$, and is such that  $u|_{\gamma}=\mathcal{B}_\nu u|_{\gamma}=0$, then $u\equiv 0$ in $\Omega$. 
\end{prop}

\begin{proof}
Let $\tilde \Omega\supset \Omega$ be a bounded domain with Lipschitz boundary such that $\p \Omega\setminus \gamma\subset \p \tilde \Omega$ and $\tilde \Omega\setminus\overline{\Omega}$ is non-empty. 
Let us extend $a_{jk}$, $b_j$ and $c$ to the whole of $\tilde \Omega$ so that the extension $(\tilde a_{jk})_{1\le j,k\le n}$ is a real symmetric matrix with $\tilde a_{jk}$ being  Lipschitz continuous on $\tilde \Omega$ and satisfying \eqref{eq_UC_1} on $\tilde \Omega$, $\tilde b_j,\tilde c \in L^\infty(\tilde \Omega, \C)$. We shall denote by $\tilde P$ the corresponding elliptic differential operator on $\tilde \Omega$, defined as in \eqref{eq_UC_2}. 

Let $u\in H^1(\Omega)$ satisfy $Pu=0$ in $\Omega$ with the Cauchy data $u|_{\gamma}=\mathcal{B}_\nu u|_{\gamma}=0$.  Then we extend $u$ by zero on $\tilde \Omega\setminus\Omega$ and denote by $\tilde u$ the extension.  Since $u|_{\gamma}=0$ and $\gamma$ is Lipschitz,  we conclude that  $\tilde u\in H^1(\tilde \Omega)$.  

Let us show that $\tilde P\tilde u=0$ in $\tilde \Omega$ in the sense of distribution theory. Indeed, let $\tilde \varphi\in C^\infty_0(\tilde \Omega)$. Then we have
\begin{align*}
(\tilde P\tilde u) (\tilde \varphi)&=\int_{\tilde \Omega}\bigg(-\sum_{j,k=1}^n \tilde  a_{jk}\p_{x_k}\tilde u \p_{x_j} \tilde  \varphi + \sum_{j=1}^n (\tilde b_j \p_{x_j} \tilde u)\tilde \varphi+  \tilde c \tilde u \tilde \varphi\bigg)dx\\
&=\int_{\Omega}\bigg(-\sum_{j,k=1}^n   a_{jk}\p_{x_k} u \p_{x_j}  \varphi + \sum_{j=1}^n ( b_j \p_{x_j}  u) \varphi+   c u \varphi \bigg)dx,
\end{align*}
where $\varphi=\tilde\varphi|_{\Omega}$.  Since the boundary of $\Omega$ is Lipschitz, we can integrate by parts and get
\[
(\tilde P\tilde u) (\tilde \varphi)=\int_{\Omega}(P  u) \varphi dx- (\mathcal{B}_\nu u, \varphi)_{H^{-1/2}(\p \Omega),H^{1/2}(\p \Omega)}=0.
\]
Here we have used the fact that $\mathcal{B}_\nu u|_{\gamma}=0$ and $\varphi|_{\p\Omega\setminus\gamma}=0$. 

As $\tilde u\in H^1(\tilde \Omega)$ satisfies $\tilde P\tilde u=0$ in $\tilde \Omega$ and  $\tilde u=0$ on $\tilde\Omega\setminus\overline{\Omega}$, by the classical unique continuation result $\tilde u$ vanishes identically on $\tilde \Omega$, see \cite[Chapter 17]{Horm_book_3}.  The proof is complete.

\end{proof}

\section{Fundamental solution for a magnetic Schr\"odinger operator}

\label{appendix_fundamental_sol}

Let $\Omega\subset\R^n$, $n\ge 3$, be a bounded domain.  Consider a magnetic Schr\"odinger operator 
\[
\mathcal{L}_{A,q}(x,D)=\sum_{j=1}^n(D_j+A_j(x))^2+q(x)=-\Delta-2iA(x)\cdot\nabla-i(\nabla\cdot A(x))+A(x)^2+q(x),
\]
where $A\in W^{1,\infty}(\Omega,\R^n)$, $q\in L^\infty(\Omega,\R)$, and $D=-i\nabla$.  Let $G(x,y)$ be a fundamental solution of the operator $\mathcal{L}_{A,q}$, i.e. 
\[
\mathcal{L}_{A,q}(x,D_x)G(x,y)=\delta(x-y), \quad x,y\in \Omega. 
\]  
It is shown in   \cite[Theorem 17.1.1]{Horm_book_3}, \cite[Theorem 19, VIII]{Miranda_book}, \cite{Serov_2009}, that there exists a fundamental solution $G(x,y)$, which satisfies the following estimates,
\begin{equation}
\label{eq_fund_sol}
\begin{aligned}
|G(x,y)|&\le C|x-y|^{2-n},\\
|\nabla_xG(x,y)|&\le C|x-y|^{1-n},\quad C>0,
\end{aligned} 
\end{equation}
for all $x,y\in\Omega$.  Furthermore, $G(x,y)$ satisfies the following integral equation
\begin{equation}
\label{eq_integral}
G(x,y)=G_0(x,y)+\int_{\Omega} G_0(x,z)(2iA(z)\cdot\nabla_z G(z,y)-\tilde q(z)G(z,y))dz.
\end{equation}
Here $\tilde q=A^2-i\nabla\cdot A+ q$ and $G_0(x,y)$ is a fundamental solution of $-\Delta$, which is given by
\begin{equation}
\label{eq_G_0}
G_0(x,y)=\frac{1}{(n-2)\Upsilon_n |x-y|^{n-2}}, \quad \Upsilon_n=2\frac{\pi^{n/2}}{\Gamma(n/2)},
\end{equation}
 see \cite[p. 247]{McLean_book}. 

The following result describes more precisely the behavior of the fundamental solution $G(x,y)$ of the magnetic Schr\"odinger operator as $|x-y|\to 0$.  
 
\begin{prop}
\label{prop_assym}
 Let $A\in W^{1,\infty}(\Omega,\R^n)$ and $q\in L^\infty(\Omega,\R)$.  Then
\begin{equation}
\label{eq_prop_1}
G(x,y)-G_0(x,y)=\begin{cases} \mathcal{O}(\log\frac{1}{|x-y|}), & n=3,\\
\mathcal{O}(|x-y|^{3-n}), & n\ge 4,
\end{cases}
\end{equation} 
and  
\begin{equation}
\label{eq_prop_2}
\nabla_x(G(x,y)-G_0(x,y))=\mathcal{O}(|x-y|^{2-n}),\quad  n\ge 3,
\end{equation}
as $|x-y|\to 0$.
 
\end{prop}

In the case when $A\in C^\infty(\overline{\Omega}, \R^n)$ and $q\in C^\infty(\overline{\Omega},\R)$, Proposition \ref{prop_assym}  is well-known, see  \cite[Theorem 6.3]{McLean_book}. 
Since we did not find a reference for this result in the case when $A\in W^{1,\infty}(\Omega,\R^n)$ and $q\in L^\infty(\Omega,\R)$, we shall sketch the proof here. 

\begin{proof}
It follows from \eqref{eq_fund_sol}, \eqref{eq_integral}, and \eqref{eq_G_0} that
\begin{align*}
|G(x,y)-G_0(x,y)|&\le C\int_\Omega |x-z|^{2-n}(2|A(z)||z-y|^{1-n}+|\tilde q(z)||z-y|^{2-n})dz\\
&\le C \int_\Omega |x-z|^{2-n}|z-y|^{1-n}dz,\quad x,y\in \Omega,
\end{align*}
where we use the fact that $|z-y|\le C$ for $z,y\in \Omega$. 
Let $K>0$ be sufficiently large. Then
\[
\int_\Omega |x-z|^{2-n}|z-y|^{1-n}dz\le \int_{|z|\le K} |z|^{2-n}|z+x-y|^{1-n}dz:=I_1+I_2+I_3, 
\]
where
\begin{align*}
I_1:=&  \int_{|z|\le \frac{|x-y|}{2}} |z|^{2-n}|z+x-y|^{1-n}dz,\\
I_2:= &\int_{\frac{|x-y|}{2}\le |z|\le 2|x-y|} |z|^{2-n}|z+x-y|^{1-n}dz,\\
I_3:= & \int_{2|x-y| \le |z|\le K} |z|^{2-n}|z+x-y|^{1-n}dz. 
\end{align*}
We have
\begin{align*}
I_1\le \bigg(\frac{|x-y|}{2}\bigg)^{1-n} \int_{|z|\le \frac{|x-y|}{2}} |z|^{2-n}dz \le \mathcal{O}(|x-y|^{3-n}),
\end{align*}
\begin{align*}
I_2&\le \bigg(\frac{|x-y|}{2}\bigg)^{2-n} \int_{\frac{|x-y|}{2}\le |z|\le 2|x-y|} |z+x-y|^{1-n}dz \\
&\le 
\bigg(\frac{|x-y|}{2}\bigg)^{2-n} \int_{|\omega|\le 3|x-y|}|\omega|^{1-n}d\omega \le \mathcal{O}(|x-y|^{3-n}),
\end{align*}
and
\begin{align*}
I_3&\le 2^{n-1} \int_{2|x-y| \le |z|\le K} |z|^{3-2n}dz\le C\int_{2|x-y| \le r\le K} r^{2-n}dr\\
&= \begin{cases} \mathcal{O}(\log\frac{1}{|x-y|}),& n=3,\\
\mathcal{O}(|x-y|^{3-n}),& n\ge 4.
\end{cases}
\end{align*}
Thus, \eqref{eq_prop_1} follows. The estimate \eqref{eq_prop_2} can be proved in a similar way. The proof is complete.

\end{proof}

\section{Auxiliary estimates for some integrals} 

\label{appendix_integrals}

The purpose of this appendix is to provide a brief discussion of asymptotic bounds of some volume and surface integrals, required in the main text. We shall follow the methods of \cite[pp. 131--133]{Isakov_book} and \cite[Chapter 5]{Salo_2004}, and the following discussion is mainly for the completeness and convenience of the reader.

Let $\Omega\subset\R^n$, $n\ge 3$, be a bounded domain with Lipschitz boundary, and let $x_0\in \p \Omega$. We may assume that there is the unit outer normal $\nu(x_0)$ at the point $x_0$ to $\p \Omega$. 
We may choose coordinates $x=(x',x_n)$ isometric to the standard ones so that $x_0=0$ and for some $r_0>0$, we have 
\begin{align*}
\Omega\cap B(0,r_0)&=\{x\in B(0,r_0):x_n>\varphi(x'), x'\in U\},\\
\p \Omega\cap B(0,r_0)&=\{x\in B(0,r_0):x_n=\varphi(x'), x'\in U\},
\end{align*}
 where $\varphi:\R^n\to \R$ is a Lipschitz function and $U\subset\R^{n-1}$ is an open neighborhood of $0$.  Furthermore, 
\begin{equation}
\label{eq_D_nu}
\nu(0)=\frac{(\nabla_{x'}\varphi(0),-1)}{\sqrt{1+|\nabla_{x'}\varphi(0)|^2}},
\end{equation}
and we assume as we may that $\nu(0)=-e_n=-(0,\dots,0,1)$.

Let $0<\varepsilon<1$ and $y\in \R^n$. Then we define the cones
\begin{align*}
C_\varepsilon(y)&=\bigg\{x\in\R^n: \bigg(\frac{x-y}{|x-y|}\bigg)_n>1-\varepsilon\bigg\},\\
 C_\varepsilon^-(y)&=\bigg\{x\in\R^n: \bigg(\frac{x-y}{|x-y|}\bigg)_n<-(1-\varepsilon)\bigg\}. 
\end{align*}

Since $\p \Omega$ is Lipschitz, there is $\varepsilon_0>0$ small such that 
\[
C_{\varepsilon_0}(0)\cap B(0,r_0)\subset \Omega\textrm{ and } C_{\varepsilon_0}^-(0)\cap B(0,r_0)\subset \R^n\setminus\overline{\Omega}. 
\]

We define 
\[
x_\delta=x_0+\delta\nu(x_0)=-\delta e_n,
\]
for $\delta>0$ small parameter such that there is a constant $c=c(\varepsilon_0)>0$ small fixed so that 
\begin{equation}
\label{eq_D_1}
B(x_\delta, c\delta)\subset \R^n\setminus\overline{\Omega}.
\end{equation}
Furthermore, it is easy to see that there is $\varepsilon_1=\varepsilon_1(\varepsilon_0)>0$ small and $C=C(\varepsilon_0)>0$ large enough fixed so that 
\[
C_{\varepsilon_1}(x_\delta)\cap\{x\in\R^n:|x-x_\delta|>C\delta\}\subset C_{\varepsilon_0}(0). 
\]
We define the set
\begin{equation}
\label{eq_D_2}
E_\delta=C_{\varepsilon_1}(x_\delta)\cap\{x\in\R^n:C\delta<|x-x_\delta|<r_0\}\subset \Omega. 
\end{equation}

\begin{prop} Let $k>0$. Then we have, as $\delta\to 0$,
\begin{equation}
\label{eq_D_3}
\int_\Omega \frac{dx}{|x-x_\delta|^k} \le 
\begin{cases} C, & k<n,\\
C\log\frac{1}{\delta}, & k=n,\\
C\delta^{n-k}, & k>n,
\end{cases}\quad C>0,
\end{equation}
and 
\begin{equation}
\label{eq_D_4}
\int_\Omega \frac{dx}{|x-x_\delta|^k} \ge 
\begin{cases} \frac{1}{C}, & k<n,\\
\frac{1}{C}\log\frac{1}{\delta}, & k=n,\\
\frac{1}{C}\delta^{n-k}, & k>n,
\end{cases}\quad C>0.
\end{equation}

\end{prop}

\begin{proof}
It follows from \eqref{eq_D_1} that for any $x\in \Omega$, $|x-x_\delta|\ge c\delta$. Let $R>0$ be independent of $\delta$ such that $\overline{\Omega}\subset B(x_\delta, R)$. Then we get
\[
\int_\Omega \frac{dx}{|x-x_\delta|^k} \le \int_{c\delta\le |x-x_\delta|\le R}\frac{dx}{|x-x_\delta|^k}= C\int_{c\delta}^R r^{n-k-1}dr,
\]
which shows \eqref{eq_D_3}.

Using \eqref{eq_D_2} and polar coordinates $x-x_\delta=r\omega$, $r>0$, $\omega\in\mathbb{S}^{n-1}$, $dx=r^{n-1}drd\omega$,  we obtain that 
\[
\int_\Omega \frac{dx}{|x-x_\delta|^k} \ge \int_{E_\delta} \frac{dx}{|x-x_\delta|^k} =\int_{C\delta}^{r_0} r^{n-k-1}dr\int_{\omega_n>1-\varepsilon_1}d\omega=C\int_{C\delta}^{r_0} r^{n-k-1}dr,
\]
which proves \eqref{eq_D_4}. 

\end{proof}

We shall also need the following estimate, when $n=3$, 
 \begin{equation}
\label{eq_B_-1_2}
\int_\Omega \frac{1}{|x-x_\delta|^{2}}\log\frac{1}{|x-x_\delta|}dx \le C,
\end{equation}
as $\delta\to 0$.

As a consequence of the estimates \eqref{eq_fund_sol} and \eqref{eq_D_3}, we obtain the following result. 

\begin{cor} We have, as $\delta\to 0$,
\begin{equation}
\label{eq_B_0}
\|\nabla_xG(x,x_\delta)\|_{L^2(\Omega)}\le C\delta^{1-n/2},
\end{equation}
\begin{equation}
\label{eq_B_0_1}
\|G(\cdot,x_\delta)\|_{L^2(\Omega)}\le 
\begin{cases}
C,& n=3,\\
C\bigg(\log\frac{1}{\delta}\bigg)^{1/2}, & n=4,\\
 C\delta^{2-n/2}, & n\ge 5,
 \end{cases}.
\end{equation}
Thus, 
\begin{equation}
\label{eq_B_0_2}
\|G(\cdot,x_\delta)\|_{H^1(\Omega)}\le C\delta^{1-n/2}.
\end{equation}

\end{cor}

\begin{prop}
Let $k>0$. Then we have, as $\delta\to 0$,
\begin{equation}
\label{eq_D_5}
\int_{\p \Omega} \frac{dS}{|x-x_\delta|^k} \le 
\begin{cases} C, & k<n-1,\\
C\log\frac{1}{\delta}, & k=n-1,\\
C\delta^{n-1-k}, & k>n-1,
\end{cases}\quad C>0, 
\end{equation}
and 
\begin{equation}
\label{eq_D_7}
\int_{\p \Omega} \frac{dS}{|x-x_\delta|^k} \ge 
\begin{cases} \frac{1}{C}, & k<n-1,\\
\frac{1}{C}\log\frac{1}{\delta}, & k=n-1,\\
\frac{1}{C}\delta^{n-1-k}, & k>n-1,
\end{cases}\quad C>0. 
\end{equation}

\end{prop}

\begin{proof}
First recall that for $x\in\p \Omega\cap B(0,r_0)$, we get $x_n=\varphi(x')$, $x'\in U$, and $\varphi(0)=0$.  As 
$\nu(0)=-e_n$,  it follows from \eqref{eq_D_nu} that $\nabla\varphi(0)=0$, and therefore,
\[
\lim_{x'\to 0}\frac{|\varphi(x')|}{|x'|}=0.
\] 
Thus, without loss of generality, we may assume that $r_0>0$ is so small that 
\begin{equation}
\label{eq_D_6}
|\varphi(x')|\le \frac{1}{2}|x'|,
\end{equation}
 for any $x'\in U$.

We write 
\[
\int_{\p \Omega} \frac{dS}{|x-x_\delta|^k} =\int_{\p \Omega\cap B(0,r_0)} \frac{dS}{|x-x_\delta|^k} + \int_{\p \Omega\setminus\overline{\p \Omega\cap B(0,r_0)}} \frac{dS}{|x-x_\delta|^k}, 
\]
where the second integral is bounded as $\delta\to 0$. For the first integral, we have
\[
\int_{\p \Omega\cap B(0,r_0)} \frac{dS}{|x-x_\delta|^k}=\int_{U} \frac{\sqrt{1+|\nabla \varphi(x')|^2}dx'}{(|x'|^2+(\varphi(x')+\delta)^2)^{k/2}} \le C\int_U \frac{dx'}{(|x'|^2+\delta^2)^{k/2}}.
\]
Here we have used the fact that 
\[
|x'|^2+(\varphi(x')+\delta)^2\ge \frac{1}{2}(|x'|^2+\delta^2),
\]
which follows from \eqref{eq_D_6}.  
Let $B(0,R)\subset\R^{n-1}$ be a ball of radius $R$, centered at $0$, such that $U\subset B(0,R)$. Using the polar coordinates $r=|x'|$, we get
\[
I_1:=\int_U \frac{dx'}{(|x'|^2+\delta^2)^{k/2}}\le  C \int_0^R \frac{r^{n-2}dr}{(r^2+\delta^2)^{k/2}}.
\]
Setting $\sigma=r^{n-1}$ and using the fact that $\sigma^{2/(n-1)}+\delta^2\ge C^{-1}(\sigma+\delta^{n-1})^{2/(n-1)}$, we obtain that
\[
I_1\le C\int_0^{R^{n-1}}\frac{d\sigma}{(\sigma^{2/(n-1)}+\delta^2)^{k/2}}\le C\int_0^{R^{n-1}}\frac{d\sigma}{(\sigma+\delta^{n-1})^{k/(n-1)}}.
\]
Thus, \eqref{eq_D_5} follows.

In order to show \eqref{eq_D_7}, we get
\begin {align*}
\int_{\p \Omega} \frac{dS}{|x-x_\delta|^k} &\ge \int_{\p \Omega\cap B(0,r_0)} \frac{dS}{|x-x_\delta|^k}=  \int_{U} \frac{\sqrt{1+|\nabla \varphi(x')|^2}dx'}{(|x'|^2+(\varphi(x')+\delta)^2)^{k/2}}\\
&\ge \frac{1}{C}\int_U \frac{dx'}{(|x'|^2+\delta^2)^{k/2}}.
\end{align*}
 Here we have used the fact that 
 \[
 |x'|^2+(\varphi(x')+\delta)^2\le 2(|x'|^2+\delta^2),
 \]
which is a consequence of \eqref{eq_D_6}. As above, using 
the polar coordinates $r=|x'|$, we get
\[
I_2:=\int_U \frac{dx'}{(|x'|^2+\delta^2)^{k/2}}\ge C \int_0^{R_1} \frac{r^{n-2}dr}{(r^2+\delta^2)^{k/2}},\quad R_1>0.
\]
Setting $\sigma=r^{n-1}$ and using the fact that $\sigma^{2/(n-1)}+\delta^2\le C(\sigma+\delta^{n-1})^{2/(n-1)}$, we conclude that 
\[
I_2\ge \frac{1}{C}\int_0^{R_1^{n-1}}\frac{d\sigma}{(\sigma^{2/(n-1)}+\delta^2)^{k/2}}\ge \frac{1}{C}\int_0^{R_1^{n-1}}\frac{d\sigma}{(\sigma+\delta^{n-1})^{k/(n-1)}},
\]
which shows \eqref{eq_D_7}. The proof is complete.

\end{proof}

As a consequence of  \eqref{eq_fund_sol} and \eqref{eq_D_5},  we get the following result. 

\begin{cor}
We have
\begin{equation}
\label{eq_B_0_3}
\|G(\cdot,y)\|_{L^2(\p \Omega)}\le 
\begin{cases}
C\bigg(\log\frac{1}{\delta}\bigg)^{1/2} & n=3,\\
 C\delta^{(3-n)/2}, & n\ge 4,
 \end{cases} \quad\textrm{as}\quad \delta\to 0.
\end{equation}
\end{cor}

We shall also need the following estimate,  when $n=3$, 
\begin{equation}
\label{eq_B_12}
\int_{\p \Omega} \frac{1}{|x-x_\delta|}\log \frac{1}{|x-x_\delta|} dS \le C, \quad \textrm{as}\quad \delta\to 0. 
\end{equation}

\end{appendix}

\section*{Acknowledgements}  
I am very grateful to Victor Isakov for the most useful comments and suggestions on this work. I would also like to thank Mikhail Agranovich, Matti Lassas, Mikko Salo, Valeriy Serov, Semen Podkorytov, and Gunther Uhlmann for helpful discussions.  This research is supported by the Academy of Finland (project 125599).


\begin{thebibliography} {1}

\bibitem{Ammari_Uhlmann_2004}
Ammari, H.,  Uhlmann, G., \emph{Reconstruction of the potential from partial Cauchy data for the Schr\"odinger equation}, Indiana Univ. Math. J. \textbf{53} (2004), no. 1, 169--183. 


\bibitem{Behrndt_Roh}
Behrndt, J.,  Rohleder, J., \emph{An inverse problem of Calderon type with partial data}, to appear in Comm. Partial Differential Equations.



\bibitem{Brown_Salo_2006}
Brown, R.,  Salo, M., \emph{Identifiability at the boundary for first-order terms},  Appl. Anal.  \textbf{85} (2006),  no. 6--7, 735--749. 

\bibitem{BukhUhl_2002} 
Bukhgeim, A.,  Uhlmann, G.,  \emph{Recovering a potential from partial Cauchy data},  Comm. Partial Differential Equations  \textbf{27}  (2002),  no. 3--4, 653--668. 



\bibitem{Choulli_book}
Choulli, M., \emph{Une introduction aux probl\`emes inverses elliptiques et
  paraboliques}, volume~65 of {\em Math\'ematiques \& Applications (Berlin)
  [Mathematics \& Applications]},  Springer-Verlag, Berlin, 2009.


\bibitem{Col_Kress_book}
Colton, D.,  Kress, R., \emph{Inverse acoustic and electromagnetic scattering theory}, Second edition. Applied Mathematical Sciences, 93. Springer-Verlag, Berlin, 1998.


\bibitem{DKSU_2007}
Dos Santos Ferreira, D.,  Kenig, C., Sj\"ostrand, J., and Uhlmann, G.,  \emph{Determining a magnetic Schr\"odinger operator from partial Cauchy data},  Comm. Math. Phys.  \textbf{271}  (2007),  no. 2, 467--488. 



\bibitem{Grubbbook2009}
Grubb, G., \emph{Distributions and operators}, volume 252 of {\em Graduate Texts
  in Mathematics}.
\newblock Springer, New York, 2009.


\bibitem{Horm_1973}
H\"ormander, L., \emph{Lower bounds at infinity for solutions of differential equations with constant coefficients},   Israel J. Math.  \textbf{16}  (1973), 103--116.



\bibitem{Horm_book_3}
H\"ormander, L.,  \emph{The analysis of linear partial differential operators. III. Pseudo-differential operators}, Classics in Mathematics. Springer, Berlin, 2007.

\bibitem{Isakov_1988}
Isakov, V.,  \emph{On uniqueness of recovery of a discontinuous conductivity coefficient}, Comm. Pure Appl. Math. \textbf{41} (1988), no. 7, 865--877. 

\bibitem{Isakov_book_source}
Isakov, V.,  \emph{Inverse source problems},  Mathematical Surveys and Monographs, 34. American Mathematical Society, Providence, RI, 1990.

\bibitem{Isakov_1990}
Isakov, V., \emph{On uniqueness in the inverse transmission scattering problem},  Comm. Partial Differential Equations \textbf{15} (1990), no. 11, 1565--1587.

\bibitem{Isakov_book}
Isakov, V.,  \emph{Inverse problems for partial differential equations}, Second edition. Applied Mathematical Sciences, 127. Springer, New York, 2006.


\bibitem{Isakov_2008}
Isakov, V., \emph{On uniqueness in the general inverse transmission problem},  Comm. Math. Phys. \textbf{280} (2008), no. 3, 843--858. 

\bibitem{Isakov_2009}
Isakov, V.,  \emph{Inverse obstacle problems}, Inverse Problems \textbf{25} (2009), no. 12, 123002, 18 pp.








\bibitem{Ken_Sjo_Uhl_2007}
Kenig, C.,  Sj\"ostrand, J., and Uhlmann, G.,  \emph{The Calder\'on problem with partial data},  Ann. of Math. (2)  \textbf{165}  (2007),  no. 2, 567--591. 

\bibitem{Kirsch_Kress_1993}
Kirsch, A.,  Kress, R., \emph{Uniqueness in inverse obstacle scattering}, Inverse Problems \textbf{9} (1993), no. 2, 285--299.

\bibitem{Kirsch_Paivar_1998}
Kirsch, A., P\"aiv\"arinta, L., \emph{On recovering obstacles inside inhomogeneities}, 
Math. Methods Appl. Sci. \textbf{21} (1998), no. 7, 619--651. 


\bibitem{Knu_Salo_2007}
Knudsen, K.,  Salo, M., \emph{Determining nonsmooth first order terms from partial boundary measurements}, Inverse Probl. Imaging \textbf{1} (2007), no. 2, 349--369. 


\bibitem{KrupLassasUhlmann} Krupchyk, K., Lassas, M., and Uhlmann, G., \emph{Inverse boundary value problems for the perturbed polyharmonic operator}, to appear in  Trans. Amer. Math. Soc. 

\bibitem{KrupLassasUhlmann_slab} Krupchyk, K., Lassas, M., and Uhlmann, G., \emph{Inverse problems with partial data for a magnetic Schr�dinger operator in an infinite slab and on a bounded domain}, to appear in Comm. Math. Phys.

\bibitem{Lax_Phillips_book}
Lax, P.,  Phillips, R., \emph{Scattering theory},  Second edition. With appendices by Cathleen S. Morawetz and Georg Schmidt. Pure and Applied Mathematics, 26. Academic Press, Inc., Boston, MA, 1989.

\bibitem{Lax_Phillips}
Lax, P.,  Phillips, R., \emph{Scattering theory for the acoustic equation in an even number of space dimensions},
Indiana Univ. Math. J. 22 (1972/73), 101--134.

\bibitem{Leis_1967}
Leis, R.,  \emph{Zur Monotonie der Eigenwerte selbstadjungierter elliptischer Differentialgleichungen}, (German) Math. Z. \textbf{96} (1967),  26--32.

\bibitem{Leis_book_1986}
Leis, R., \emph{Initial-boundary value problems in mathematical physics}, B. G. Teubner, Stuttgart; John Wiley \& Sons, Ltd., Chichester, 1986.




\bibitem{McLean_book}
McLean, W., \emph{Strongly elliptic systems and boundary integral equations}, Cambridge University Press, Cambridge, 2000.

\bibitem{Miranda_book}
Miranda, C., \emph{Partial differential equations of elliptic type}, Second revised edition. Translated from the Italian by Zane C. Motteler. Ergebnisse der Mathematik und ihrer Grenzgebiete, Band 2. Springer-Verlag, New York-Berlin 1970.

\bibitem{NakSunUlm_1995}
Nakamura, G.,  Sun, Z., and  Uhlmann, G., \emph{Global identifiability for an inverse problem for the Schr\"odinger equation in a magnetic field},
Math. Ann. \textbf{303} (1995), no. 3, 377--388.



\bibitem{Odell_2006}
O'Dell, S., \emph{Inverse scattering for the Laplace-Beltrami operator with complex electromagnetic potentials and embedded obstacles},  Inverse Problems  \textbf{22}  (2006),  no. 5, 1579--1603.

\bibitem{Odell_2008}
O'Dell, S., \emph{Inverse scattering for Schr�dinger-type operators with interface conditions across smooth surfaces},  SIAM J. Math. Anal. \textbf{39} (2008), no. 5, 1595--1626. 



\bibitem{Reed_Simon_book_1} 
Reed, M., Simon, B.,  \emph{Methods of modern mathematical physics. I. Functional analysis}. Second edition. Academic Press, Inc. [Harcourt Brace Jovanovich, Publishers], New York, 1980. 

\bibitem{Salo_2004}
Salo, M.,  \emph{Inverse problems for nonsmooth first order perturbations of the Laplacian}, Ann. Acad. Sci. Fenn. Math. Diss. \textbf{139} (2004).

\bibitem{SaloTzou_2009}
Salo, M., Tzou, L., \emph{Carleman estimates and inverse problems for Dirac operators},  Math. Ann.  \textbf{344}  (2009),  no. 1, 161--184.

\bibitem{Serov_2009}
Serov, V.,  \emph{The fundamental solution and Fourier series in eigenfunctions of the magnetic Schr\"odinger operator},  J. Phys. A  \textbf{42}  (2009),  no. 22, 225205, 12 pp. 

\bibitem{Stefanov_1990}
Stefanov, P.,  \emph{Stability of the inverse problem in potential scattering at fixed energy},  Ann. Inst. Fourier (Grenoble) \textbf{40} (1990), no. 4, 867--884.

\bibitem{Stein_book}
Stein, E., \emph{Singular integrals and differentiability properties of functions},  Princeton Mathematical Series, No. 30 Princeton University Press, Princeton, N.J. 1970 




\bibitem{Sun_1993}
Sun, Z.,  \emph{An inverse boundary value problem for Schr\"odinger operators with vector potentials}, Trans. Amer. Math. Soc.  \textbf{338}  (1993),  no. 2, 953--969.

\bibitem{Valdivia_2004}
Valdivia, N.,  \emph{Uniqueness in inverse obstacle scattering with conductive boundary conditions},  Appl. Anal. \textbf{83} (2004), no. 8, 825--851. 

\end{thebibliography}
\end{document}